\documentclass[reqno]{amsart}
\usepackage{amsmath,amsthm,amssymb,amsfonts,color}
\usepackage{hyperref,comment}
\usepackage{todonotes}
\usepackage{esint}
\usepackage[latin1]{inputenc}

\usepackage[shortlabels]{enumitem}

\newcommand\bR{\mathbb{R}}

\newcommand\bM{\mathbb{M}}

\newcommand\cR{\mathcal{R}}

\newcommand\tQ{\widetilde Q}

\newcommand\cM{\mathcal{M}}

\newcommand\cI{\mathcal{I}}

\newcommand\cW{\mathcal{W}}

\newcommand\cK{\mathcal{K}}

\numberwithin{equation}{section}

\newcommand\cbrk{\text{$]$\kern-.15em$]$}}
\newcommand\opar{
\text{\,\raise.2ex\hbox{${\scriptstyle |}$}\kern-.34em$($}}

 \theoremstyle{definition}

\newtheorem{theorem}{Theorem}[section]
\newtheorem{lemma}[theorem]{Lemma}
\newtheorem{proposition}[theorem]{Proposition}
\newtheorem{corollary}[theorem]{Corollary}

\newtheorem{definition}{Definition}[section]

\theoremstyle{remark}
\newtheorem{remark}[theorem]{Remark}

\newtheorem{assumption}[theorem]{Assumption}

\newcommand{\nlimsup}{\operatornamewithlimits{\overline{lim}}}

\makeatletter
\def\dashint{\operatorname%
{\,\,\text{\bf--}\kern-.98em\DOTSI\intop\ilimits@\!\!}}
\makeatother

\begin{document}

\title[$L_p$ estimates for kinetic Kolmogorov-Fokker-Planck equations]{Global $L_p$ estimates for kinetic Kolmogorov-Fokker-Planck equations in nondivergence form}

\author[H. Dong]{Hongjie Dong}
\address[H. Dong]{Division of Applied Mathematics, Brown University, 182 George Street, Providence, RI 02912, USA}
\email{Hongjie\_Dong@brown.edu }
\thanks{H. Dong was partially supported by the Simons Foundation, grant no. 709545, a Simons fellowship, grant no. 007638, and the NSF under agreement DMS-2055244.}

\author[T. Yastrzhembskiy]{Timur Yastrzhembskiy}
\address[T. Yastrzhembskiy]{Division of Applied Mathematics, Brown University, 182 George Street, Providence, RI 02912, USA}
\email{Timur\_Yastrzhembskiy@brown.edu}

\subjclass[2010]{35K70, 35H10, 35B45, 34A12}
\keywords{Kinetic Kolmogorov-Fokker-Planck equations, mixed-norm Sobolev estimates, Muckenhoupt weights, vanishing mean oscillation coefficients.}

\maketitle

\begin{abstract}
We  study the degenerate Kolmogorov equations
(also known as kinetic Fokker-Planck equations) in  nondivergence form.
The leading coefficients $a^{ij}$ are merely measurable in $t$ and  satisfy
the  vanishing mean oscillation (VMO) condition in $x, v$  with respect to some quasi-metric.
We also assume boundedness and uniform nondegeneracy of $a^{ij}$ with respect to $v$. We prove global a priori estimates
in  weighted mixed-norm Lebesgue spaces
and solvability results.
We also show an application of the main result to the Landau equation.
Our proof does not rely on any kernel estimates.
\end{abstract}

\section{Introduction}

                        \label{section 1}

Let $d\ge 1$, $\bR^d$ be a Euclidean space of points $(x_1, \ldots, x_d)$,
and for $T \in (-\infty, \infty]$
 we denote
$\bR^d_T = (-\infty, T) \times \bR^{d-1}$.
 By
  $z$  we denote the triple $(t, x, v)$,
 where $t \in \bR$, and  $x, v \in \bR^d$.

In this paper, we study kinetic Kolmogorov-Fokker-Planck (KFP) operator
in  nondivergence form given by
$$
P u = \partial_t u - v \cdot D_x u - a^{i j} (z) D_{v_i v_j} u.
$$
Here the  coefficients $a (z) = (a^{i j} (z), i, j = 1, \ldots, d)$ are bounded measurable  and uniformly nondegenerate. When the coefficients $a^{i j}$
are independent of $x$ and $v$, we denote $P$  by $P_0$.
This operator appears in the theory of diffusion processes \cite{Z_19},  mathematical finance \cite{P_05},
and kinetic equations of plasma.
In particular, the linearized Landau equation near Maxwellian can  be rewritten as a  Cauchy problem
\begin{equation}
			\label{1.0}
P f + b \cdot D_{v} f + c f    = h, \quad f (0, \cdot, \cdot) = f_0 (\cdot, \cdot),
\end{equation}
 where
$b$ is a  vector-valued function (see, for instance, \cite{KGH_20}), and $c$ is a bounded function.

The goal of this article is to prove a priori estimates
and the unique solvability
in the weighted mixed-norm spaces $S_{p, r_1, \ldots, r_d, q} (\bR^{1+2d}_T, w)$ (see Section \ref{section 2}),
which  generalize  the ultraparabolic Sobolev space $S_p$ (see, for instance, \cite{BCM_96}).
We do this under a relaxed $VMO$ type assumption, which appears to be new.
 In particular, our coefficients $a^{i j}$
are merely measurable in time and $VMO$  in the $x, v$ variables  with  respect to some quasi-distance (see Assumption \ref{assumption 2.2}).
This assumption is analogous to the $VMO_x$ condition from the theory of nondegenerate parabolic PDEs with discontinuous coefficients (see \cite{Kr_08}, \cite{DK_18}).

The kinetic Kolmogorov-Fokker-Planck equations have been studied extensively
(\cite{AT_19}, \cite{AP_20} -  \cite{DGY_21},  \cite{GIMV_19}, \cite{HMP_19} - \cite{IS_21},  \cite{NZ_20} - \cite{PR_98}, \cite{ZZ_21}), including nonlocal equations (\cite{CZ_18}, \cite{HMP_19}, \cite{IS_21}, \cite{NZ_20}).
Here we discuss the results of the Sobolev space theory.
The interior $S_p$-estimate under a $VMO$  assumption with respect to all  variables
 was established in \cite{BCM_96} for nondivergence equations, and in \cite{MP_98} for divergence form equations.
The interior estimate in the ultraparabolic Morrey spaces  was proved in \cite{PR_98}.
 The first global $S_p$ estimate  was discovered in \cite{BCLP_10} for the constant coefficients $a^{i j}$ case.
For $a^{i j}$ independent of $(x, v)$, the a priori estimates can be found in \cite{CZ_19}.
 The first global $S_p$ estimate for the variable coefficients $a^{i j}$ case  was established  by the authors of \cite{BCLP_13}.
Assuming the uniform continuity of the leading coefficients, they proved the a priori estimate of $D^2_v u$ on a sufficiently small strip $[-T, T] \times \bR^{2d}$.
 To the best of our knowledge, no solvability results in ultraparabolic Sobolev spaces  were presented in these works except for the case of constant coefficients, and the operators considered
are more general  than $P$. On the other hand, when the coefficients are regular enough, say H\"older continuous (in the appropriate sense),  the unique solvability is established in \cite{DFP_05}  (see also \cite{AMP_20})   by studying the fundamental solution to a KFP operator.
It seems that the first results in the weighted mixed-norm space belong to \cite{NZ_20}.
The authors proved the a priori estimates and unique solvability  in the   $L_p ((0, T), L_q (\bR^{2d}))$ spaces with the  power weight in time (see Theorem 8.1 of \cite{NZ_20}) assuming that  the functions $a^{i j}$ are uniformly continuous with respect to some quasi-metric.
They also studied quasilinear kinetic KFP equations.
In the case of the zero initial condition, the main result of this article covers Theorem 8.1 in \cite{NZ_20}  because we work with spaces that have a Muckenhoupt weight in the $(t, v)$ variables.

The  current paper generalizes the aforementioned results in  several directions.
First, our assumption on the coefficients $a^{i j}$ is weaker than the ones presented in the literature.
Second,  we prove the a priori estimates in the weighted mixed-norm space where each direction $t, v_1, \ldots, v_d$ has a different Muckenhoupt weight.
Third,  we also discover the a priori estimate of $D_v (-\Delta_x)^{1/6} u$ for $u \in S_{p, r_1, \ldots, r_d, q} (\bR^{1+2d}_T, w)$,
 which appears to be new.
Finally, we show that the constant on the right-hand side in the $S_p$ estimate \eqref{2.1.2} (see also \eqref{2.2.1}) grows polynomially as the lower eigenvalue bound of the coefficient matrix $a^{i j}$ decreases. This fact is crucial in the  application to the linearized Landau equation near Maxwellian \eqref{1.0}. In particular, the matrix of the leading coefficients $a^{i j}$ has a lower eigenvalue bound of order $n^{-3}$ on the annulus $(0, T) \times \Omega \times \{|v| \sim n\}, n \in \{1, 2, \ldots\}$ (see \cite[Lemma 2.4]{KGH_20}), where $\Omega \subset \bR^3$ is a domain.  Then, one can use the $S_p$ estimate (see \eqref{2.2.1}) to obtain an upper bound of the $L_p$-norm of $D^2_v f$.  We elaborate on this  in Remark \ref{remark 2.3}. See also \cite{DGO_20}, \cite{KGH_20}, and \cite{DGY_21}.

We comment on the ideas of the proof.
To prove the main result, one needs to, first, work with the model equation
\begin{equation}
            \label{2.3}
P_0 u + \lambda u = f,
\end{equation}
where the coefficients $a^{i j}$ are independent of $x$ and $v$.
 Recall that in the unmixed and unweighted case, the a priori estimates for Eq. \eqref{2.3}  were already proved in \cite{BCLP_10} and \cite{CZ_19} by using the estimates of fundamental solutions.
 We give a new proof of these a priori estimates and establish a unique solvability result for Eq. \eqref{2.3} in weighted mixed-norm spaces.
As a corollary, we obtain  pointwise estimates of
$$
(D^2_v u)^{\#}_T,\quad
((-\Delta_x)^{1/3} u)^{\#}_T,\quad
(D_v (-\Delta_x)^{1/6} u)^{\#}_T,
$$
in the  case of variable coefficients $a^{i j}$,
which are similar to those in Lemma 6.2.2 of \cite{Kr_08} (see Lemma \ref{lemma 6.1}).
Here the superscript $\#$ stands for the sharp function with respect to some quasi-distance.
By designing a family of maximal functions and using a variant of the Hardy-Littlewood and Fefferman-Stein theorems (see Corollary \ref{corollary 4.9}),  we prove the a priori estimate of Theorem \ref{theorem 2.1} $(i)$,  which also implies the uniqueness part of Theorem \ref{theorem 2.1}. To prove the existence of solutions, we use the density argument as well as the argument of Section 8 of \cite{DK_18}.

The novelty  of this work lies in the fact that
 we do not use an analytic expression of the fundamental solution of the operator $P_0 + \lambda$.
 We use a kernel-free approach, which can be found in papers by N. V. Krylov,  the first named author, D. Kim, and others (see, for example, \cite{Kr_08}, \cite{DK_18}). Such a method is useful in developing the solvability theory in Sobolev spaces for second-order operators whose fundamental solutions do not have an  explicit form.  The reader can find examples of such equations in \cite{BB_20} and \cite{Kr_09}.
Interestingly, the authors of \cite{CZ_18}, \cite{CZ_19} also proved a variant of the Fefferman-Stein theorem (see Theorem 2.11 of \cite{CZ_18}).
In the same papers, they use the Fefferman-Stein inequality to prove the a priori estimates of $D^2_v u$ and $(-\Delta_x)^{1/3} u$ in the $L_p$ spaces.
  However, the main difference is that, instead of estimating  $(D^2_v u)^{\#}_T$ and $((-\Delta_x)^{1/3} u)^{\#}_T$
in terms of the maximal functions of $D^2_v u, (-\Delta_x)^{1/3}u$ and the right-hand side of  Eq. \eqref{2.3} like we do,
these authors used a variant of the Stampacchia interpolation theorem (see  \cite[Theorem 2.12]{CZ_18} and \cite[Theorem 2.4]{CZ_19}),
 which they derived from the Fefferman-Stein inequality.
One more difference with our work is that the argument of \cite{CZ_19} involves kernel estimates. The method in the present paper allows us to further treat the $VMO$ coefficients by incorporating the perturbation in the mean oscillation estimates.
 It would be interesting to find out if the methods of \cite{BCM_96}, \cite{BCLP_13}, \cite{CZ_19}, \cite{NZ_20}, \cite{MP_98} could be used to prove Theorem \ref{theorem 2.1}.

The $L_p$ theory of KFP equations developed in this article can be used in mathematical theory of plasma and filtering of signals.
For example, in  \cite{DGY_21},  the present authors and Yan Guo  showed the well-posedness and higher regularity of the linear
Landau equation with the specular reflection boundary condition by applying the results of the current paper.
 The crucial difficulty  in this problem is that the presence of the boundary condition forces one to work with  Kolmogorov type equations with `rough in time' coefficients, which is why Theorem \ref{theorem 2.1} and Corollary \ref{corollary 2.2} are useful for such equations.
See the details in Section 2 of \cite{DGY_21}.
Motivated by such an example, we plan to study kinetic KFP equations with rough coefficients in divergence form in suitable Sobolev spaces and equations in nondivergence form in Morrey-Campanato spaces.
In addition, the a priori estimates for KFP equation might be useful in developing the solvability theory for  its stochastic counterparts  which arise in  filtering of diffusion processes and interacting particle systems. For discussion and related studies, see \cite{ZZ_21}, \cite{PP_19}, and \cite{HZZZ_21}.
We plan to further investigate the stochastic KFP equation in subsequent papers.

The paper is organized as follows.
In the next section, we introduce the notation and assumptions and state the main result of the article. In Section \ref{section 3}, we prove some auxiliary results, including  variants of the Hardy-Littlewood and Fefferman-Stein theorems for the maximal and sharp functions
defined with respect to an ultraparabolic quasi-distance.
In Section \ref{section 10}, we prove Theorem \ref{theorem 2.1} for Eq. \eqref{2.3} with $P = P_0$ and $p = 2$. We extend this result to $p \in (1, \infty)$ in Section \ref{section 4}. In Section \ref{sec7}, we prove
Theorem \ref{theorem 2.1}
with $P = P_0$. In the last section, we prove Theorem \ref{theorem 2.1}.

\section{Notation and statement of the main results}
                                                    \label{section 2}
\subsection{Notation and assumptions}
For $r>0$ and $x_0\in \bR^d$, denote
$$
	B_r (x_0) = \{\xi \in \bR^d: |\xi-x_0| < r\},\quad B_r  =  B_r (0).
$$
For $r, R > 0$ and $z_0 \in \bR^{1+2d}$, we set
 $$
	Q_{r, R} (z_0) =  \{z: -r^2<t - t_0<0, |v-v_0| < r, |x - x_0 + (t - t_0) v_0|^{1/3} < R\},
 $$
$$
	\tQ_{r, R} (z_0) =  \{z: |t - t_0|<r^2, |v-v_0| < r, |x - x_0 + (t - t_0) v_0|^{1/3} < R\},
 $$
$$
    	Q_r (z_0) = Q_{r, r} (z_0), \quad
\quad \tQ_r (z_0) = \tQ_{r, r} (z_0),
$$
 $$
Q_{r, R} = Q_{r, R} (0), \quad  \tQ_{r, R}  =  \tQ_{r, R} (0), \quad Q_r = Q_r (0), \quad  \tQ_r =  \tQ_r (0).
 $$
For an open set $G$, by $C_b (G)$, we denote the space of all bounded continuous functions on $G$.
By $C^{\infty}_0 (\bR^d)$, we denote the space of all smooth compactly supported functions on $\bR^d$.

We say that a function $w$ is a weight on $\bR^d$
if $w$ is nonnegative, locally integrable,
and $w > 0$ almost everywhere.
For  $p > 1$, we write $w \in A_p (\bR^d)$
if $w$ is a weight on $\bR^d$
such that
$$
    [w]_{A_p (\bR^d)} : = \sup_{ x_0 \in \bR^d, r > 0}
    \Big(\fint_{    B_r (x_0) } w (x) \, dx\Big)\Big(\fint_{ B_r (x_0) } w^{-1/(p-1)}(x) \, dx\Big)^{p-1} < \infty.
$$

\begin{remark}
                \label{rem2.1}
For $\alpha \in (-d, d(p-1))$,  $w (x) = |x|^{\alpha}$ is an $A_p$ weight in $\bR^d$ (see the details in Example 7.1.7 in \cite{G_14}).
\end{remark}

Furthermore, for any numbers
 $p, r_1, \ldots, r_d, q > 1$, $K \ge 1$, and a function $w$  on $\bR^{1+d}$,
 we write $[w]_{q, r_1, \ldots, r_d} \le K$
if there exist weights $w_i, i = 0, \ldots, d,$ on $\bR$ such that
\begin{equation}
			\label{eq2.6}
    w (t, v) = w_0 (t)  \, \prod_{i = 1}^d w_i (v_i)
\end{equation}
and
\begin{equation}
			\label{eq2.7}
    [w_0]_{A_q(\bR)}, [w_i]_{ A_{r_i} (\bR) } \le K,\quad i = 1, \ldots, d.
\end{equation}
By $L_{p,  r_1, \ldots, r_d, q}(G, w)$, $L_{p; r_1, \ldots, r_d } (G, |x|^{\alpha} \prod_{i = 1}^d w_i (v_i) )$, $\alpha \in (-1, p-1),$
we denote the spaces of all Lebesgue measurable functions on $\bR^{1+2d}$ such that
\begin{align}
	\label{eq2.4}
   &\|f\|_{ L_{p,  r_1, \ldots, r_d, q} ( G, w)} \notag\\
   & =
    \big|\int_{\bR}
    \big| \ldots \big|\int_{\bR} \big|\int_{\bR^d} |f|^p (z) 1_{G}(z) \,  dx\big|^{\frac {r_1} p} \, w_1(v_1)  dv_1\big|^{\frac {r_2}{r_1}}\\
     &\qquad \ldots w_d (v_d) dv_d\big|^{\frac q {r_d}}\, w_0(t) dt
    \big|^{\frac 1 q} \notag,\\
&
\label{eq2.5}
\|f\|_{ L_{p;  r_1, \ldots, r_d} (G, |x|^{\alpha} \prod_{i = 1}^d w_i (v_i) )}
 \notag\\
&
=   \big|\int_{\bR} \ldots \big|\int_{\bR} \big|\int_{\bR^{d+1}} |f|^p (z) 1_{G}(z) |x|^{\alpha} \,  dxdt\big|^{\frac {r_1} p} \, w_1(v_1)  dv_1\big|^{\frac {r_2}{r_1}}
\\
     &\qquad  \ldots w_d (v_d) dv_d\big|^{\frac{1}{r_d}}   \notag.
\end{align}
For the discussion of basic properties of weighted mixed-norm  Lebesgue spaces, see \cite{BP_61}.
We write $u \in  S_{p,  r_1, \ldots, r_d, q} (G, w)$
if
$$
    u,
        \partial_t u - v \cdot D_x u,
D_v u,     D^2_v u \in L_{p,  r_1, \ldots, r_d, q} (G, w).
$$
The $S_{p,  r_1, \ldots, r_d, q} (G, w)$-norm
is defined as
$$
	\|u\|_{ S_{p,  r_1, \ldots, r_d, q} (G, w) }
	= \|u\|_{ L_{p,  r_1, \ldots, r_d, q} (G, w) }
	+ \| D_v u\|_{ L_{p,  r_1, \ldots, r_d, q} (G, w) }
$$
$$
	+ \| D_v^2 u\|_{ L_{p,  r_1, \ldots, r_d, q} (G, w) }
+ \| \partial_t u - v \cdot D_x u\|_{ L_{p,  r_1, \ldots, r_d, q} (G, w) }.
$$
If $w \equiv 1$, we drop $w$ from the above notation.
In addition, if $p = r_1 = \ldots =r_d =q$, we replace the subscripts $p, r_1, \ldots, r_d, q$ with a single number $p$.
Replacing $L_{p, r_1, \ldots, r_d, q} (\bR^{1+2d}_T, w)$ with $L_{p; r_1, \ldots, r_d } (\bR^{1+2d}_T, |x|^{\alpha} \prod_{i  =  1}^d w_i (v_i))$, we define the space $S_{p; r_1, \ldots, r_d } (\bR^{1+2d}_T, |x|^{\alpha} \prod_{i =  1}^d w_i (v_i))$.

For any $s \in (0, 1)$ and $u \in L_p (\bR^d)$, by $(-\Delta_x)^s u$ we understand  the  distribution  given by
$$
	((-\Delta_x)^{1/3} u, \phi) = (u, (-\Delta_x)^{1/3} \phi), \quad \phi \in C^{\infty}_0 (\bR^d).
$$
Furthermore, if $s \in (0, 1/2)$ and $u$ is regular enough, say,
$$
	u \, \, \text{is of Lipschitz class on} \, \, \bR^d, \quad	|u (x)| \le K,
$$
then, the following pointwise formula hold:
$$
	(-\Delta_x)^s u (x) = c_{d, s    }\int_{\bR^d} \frac{u (x) - u (x+y)}{|y|^{d+2s}} \, dy,
$$
where $c_{d, s}$ is the constant depending only on $d$ and $s$. See, for instance, the discussion in Section 2 of \cite{S_19}.

For $c > 0$, denote
\begin{equation*}
	\rho_c (z, z_0) =  \max\{ |t-t_0|^{1/2}, c^{-1} |x - x_0 + (t - t_0) v_0|^{1/3}, |v-v_0|\}.
\end{equation*}
For a Lebesgue measurable set $A$, by $|A|$ we denote its Lebesgue measure. For a function $f \in L_1 (A)$, we denote
$$
    (f)_A = \fint_A f \, dz
        = |A|^{-1} \int_A f \, dz
$$
provided that $|A| < \infty$.\
For $c > 0$ and $T \in (-\infty, \infty]$,  the maximal and sharp functions are defined as follows:
  $$
	\bM_{c, T} f (z_0)
	= \sup_{r > 0, z_1 \in  \overline{\bR^{1+2d}_T}: z_0 \in  Q_{r, c r} (z_1) } \fint_{Q_{r, rc} (z_1) } |f (z)| \, dz,
	\quad
		\cM_T f  := \bM_{1, T} f,
$$
  $$
	f^{\#}_{ T} (z_0) = \sup_{r > 0, z_1 \in \overline{\bR^{1+2d}_T}: z_0 \in  Q_r (z_1) } \fint_{Q_{r} (z_1) } |f (z) -  (f)_{Q_{r} (z_1) }| \, dz.
  $$

For $n  \in \{0, 1, 2, \ldots\}$ and a sufficiently regular function $u$ on $\bR^{2d}$,
by $D_x^{n}  u$ we denote the set of all partial derivatives of order $n$ in the $x$ variable.
We define $D_v^n u, D_x^m D_v^n u$ in the same way.

By $N = N (\cdots)$ and $\theta = \theta (\cdots)$ we mean  constants depending only on the parameters inside the
parenthesis. The constants $N$ and $\theta$ might change from line to line. Sometimes, when it is clear what parameters $N$ or $\theta$ depend on, we omit  them.

We impose the following assumptions on the coefficients.
\begin{assumption}
                \label{assumption 2.1}
The coefficients $a (z) = (a^{ij} (z), i, j = 1, \ldots, d)$ are bounded measurable functions  such that  for some $\delta \in (0, 1)$,
$$
    \delta |\xi|^2 \le a^{i j} (z) \xi_i \xi_j \le \delta^{-1} |\xi|^2, \quad \forall \xi \in \bR^d,\, z \in \bR^{1+2d}.
$$
\end{assumption}

\begin{assumption}$(\gamma_0)$
                  \label{assumption 2.2}
There exists $R_0 > 0$ such that for any $z_0$ and $R \in (0, R_0]$,
$$
     \text{osc}_{x, v} (a, Q_r (z_0)) \le \gamma_0,
$$
where
\begin{align*}
  & \text{osc}_{x, v} (a, Q_r (z_0))\\
&= \fint_{(t_0 - r^2, t_0)} \fint_{D_r (z_0, t) \times D_r (z_0, t)}
   |a (t, x_1, v_1) - a (t, x_2, v_2)| \, dx_1dv_1 dx_2dv_2 \, dt,
\end{align*}
  and
  $$
    D_r (z_0, t) =  \{(x, v): |x  - x_0 + (t-t_0) v_0|^{1/3} < r, |v-v_0| < r      \}.
$$
\end{assumption}

Here is an example of a somewhat stronger condition
which can be viewed as a $VMO_{x, v}$ condition with respect to the anisotropic distance
$|x-x'|^{1/3} + |v-v'|$.
Let $\omega: [0, \infty) \to [0, \infty)$
 be an increasing function
such that $\omega (0+) = 0$.
Assume that
\begin{equation}
			\label{eq2.3.1}
\begin{aligned}
	&
\text{osc}_{x, v}' (a,r)    \\
	&:= \sup_{t, x, v} r^{-8d} \int_{  x_1, x_2 \in B_{r^3} (x) } \int_{v_1, v_2 \in B_r (v) } |a (t, x_1, v_1)\\
&\qquad\qquad - a (t, x_2, v_2)| \, dx_1dx_2\, dv_1dv_2 \le \omega (r).
\end{aligned}
\end{equation}
Note that since
$$
	\text{osc}_{x, v} (a, Q_r (z_0)) \le \text{osc}_{x, v}' (a,r),
$$
 the condition \eqref{eq2.3.1} implies Assumption \ref{assumption 2.2} $(\gamma_0)$ for any $\gamma_0 \in (0, 1)$.

In the present article, we consider the equation
 \begin{equation}
                \label{2.1.1}
        P u +  b^i D_{v_i} u + (c  + \lambda) u = f
\end{equation}
and for $-\infty<S<T\le \infty$, the Cauchy problem
\begin{equation}
                \label{2.1}
     P u + b^i D_{v_i} u + c u = f, \quad u (S, \cdot) = 0.
\end{equation}

\begin{assumption}
                  \label{assumption 2.3}
The functions $b  = (b^i, i = 1, \ldots, d)$ and $c$ are bounded measurable  on $\bR^{1+2d}$
 and they satisfy  the condition
$$
	|b| + |c| \le L
$$
for some constant $L > 0$.
\end{assumption}

\subsection{Main result}

\begin{definition}
                \label{definition 2.1}
Let $T \in (-\infty, \infty]$. A function $u \in S_{p, r_1, \ldots, r_d, q} (\bR^{2d}_T, w)$
is a solution to  Eq. \eqref{2.1.1} if the identity
$$
	\partial_t u -  v \cdot D_x u = a^{i j} D_{v_i v_j} u - b^i D_{v_i} u - (c + \lambda) u
$$
holds in $L_{p, r_1, \ldots, r_d, q} (\bR^{2d}_T, w)$.
We define a solution in the space
$$
S_{p; r_1, \ldots, r_d} (\bR^{2d}_T, |x|^{\alpha} \prod_{i = 1}^d w_i (v_i))
$$
in a similar way.

We say that
$$
u \in S_{p, r_1, \ldots, r_d, q} ((S, T) \times \bR^{2d}, w)
$$
is a solution to Eq \eqref{2.1} if there exists $\widetilde u \in S_{p, r_1, \ldots, r_d, q} (\bR^{1+2d}_T, w)$ such that $\widetilde u = u$
on $(S, T) \times \bR^{2d}$,
$\widetilde u = 0$ on $(-\infty, S)  \times \bR^{2d}$, and
$     P \widetilde  u   + b^i D_{v_i} \widetilde u + c \widetilde u  = f$  on $(S, T) \times \bR^{2d}$.
\end{definition}

\begin{theorem}
            \label{theorem 2.1}
Let $p$, $r_1, \ldots, r_d$, $q > 1$, $K \ge 1$ be numbers,
$T \in (-\infty, \infty]$,
and $w_i, i = 0, 1, \ldots, d,$ be weights satisfying \eqref{eq2.7}, and
$w$ be a weight defined by \eqref{eq2.6}.
Let
Assumptions \ref{assumption 2.1} and \ref{assumption 2.3}
hold. There exist constants
$$
	\beta = \beta (d, p, r_1, \ldots, r_d, q, K) > 0,
\quad
\gamma_0   =  \delta^{\beta}  \widetilde \gamma_0  (d, p, r_1, \ldots, r_d, q, K)  > 0$$
such that if Assumption \ref{assumption 2.2} $(\gamma_0)$ holds,
then, the following assertions are valid.

$(i)$ There exist constants
$$
\theta = \theta (d, p, r_1, \ldots, r_d, q, K)\quad \text{and}\quad
\lambda_0 =  \delta^{-\theta} R_0^{-2} \widetilde \lambda_0 (d, p, r_1, \ldots, r_d, q, K, L)  \ge 1
$$
such that
for any $\lambda \ge \lambda_0$
and any  $u \in S_{p, r_1, \ldots, r_d, q} (\bR^{1+2d}_T, w)$,
one has
\begin{align}
			\label{2.1.2}
 	&\lambda  \|u\| + \lambda^{1/2}  \|D_v u\| + \|D^2_v u\| +  \|(-\Delta_x)^{1/3} u\|
	+\|D_v (-\Delta_x)^{1/6} u\|
	\notag\\
& \quad + \|\partial_t u - v \cdot D_x u\|  \le N  \delta^{-\theta} \|P u +   b^i D_{v_i} u + c u + \lambda u\|,
\end{align}
where $R_0 \in (0, 1)$ is the constant in  Assumption \ref{assumption 2.2} $(\gamma_0)$,
 $$
 \|\,\cdot\,\| = \|\,\cdot\,\|_{ L_{p, r_1, \ldots, r_d, q} (\bR^{1+2d}_T, w)  },
\quad\text{and}\quad N = N (d, p, r_1, \ldots, r_d, q, K).
$$
In addition, for any $f \in L_{p, r_1, \ldots, r_d, q} (\bR^{1+2d}_T, w)$,
 Eq. \eqref{2.1.1} has a unique solution $u \in S_{p, r_1, \ldots, r_d, q} (\bR^{1+2d}_T, w).$

$(ii)$ For any numbers $-\infty<S < T<\infty$
and
$f \in L_{p, r_1, \ldots, r_d, q} ((S, T) \times \bR^{2d}, w)$,
Eq. \eqref{2.1} has a unique solution
$u  \in S_{p, r_1, \ldots, r_d, q} ((S, T) \times \bR^{2d}, w)$.
In addition,
\begin{align*}
& 	  \|u\| + \|D_v u\| + \|D^2_v u\| +  \|(-\Delta_x)^{1/3} u\|
	+\|D_v (-\Delta_x)^{1/6} u\|
	+ \|\partial_t u - v \cdot D_x u\|\\
&    \le N \|f\|,
\end{align*}
where $\|\,\cdot\,\|=\|\,\cdot\,\|_{L_{p, r_1, \ldots, r_d, q} ((S, T) \times \bR^{2d}, w)  }$ and
$$
N = N (d, \delta, p,r_1, \ldots, r_d, q, K, L, T-S).
$$

$(iii)$ The assertions $(i)$ and $(ii)$ hold   with $S_{p, r_1, \ldots, r_d, q} (\bR^{1+2d}_T, w)$ replaced with  $S_{p; r_1, \ldots, r_d} (\bR^{1+2d}_T, |x|^{\alpha} \prod_{i = 1}^d w_i (v_i))$, where  $\alpha \in (-1, p-1)$. Furthermore, the constants $\beta, \gamma_0, \theta, \lambda_0, N$ must be modified as follows: one needs to take into account the dependence on $\alpha$ and remove the dependence on $q$.
\end{theorem}

\begin{remark}
The reason why we included the term $D_v (-\Delta_x)^{1/6}$ in the a priori estimate \eqref{2.1.2} is the following.
In the  proof of Theorem \ref{theorem 2.1} with $P = P_0$, $L_{p, r_1, \ldots, r_d, q} (\bR^{1+2d}_T, w)$ replaced with $L_p (\bR^{1+2d})$, and $p \in (1, 2)$ (see Theorem \ref{theorem 4.10}),
we use an a priori bound of the $L_{p/(p-1)} (\bR^{1+2d})$ norm of  $D_v (-\Delta_x)^{1/6} u$ to prove  \eqref{2.1.2} for $\|D_v^2 u\|_{  L_p (\bR^{1+2d})}$. See page \pageref{4.10.8}.
It turns out that the former inequality  can be obtained along the lines  of the proof of the a priori estimate of $(-\Delta_x)^{1/3} u$.
\end{remark}

The following result is a direct corollary of Theorem \ref{theorem 2.1} $(i)$.
\begin{corollary}
			\label{corollary 2.2}
There exist constants
$$
\beta = \beta (d, p, r_1, \ldots, r_d, q, K) > 0,
\quad	\gamma_0  =   \delta^{\beta}  \widetilde \gamma_0 (d, p, r_1, \ldots, r_d, q, K)  > 0,
$$
and
$$
\theta = \theta (d, p, r_1, \ldots, r_d, q, K) > 0,
$$
such that if Assumption \ref{assumption 2.2} $(\gamma_0)$ holds,
then for any $u \in S_{p, r_1, \ldots, r_d, q} (\bR^{1+2d}_T, w)$ and $\lambda \ge 0$,
\begin{align}
				\label{2.2.1}
	&\|u\|_{S_{p, r_1, \ldots, r_d, q} (\bR^{1+2d}_T , w  ) }
+  \|(-\Delta_x)^{1/3} u\|
	+\|D_v (-\Delta_x)^{1/6} u\|\notag\\
&    \le N  \delta^{-\theta} ( \|P u +   b^i D_{v_i} u + c u + \lambda u\|+ R_0^{-2}\|u\|),
\end{align}
where   $\|\cdot\| = \|\cdot\|_{ L_{p, r_1, \ldots, r_d, q} (\bR^{1+2d}_T, w  )}$,     $N  = N (d, p, r_1, \ldots, r_d, q, K, L)$ , and $R_0 \in (0, 1)$ is the constant from Assumption \ref{assumption 2.2} $(\gamma_0)$.
Furthermore, the same holds if we replace $S_{p, r_1, \ldots, r_d, q} (\bR^{1+2d}_T, w)$
 with $S_{p; r_1, \ldots, r_d} (\bR^{1+2d}_T, |x|^{\alpha} \prod_{i =  1}^d w_i (v_i))$, where
 $\alpha \in (-1, p-1)$,
and modify the constants as suggested in Theorem \ref{theorem 2.1} $(iii)$.
\end{corollary}

\begin{proof}
We only need to consider the case when $\lambda \in (0, \lambda_0)$, where $\lambda_0$ is the constant in Theorem \ref{theorem 2.1}.
Then, by 	\eqref{2.1.2} and the triangle inequality, \eqref{2.2.1} holds with the term $(\lambda_0 - \lambda) \|u\|$ in place of $R_0^{-2} \|u\|$.
Replacing $\lambda_0 - \lambda$ with $\lambda_0$ and using the explicit expression of $\lambda_0$, we prove \eqref{2.2.1}.
\end{proof}

In the next remark, we explain how Corollary \ref{corollary 2.2} can be applied to the linearized Landau equation. We also show why it is useful to know how the constant on the right-hand side of \eqref{2.2.1} depends on  the lower eigenvalue bound $\delta$ of the matrix $a^{i j}$.
\begin{remark}
			\label{remark 2.3}
Let $d=3$. Here we show how Corollary \ref{corollary 2.2}  can be used to estimate the  $S_p$-norm of  the solution to the linearized Landau equation near Maxwellian \eqref{1.0} in the case of the Coulomb interaction. See  the details in  \cite{DGO_20}.
Due to Lemma 2.4 of \cite{KGH_20}, for Eq. \eqref{1.0}, there exist constants $\mu_1, \mu_2 > 0$ such that
\begin{equation}
			\label{2.3.2}
	\mu_1 (1+|v|)^{-3}|\xi|^2  \le a^{i j} (z) \xi_i \xi_j \le \mu_2 (1+|v|)^{-1} |\xi|^2, \quad \forall  z \in \bR^7,\, \forall\xi \in \bR^3.
\end{equation}
In addition, we assume that the coefficients $a^{i j}$ are H\"older continuous with respect to
 the quasi-distance
$$
	\text{d} (z, z')  =  \max\{ |t-t_0|^{1/2},  |x - x_0 - (t - t_0) v_0|^{1/3}, |v-v_0|\}.
$$
This means that $a^{i j}$ are bounded functions, and there exists a constant $\kappa \in (0, 1)$ such that
$$
	  \sup_{z, z' \in (0, T) \times \bR^6: z \ne z'} \frac{|a^{i j} (z) - a^{i j} (z')|}{  \text{d} ^{\kappa} (z, z')} < \infty, \, \, i, j = 1, 2, 3.
$$
Eq. \eqref{1.0} with the above  assumptions arises naturally  when one tries to prove the existence and uniqueness of the solution to the (nonlinear) Landau equation near Maxwellian (see \cite{KGH_20}, \cite{DGO_20}).
Note that, in this case,  Assumption \ref{assumption 2.2} $(\gamma_0)$ holds for any $\gamma_0 \in (0, 1)$ with
\begin{equation}
			\label{2.3.3}
	R_0 \sim \gamma_0^{1/\kappa}.
\end{equation}

Next, assuming that $f_0$ is a sufficiently regular function and replacing $f$ with $f - f_0\phi$, where $\phi  = \phi (t)$ is an appropriate cutoff function, we reduce Eq. \eqref{1.0}  to the Cauchy problem
$$
    P f + b \cdot \nabla_v f  + c f  = \eta, \quad f (0, \cdot, \cdot) \equiv 0.
$$
Then, localizing \eqref{2.2.1} and using \eqref{2.3.2} and \eqref{2.3.3}, for any  $n  \in  \{1, 2, \ldots, \}$, we obtain
\begin{equation}
			\label{2.3.1}
\begin{aligned}
&\|f\|_{S_{2}((0, T) \times  \bR^3 \times \{ n < |v| < n+1\}) } \\
&\le N n^{\theta}  \||\eta| +  |f|\|_{  L_{  2  } ((0, T) \times   \bR^3 \times \{n - 1/2 < |v| < n+ 3/2  \} )},
\end{aligned}	
\end{equation}
where $N$ and $\theta$ are positive numbers independent of $n$ and $T$.
If the original initial value and $h$ decay fast enough at infinity,  by using the energy estimate for the Landau equation, one can show that
$f, \eta \in L_{ 2}((0, T) \times \bR^6, |v|^{\theta'})$ for some $\theta' > \theta +1 > 0$. This combined with \eqref{2.3.1} gives
$$
	\|f\|_{S_{2} ((0, T) \times \bR^3 \times  \{ n < |v| < n+1\}) }  \le N n^{\theta - \theta'}  \||\eta| + |f|\|_{  L_{  2 } ((0, T) \times   \bR^6, |v|^{\theta'}) },
$$
and, hence, $f \in S_{2} ((0, T) \times \bR^6)$.
By a Sobolev type embedding theorem for $S_p$ spaces (see Theorem 2.1 in \cite{PR_98}), the above gives
$$
	\||f| + |\nabla_v f|\|_{  L_{7/3} ((0, T) \times \bR^6)  } \le N    \||\eta| + |f|\|_{  L_2  ((0, T) \times   \bR^6, |v|^{\theta'}) }.
$$
Similarly, one can show that $f$ belongs to a weighted $S_2$ and $L_{7/3}$ spaces.
Then, by using a bootstrap type argument, we conclude that $f \in S_p ((0, T) \times \bR^6)$ for any $p\in [2,\infty)$.
If $p$ is large enough, by using  a Morrey type  embedding theorem for the $S_p$ spaces (see  Theorem 2.1 in \cite{PR_98}), one concludes that
$f, \nabla_v f $ are H\"older continuous with respect to the  quasi-distance $\text{d}$, which is crucial in the proof of the uniqueness of solutions (see, for example, \cite[Lemma 8.2]{KGH_20}).
As mentioned in Section \ref{section 1}, the present authors used a similar argument to show the higher regularity of  a finite energy weak solution to the linear Landau equation  with the  specular reflection  boundary condition (see \cite{DGY_21}).
In particular,  near the boundary, one can reduce such an equation to an equation of the KFP type with   the coefficients  $L_{\infty} ((0, T), C^{\varkappa/3, \varkappa}_{x, v} (\bR^6))$, $\varkappa \in (0, 1]$, where
$C^{\varkappa/3, \varkappa}_{x, v} (\bR^6)$ is the space of bounded functions $u$ such that
$$
	\sup_{ (x_i, v_i) \in \bR^6:  (x_1, v_1) \neq (x_2, v_2) } \frac{|u (x_1, v_1)-  u (x_2, v_2)|}{(|x_1-x_2|^{1/3}+|v_1-v_2|)^{\varkappa}} < \infty.
$$
In this case, again, Assumption \ref{assumption 2.2} $(\gamma_0)$ holds for any $\gamma_0 \in (0, 1)$ with $R_0$ given by   \eqref{2.3.3}.
\end{remark}

\begin{remark}
The assertion $(ii)$ is derived from $(i)$  in the standard way (see, for example, Theorem  2.5.3 of \cite{Kr_08}).
We will not mention this in the sequel.
\end{remark}

\begin{remark}
From Theorem \ref{theorem 2.1} $(i)$, we can derive the corresponding results for elliptic equations when the coefficients and data are independent of $t$. See, for instance, the proof of \cite[Theorem 2.6]{Kr07}. The idea is that one can view an elliptic equation as a steady state parabolic equation.
\end{remark}

\begin{remark}
It would be interesting to study Eq. \eqref{1.0} with singular drift  under the conditions similar to those considered in \cite{Kr_R_05}, \cite{Kr_21}, \cite{RZ_21}, \cite{SX_20} (see also the references therein).
An interesting result concerning a  Langevin type SDE with the drift in the form $b (t, x, v) =  D_x F (x) + G (v)$ is established in  \cite{SX_20}.
\end{remark}

\section{Auxiliary results}
                        \label{section 3}
 The following lemma is a variant of Lemma 2.3 of \cite{CZ_19}.
\begin{lemma}
						\label{lemma 4.8}
Let $c \geq 1, r > 0$ be  numbers.
Then, the following assertions hold.

$(i)$ For any $z, z_0 \in \bR^{1+2d}$,
$$
	\rho_c (z, z_0) \leq 2 \rho_c (z_0, z).
$$

$(ii)$
For any $z, z_0, z_1 \in \bR^{1+2d}$,
$$
	\rho_c (z, z_0)
	\leq 2 (\rho_c (z, z_1) + \rho_c (z_1, z_0)).
$$

$(iii)$ Denote $  \widehat \rho_c (z, z_0) = \rho_c (z, z_0)+ \rho_c (z_0, z)$.
Then, $\widehat \rho_c$ is a (symmetric) quasi-distance.

$(iv)$ Denote $ \widehat{Q}_{r, cr} (z_0) = \{z \in \bR^{1+2d}:   \widehat \rho_c (z, z_0) < r\}$.
Then,
	$$
			 \widehat{Q}_{r, cr} (z_0) \subset  \tQ_{r, cr} (z_0) \subset   \widehat{Q}_{ 3r, 3 c r} (z_0).
	$$

$(v)$ For any $T \in (-\infty, \infty]$,
the triple $(\overline{\bR^{1+2d}_T}, \widehat \rho_c, dz)$ (with the induced topology if $T < \infty$) is a space of homogeneous type with parameters independent of $c$.

\end{lemma}

\begin{proof}

$(i)$  It suffices to show that
$$
	c^{-1} |x - x_0 + (t - t_0)v_0|^{1/3} \leq 2 \rho_c (z_0, z).
$$
By the triangle inequality, we have
\begin{align*}
&	 |x - x_0 + (t - t_0) v_0|^{1/3}
	 \leq |x_0 - x + (t_0 - t) v|^{1/3}
	 + |t-t_0|^{1/3} |v-v_0|^{1/3}\\
&	\leq |x_0 - x + (t_0 - t) v|^{1/3} + \rho_c (z_0, z).
\end{align*}
Multiplying both sides of the above inequality by $c^{-1}$ and using the fact that $c \geq 1$, we prove the claim.

$(ii)$ As in $(i)$, we only need to show that the inequality holds with the left-hand side replaced with
	$
		c^{-1} |x - x_0 + (t - t_0)v|^{1/3}.
	$
	By the triangle inequality,
\begin{align*}
&		|x - x_0 + (t-t_0) v_0|^{1/3}\\
&\leq |x   -x_1 + (t - t_1) v_1|^{1/3} + |x_1 - x_0 + (t - t_0) v_0 - (t - t_1) v_1|^{1/3}\\
&		\leq  c \rho_c (z, z_1) + c \rho_c (z_1, z_0) + | t - t_1|^{1/3} |v_0 - v_1|^{1/3}.
\end{align*}
	   By Young's inequality,
	$$
		|t-t_1|^{1/3} |v_0 - v_1|^{1/3}
		\leq
		(2/3) \rho_c (z, z_1) + (1/3) \rho_c (z_1, z_0).
	$$
	This combined with the fact that $c \geq 1$ yields the desired estimate.

	$(iii)$ This assertion follows from $(ii)$.

	$(iv)$ The claim is a direct corollary of $(i)$.

$(v)$
We only need to check the doubling property.
	For any  $z \in \overline{\bR^{1+2d}_T}$, the assertion $(iv)$ of this lemma gives
	$$
		\mu:  =	\frac{| \widehat{Q}_{2r, 2cr} (z)  \cap \bR^{1+2d}_T |}{|\widehat{Q}_{r, cr} (z)   \cap \bR^{1+2d}_T |}
	\le
			\frac{| \tQ_{2r, 2cr} (z)   |}{|\tQ_{r/3, cr/3} (z)   \cap \bR^{1+2d}_T |}.
	$$
	Since $z \in \overline{\bR^{1+2d}_T}$, we have
	$$
 		 Q_{r/3, cr/3} (z) \subset  \tQ_{r/3, cr/3} (z)   \cap \bR^{1+2d}_T,
	$$
	and, hence,
	\begin{equation}
				\label{4.8.1}
		\mu \le N (d).
	\end{equation}
	The claim is proved.
\end{proof}

Denote
$$
	 f^{\#}_{c, T} (z_0) = \sup_{r > 0, z_1 \in \overline{\bR^{1+2d}_T}: z_0 \in   Q_{r, c r} (z_1)  } \fint_{  Q_{r, c r} (z_1) } |f (z) -  (f)_{  Q_{r, c r} (z_1)   }| \, dz.
$$

\begin{corollary}
			\label{corollary 4.9}
Let  $c \geq 1$, $K \ge 1$, $p, q, r_1, \ldots, r_d > 1$
be  numbers,
$T \in (-\infty, \infty]$,
and $f \in L_{p, r_1, \ldots, r_d,q} (\bR^{1+2d}_T, w)$, where $w$ is a weight such that
$$
[w]_{q, r_1, \ldots, r_d}\le K.
$$
Then, the following assertions hold.

$(i)$ (Hardy-Littlewood type theorem)
$$
	\| \bM_{c, T} f \|_{ L_{p,r_1, \ldots, r_d, q} (\bR^{1+2d}_{T}, w) }
	\leq N (d, p, q, r_1, \ldots, r_d, K) \| f \|_{ L_{p, r_1, \ldots, r_d, q} (\bR^{1+2d}_T, w) }.
$$

$(ii)$ (Fefferman-Stein type theorem)
$$
	\| f \|_{ L_{p , r_1, \ldots, r_d, q} (\bR^{1+2d}_T, w) }
	\leq N (d, p, q, r_1, \ldots, r_d, K)	
	\| f^{\#}_{c, T} \|_{  L_{p,  r_1, \ldots, r_d, q} (\bR^{1+2d}_T, w) }.
$$
\end{corollary}

\begin{proof}
We only present the proof of the assertion in the case when $T < \infty$. The case when $T = \infty$
is treated similarly.
We extend $f$ to be zero for $t>T$.

$(i)$
Denote
$$
	 \widehat{\bM}_{c, T} f  (z_0)
	= \sup_{r > 0, z_1 \in   \overline{\bR^{1+2d}_T}: z_0 \in   \widehat{Q}_{r,cr} (z_1)}
	\fint_{   \widehat{Q}_{r,cr} (z_1) \cap \bR^{1+2d}_T } |f (z)| \, dz,
$$
where $\widehat{Q}_{r, cr} (z_1)$ is defined in Lemma \ref{lemma 4.8}.
By Lemma \ref{lemma 4.8} $(iv)$, for any $r > 0$ and $z_1 \in \overline{\bR^{1+2d}_T}$,
\begin{align*}
	\fint_{Q_{r, cr} (z_1)} |f (z)| \, dz
	&\leq 	\frac{|    \widehat{Q}_{3r,3cr}   |}{|Q_{r, cr}|} \fint_{    \widehat{Q}_{3r, 3cr} (z_1)  \cap \bR^{1+2d}_T  } |f (z)| \, dz\\
& 	\leq     \frac{|  \tQ_{3r, 3cr}|}{|Q_{r, cr}|} \fint_{    \widehat{Q}_{3 r, 3 c r}  (z_1)  \cap \bR^{1+2d}_T    } |f (z)| \, dz\\
&	= N (d)\fint_{ \widehat{Q}_{3r, 3cr} (z_1) \cap \bR^{1+2d}_T} |f (z)| \, dz,
\end{align*}
and by this
\begin{equation}
			\label{4.9}
	 \bM_{c, T} f (z)  \leq
	 N (d) \widehat{\bM}_{c, T} f (z), \quad \forall z \in \bR^{1+2d}_T.
\end{equation}

We recall that $\widehat{\bM}_{c, T} f$ is a maximal function on a space of homogeneous type $(\overline{\bR^{1+2d}_T}, \widehat \rho_c, dz)$ (see Lemma \ref{lemma 4.8} $(v)$).
Then, by the weighted Hardy-Littlewood theorem (see \cite{AM_84}), for any $\omega_0\in A_p(\bR^{1+d})$,
$$
    \int_{   \bR^{1+2d}_T  } | \widehat{\bM}_{c, T} f (z)|^p \omega_0 (t, v)  \, dz \leq N (d, p, [\omega_0]_{A_p (\bR^{d+1})})
    \int_{\bR^{1+2d}_T} |f (z)|^p \omega_0 (t, v)  \, dz.
$$
By a variant of the Rubio de Francia extrapolation theorem (see, for example, Theorem 7.11 of \cite{DK19} and also \cite{Kr_20}),
$$
    \|   \widehat{\bM}_{c, T} f \|_{ L_{p, r_1, \ldots, r_d, q} (\bR^{1+2d}_T, w)   }
    \leq N  \|f\|_{ L_{p, r_1, \ldots, r_d, q} (\bR^{1+2d}_T, w)   }.
$$
Now the assertion follows from \eqref{4.9}.

$(ii)$
Let
$$
	\widehat f^{\#}_{c, T} (z_0) = \sup_{r > 0, z_1 \in \overline{\bR^{1+2d}_T}: z_0 \in  \widehat{Q}_{r, c r} (z_1)  } \fint_{ \widehat{Q}_{r, c r} (z_1)   \cap \bR^{1+2d}_T} |f (z) -  (f)_{ \widehat{Q}_{r, c r} (z_1)  \cap  \bR^{1+2d}_T }| \, dz,
$$
$$
    \widetilde f^{\#}_{c, T} (z_0)
    = \sup_{r > 0, z_1 \in \overline{\bR^{1+2d}_T}: z_0 \in  \tQ_{r,cr} (z_1)}
    \fint_{  \tQ_{r, cr} (z_1)  \cap \bR^{1+2d}_T }
    |f (z)
    - (f)_{ \tQ_{r, cr}(z_1) \cap \bR^{1+2d}_T }| \, dz.
$$
Clearly, $ \widehat f^{\#}_{c, T}$ is a sharp function defined on the space of homogeneous type
$(\overline{\bR^{1+2d}_T}, \widehat \rho_c, dz)$.
By the  generalized Fefferman-Stein theorem for spaces of homogeneous type
(see Theorem 2.3 of \cite{DK_18}) combined with the extrapolation argument as above (see Theorem 7.11 of \cite{DK19}),
 we get
$$
	\|f\| \le N	\|\widehat f^{\#}_{c, T}\|.
$$
Therefore, it suffices to show that
for any $z_0 \in \bR^{1+2d}_T$,
\begin{equation}
       \label{4.9.1}
      \widehat f^{\#}_{c,T} (z_0)
    \leq N (d) f^{\#}_{c,T} (z_0).
\end{equation}

We prove this inequality in two steps.

\textit{Step 1.} First, we show that
\begin{equation}
			\label{4.9.2}
	 \widehat f^{\#}_{c, T} (z_0) \le N (d) \widetilde f^{\#}_{c, T} (z_0).
\end{equation}
We fix any cylinder $\widehat{Q}_{r, c r} (z_1)$ containing $z_0$ such that $z_1 \in \overline{\bR^{1+2d}_T}$.
 By Lemma 3.1 $(iv)$ and the doubling property (see \eqref{4.8.1}),
\begin{align}
\label{eq4.9.2.1}
&   \fint_{ \widehat{Q}_{r, cr} (z_1)  \cap \bR^{1+2d}_T  }
    |f (z)
    - (f)_{ \widehat{Q}_{r, cr}(z_1)  \cap \bR^{1+2d}_T  }| \, dz\\
&  \leq \frac{| \tQ_{r, cr} (z_1)  \cap \bR^{1+2d}_T |^2}{|\widehat{Q}_{r, cr} (z_1)   \cap \bR^{1+2d}_T |^2}
 \fint_{    \tQ_{r, cr} (z_1)   \cap \bR^{1+2d}_T }
    \fint_{  \tQ_{r, cr} (z_1)  \cap \bR^{1+2d}_T  }
    |f (z) - f (z')|\, dz\, dz' \notag\\
&    \leq
     N (d)   \frac{|  \tQ_{r, cr} (z_1)  \cap \bR^{1+2d}_T |^2}{|\widehat{Q}_{3 r, 3 c r} (z_1)   \cap \bR^{1+2d}_T |^2}
   \fint_{  \tQ_{r, cr} (z_1)  \cap \bR^{1+2d}_T }
    |f (z)
    - (f)_{   \tQ_{r, cr}(z_1)  \cap \bR^{1+2d}_T}| \, dz \notag\\
&       \leq N (d)
        \fint_{  \tQ_{r, cr} (z_1)  \cap \bR^{1+2d}_T }|f (z)
    - (f)_{ \tQ_{r, cr}(z_1)  \cap \bR^{1+2d}_T }| \, dz \notag,
\end{align}
which implies \eqref{4.9.2}.

\textit{Step 2.} We claim that
\begin{equation}
                                \label{eq10.23}
	\widetilde f^{\#}_{c, T}  (z_0) \le N (d)    f^{\#}_{c, T}  (z_0).
\end{equation}
For the sake of simplicity, we assume that $T  = 0$.
Let $ \tQ_{r, r c } (z_1)$  be a cylinder containing $z_0$
 such that $z_1 \in \overline{\bR^{1+2d}_{0}}$.
Note that if $t_1 < - r^2$, one has
$$
	 \tQ_{r, c r} (z_1)  \subset Q_{2r, 2 c r} (t_1 + r^2, x_1 -  r^2 v_1, v_1) \subset \bR^{1+2d}_0,
$$
and then,
\begin{align*}
&	\fint_{ \tQ_{r, c r} (z_1) \cap \bR^{1+2d}_0 } |f (z) -  (f)_{ \tQ_{r, c r} (z_1) \cap \bR^{1+2d}_0 }| \, dz\\
&	\le
	N (d) \fint_{  Q_{2 r, 2 c r} (t_1 + r^2, x_1 - r^2 v_1, v_1) } |f (z) -  (f)_{ Q_{2 r, 2 c r} (t_1 + r^2, x_1 - r^2 v_1, v_1)  }| \, dz\\
&		\le N (d)   f^{\#} (z_0).
\end{align*}
Next, if $t_1 \ge -r^2$, then,
$$
	 Q_{r, c r} (z_1) \subset  \tQ_{r, c r} (z_1) \cap \bR^{1+2d}_0 \subset \bar Q_{2 r, 2 c r} (0, x_1 + t_1 v_1, v_1).
$$
By this,
\begin{align*}
&	\fint_{  \tQ_{r, c r} (z_1)  \cap \bR^{1+2d}_0  } |f (z) -  (f)_{ \tQ_{r, c r} (z_1)  \cap \bR^{1+2d}_0  }| \, dz\\
&\le
	N (d) \fint_{  Q_{2 r, 2 c r} (0, x_1 + t_1 v_1, v_1)   } |f (z) -  (f)_{ Q_{2 r, 2 c r} (0, x_1 + t_1 v_1, v_1)   }| \, dz\\
&		\le N (d)   f^{\#}_{c, T} (z_0),
\end{align*}
which proves the claim. Thus, \eqref{4.9.1} holds, and the lemma is proved.
\end{proof}

\begin{corollary}
			\label{corollary 3.3}
Let $p > 1$, $K \ge 1$, $c \ge 1$ be numbers, $T \in (-\infty, \infty]$ and $\omega$ be an  $A_p$ weight  on the space of homogeneous type $(\overline{\bR^{1+2d}_T}, \rho_c, dz)$ (see Lemma \ref{lemma 4.8} $(v)$)
such that its $A_p$ constant is bounded by $K$.
Then, for any $f \in L_{p} (\bR^{1+2d}_T, \omega)$,
\begin{align*}
&(i) \, \|\bM_{c, T} f\|_{   L_{p} (\bR^{ 1+2d    }, \omega)  } \le N (d, p, K)  \|f\|_{   L_{p} (\bR^{1+2d}_T, \omega)  },\\
&(ii) \, \|f\|_{   L_{p} (\bR^{1+2d}_T, \omega)  } \le N (d, p, K) \|f^{\#}_{c, T}\|_{   L_{p} (\bR^{1+2d}_T, \omega)  }.
\end{align*}
\end{corollary}

\begin{proof}
As we pointed out in the proof of Corollary \ref{corollary 4.9}, the aforementioned inequalities hold  with
$$
	\bM_{c, T} f,  f^{\#}_{c, T}   \quad \text{replaced with} \quad  \widehat \bM_{c, T} f, \widehat  f^{\#}_{c, T},
$$
 respectively.
To conclude the validity of the assertions, we invoke \eqref{4.9}, \eqref{4.9.2}, and \eqref{eq10.23}.
\end{proof}

\begin{lemma}
		\label{lemma 3.2}
Let $p \in (1, \infty),  r > 0$, $c \ge 1$, $ \alpha \in (-1, p-1)$ be numbers,
and $T \in (-\infty, \infty]$.
Then, $|x|^{\alpha}$
is an $A_p$ weght on the space of homogeneous type $(\overline{\bR^{1+2d}_T}, \widehat \rho_c, dz)$,
and, furthermore, for any $r > 0$, $z_0 \in \overline{\bR^{1+2d}_T}$,
\begin{equation}
\begin{aligned}
			\label{eq3.2.1}
	\mathcal{A} (r, z_0) &= \fint_{\widehat  Q_{r, c r} (z_0) \cap \overline{\bR^{1+2d}_T} } |x|^{\alpha} \, dz  \\
	&  \bigg(\fint_{ \widehat  Q_{r, c r} (z_0) \cap \overline{\bR^{1+2d}_T}  } |x|^{-\alpha/(p-1)} \, dz\big)^{p-1} \le N (d, \alpha, p).
\end{aligned}
\end{equation}
 \end{lemma}

\begin{proof}
Allowing a constant $N$  in \eqref{eq3.2.1} to depend on $d$, we may replace  $\widehat  Q_{r, c r} (z_0) \cap \overline{\bR^{1+2d}_T}$ with $\widetilde  Q_{r, c r} (z_0) \cap \overline{\bR^{1+2d}_T}$ (see the argument in \eqref{eq4.9.2.1}).
Furthermore, in the case $T < \infty$, we may  assume that $T =  0$ and replace $\widetilde  Q_{r, c r} (z_0) \cap \overline{\bR^{1+2d}_0}$ with $Q_{r, c r} (z_0)$ in the expression for $\mathcal{A} (r, z_0)$.
This follows  from the fact that if $t_0 \le -r^2$, then
$$
	 \tQ_{r, c r} (z_0)  \subset Q_{2r, 2 c r} (t_0 + r^2, x_0 -  r^2 v_0, v_0) \subset \bR^{1+2d}_0,
$$
and, otherwise,
$$
	  \tQ_{r, c r} (z_0) \cap \bR^{1+2d}_0 \subset  Q_{2 r, 2 c r} (0, x_0 + t_0 v_0, v_0).
$$

We may assume that $|x_0|+|v_0| > 0$.
 Denote $X_0 = r^{-3} x_0$, $V_0 = r^{-1} v_0$. By the argument of the previous paragraph and a change of variables, we get
\begin{align*}
	\mathcal{A} (r, z_0) \le & N (d) \int_{0}^1 \fint_{B_{c^3} (X_0 + t V_0)}  |x|^{\alpha } \, dxdt   \\
& \times  \bigg(\int_{0}^1 \fint_{B_{c^3} (X_0 + t V_0)} |x|^{- \alpha/(p-1)} \,   dxdt\big)^{p-1}.
\end{align*}
We will consider two cases:
$$
	A: |X_0 + t  V_0| > 3 c^3, \, \, \forall t \in (0, 1), \quad B: \exists \tau \in (0, 1): |X_0 + \tau  V_0| \le 3 c^3.
$$

\textit{Case A.} Note that for any $t \in (0, 1)$, $x \in B_{c^3} (X_0 + t V_0)$, one has
\begin{align*}
  	&|x| \le |x - X_0 - t   V_0| + |X_0 + t V_0| \le  (4/3)|X_0+t V_0|, \\
&  |x| \ge |X_0+t V_0|  - |x - X_0 - t V_0| \ge (2/3)  |X_0+t V_0|,
\end{align*}
and hence,
\begin{align*}
	 \mathcal{A} (r, z_0) & \le N (\alpha, d)  \int_0^1  |X_0 + t V_0|^{\alpha} \, dt \bigg(\int_0^1  |X_0 + t V_0|^{-\alpha/(p-1)} \, dt\big)^{p-1}\\
&
	\le N ( \alpha, d) \int_0^1 |\omega' + t \omega|^{\alpha}\, dt  \bigg(\int_0^1  |\omega' + t \omega|^{-\alpha/(p-1)} \, dt\big)^{p-1},
\end{align*}
where  $\omega =  (|X_0|^2+|V_0|^2)^{-1/2} V_0$, $\omega' =  (|X_0|^2+|V_0|^2)^{-1/2} X_0$.

Next, if $|\omega'| \ge 2|\omega|$,  we have
$$
	\frac 1 2 |\omega'|	
\le 	|\omega' + t \omega| \le 2 |\omega'|,
$$
 which gives
$$
	\mathcal{A} (r, z_0) \le N (d, \alpha, p).
$$
If
$$
	|\omega'| < 2|\omega|,
$$
we decompose $\omega' = \omega_{\perp} '+ \omega_{||}'$, where $\omega_{\perp}'$ is perpendicular to $\omega$, and $\omega_{||}'$ is parallel to $\omega$.
Then, for some $\lambda \in (-2, 2)$,
$$
	|\lambda + t|^2 \, |\omega|^2 \le 	|\omega' + t \omega|^2 = |\omega_{\perp}'|^2 + |\lambda + t|^2 \,  |\omega|^2  \le |\omega|^2 (4+  |\lambda + t|^2).
$$
Then, in the case $\alpha \in [0, p-1)$,
$$
	\mathcal{A} (r, z_0) \le N (\alpha, d)  \bigg(\int_0^1  |\lambda + t |^{-\alpha/(p-1)} \, dt\big)^{p-1} \le N (\alpha,  d, p).
$$
The case $\alpha \in (-1, 0)$ is handled in the same way.

\textit{Case B.} Observe that  for any $t \in (0, 1)$,  $x \in B_{c^3} (X_0 + t V_0)$,
$$
	c^{-1} |x  - (t - \tau) V_0|^{1/3} \le c^{-1} |x - X_0 - t V_0|^{1/3} + c^{-1} |X_0 + \tau V_0|^{1/3} \le 3,
$$
and then,
\begin{align*}
&	\{ (t, x): t \in (0, 1),  x \in B_{c^3} (X_0 + t V_0)\} \\
&\subset 	\{ (t, x): t  \in (\tau-1, \tau+1),   B_{(3 c)^3} ( (t - \tau) V_0)\}.
\end{align*}
Thus, we may assume $X_0 = 0$, in addition, by shifting in the time variable, we may assume $\tau = 0$.

Next, if $|V_0| \le 3  c^3$, then for any $t \in (-1, 1)$, we have
$$
	B_{ c^3} (t V_0) \subset B_{4 c^3} (0),
$$
 which implies \eqref{eq3.2.1} in view of Remark \ref{rem2.1}.
We now consider the case $|V_0| > 3 c^3$. We denote
$$
	\kappa = \frac{2 c^3}{|V_0|}
$$
 and note that for $|t| \le \kappa$,
$$
	B_{c^3} (t V_0) \subset B_{ 3 c^3} (0),
$$
while for $\kappa < |t| < 1$ and $x \in B_{c^3} (t V_0)$, one has
$$
  	  \frac t 2 |V_0| \le 	|x| \le 2 t |V_0|.
$$
Then, in the case  $\alpha \in [0, p-1)$,
\begin{align*}
	&\int_{-1}^1 \fint_{ B_{c^3} (t V_0) } |x|^{\alpha} \, dx dt \le  N (d) \int_{-\kappa}^\kappa \fint_{ B_{2 c^3} (0)}  |x|^{\alpha} \, dxdt\\
&\qquad + (2V_0)^\alpha \int_{-1}^1 |t|^{  \alpha }  \, dt \le N (d, \alpha)|V_0|^\alpha,\\
& \int_{-1}^1 \fint_{ B_{c^3} (t V_0) } |x|^{-\alpha/(p-1)} \, dx dt
\le   N (d) \int_{-\kappa}^\kappa \fint_{ B_{2 c^3} (0)}  |x|^{-\alpha/(p-1)} \, dx\\
&\qquad + (|V_0|/2)^{-\alpha/(p-1)}\int_{-1}^1 |t|^{-\alpha/(p-1)}\,dt
 \le N (d, \alpha, p) |V_0|^{-\alpha/(p-1)},
\end{align*}
and, thus, the estimate \eqref{eq3.2.1} is valid.  The case $\alpha \in  (-1, 0)$ is handled in the same way. The lemma is proved.
\end{proof}

\begin{corollary}
Let $p, r_1, \ldots, r_d > 1$, $c \ge 1$,  $\alpha \in (-1, p-1)$  be numbers, $T \in (-\infty, \infty]$, and $w_j$, $j = 1, \ldots, d,$ be weights satisfying \eqref{eq2.7}.
Then,
$$
	\|\bM_{c, T} f\| \le N \|f\|,\quad    \|f\| \le N \|f^{\#}_{c, T}\|,
$$
where
\begin{align*}
	&\|\cdot\| = \|\cdot\|_{   L_{p;  r_1, \ldots, r_d} (\bR^{1+2d}_T,  |x|^{\alpha} \prod_{i = 1}^d w_i (v_i))},\\
	& N  = N (d, p, r_1, \ldots, r_d, K, \alpha),	
\end{align*}
and $L_{p; r_1, \ldots, r_d} (\bR^{1+2d}_T,  |x|^{\alpha} \prod_{i = 1}^d w_i (v_i))$ is defined in \eqref{eq2.5}.
\end{corollary}

\begin{proof}
First, we claim that for any $\omega  \in A_p (\bR^d)$, the function $|x|^{\alpha} \omega (v)$ is an $A_p$ weight on the space of homogeneous type $(\overline{\bR^{1+2d}_T}, \rho_c, dz)$.
This follows from Lemma \ref{lemma 3.2} and the fact that in the $A_p$ condition for $\omega (v) |x|^{\alpha}$ (cf. \eqref{eq3.2.1}),  the integral is factored into a product of the integral over $t, x$ and the integral over $v$.
Then, by Corollary \ref{corollary 3.3},
\begin{align*}
	& \int_{\bR^{1+2d}_T} |\bM_{c, T} f|^p |x|^{\alpha} \omega (v) \, dz \le N \int_{\bR^{1+2d}_T} |f|^p  |x|^{\alpha} \omega (v) \,   dz,  \\
       & \int_{\bR^{1+2d}_T} |f|^p |x|^{\alpha} \omega (v) \, dz \le N \int_{\bR^{1+2d}_T} |f^{\#}_{c, T}|^p  |x|^{\alpha} \omega (v) \, dz.
\end{align*}
The assertion now follows from a variant of the Rubio de Francia extrapolation theorem (see, for example, Theorem 7.11 of \cite{DK19} and also \cite{Kr_20}).
\end{proof}

\begin{lemma}
			\label{lemma 4.1}
Let $p\in [1,\infty]$, $T \in (-\infty, \infty]$, and
$u \in S_{p, \text{loc} } (\bR^{1+2d}_T)$.
For any $z_0 \in \bR^{1+2d}_T$, denote
$$
	\widetilde z = (r^2 t + t_0, r^3 x + x_0 -  r^2 t v_0, r v + v_0),\quad
	\widetilde u (z) = u (\widetilde z),
$$
$$
    Y = \partial_t  - v \cdot D_x,
\quad    \widetilde P = \partial_t
     - v \cdot D_x
     - a^{ij} (\widetilde z) D_{   v_i v_j }.
 $$
Then,
$$
    Y \widetilde u (z) = r^2 Y u (\widetilde z),
\quad		\widetilde P \widetilde u (z) = r^2 (P u) (\widetilde z).
$$
\end{lemma}
\begin{proof}
The  assertion is a direct consequence of the following calculations:
\begin{align*}
	\partial_t \widetilde u (z) &=  r^2 (\partial_t u) (\widetilde z) - r^2 v_0 \cdot (D_x u) (\widetilde z),\\
	v \cdot D_x \widetilde u (z) &= r^3 v \cdot (D_x u) (\widetilde z)
	= r^2 (r v + v_0) \cdot (D_x u) (\widetilde z) - r^2 v_0 \cdot (D_x u) (\widetilde z),\\
	D_{    v_i v_j  } \widetilde u (z) &= r^2 (D_{   v_i v_j  } u) (\widetilde z).
\end{align*}
The lemma is proved.
\end{proof}

\section{\texorpdfstring{$S_2$}{}-estimate for the model equation}
						\label{section 10}
The goal of this section is to prove Theorem \ref{theorem 4.1}, which is
a version of Theorem \ref{theorem 2.1} in the case when $p = 2$, $b \equiv 0$, $c \equiv 0$, $w=1$, and the coefficients $a^{ij}$
are independent of $x$ and $v$.
Here is the outline of the proof.
First, we prove the a priori estimate for  smooth  and compactly supported functions in the $x, v$ variables by taking the Fourier transform in the $x, v$ variables
and reducing the equation to a first-order PDE. Furthermore, we use a limiting argument to extend the result to the space $S_2 (\bR^{1+2d}_T)$.
Then we prove the denseness of $(P_0 + \lambda)C^{\infty}_0 (\bR^{1+2d})$ in $L_2 (\bR^{1+2d})$ by using a localized version of the $L_2$-estimate.
The above ingredients yield the desired existence and uniqueness result.
\begin{theorem}
			\label{theorem 4.1}
The following assertions hold.

$(i)$
For any number $\lambda \ge 0$,
$T \in (-\infty, \infty]$, and $u \in S_2(\bR^{1+2d}_T)$,
one has
\begin{equation}
                \label{4.1.0}
    \begin{aligned}
&	\lambda \|u \| + \lambda^{1/2}  \|D_v u\|+\| D^2_v u \| +  \|(- \Delta_x)^{1/3} u \|
    \\
&	\quad + \|D_v (-\Delta_x)^{1/6} u\| \leq N (d) \delta^{-1} \| P_0 u    + \lambda  u \|,
\end{aligned}	
\end{equation}
where $\|\cdot\|=\|\cdot\|_{ L_2  (\bR^{1+2d}_T) }$.

$(ii)$ For any $\lambda > 0$, $T \in (-\infty, \infty]$, and
$f \in  L_2 (\bR^{1+2d}_T)$,  Eq. \eqref{2.3} has a unique solution
$u \in S_2 (\bR^{1+2d}_T)$.

$(iii)$
For any finite numbers $S < T$ and $f \in L_2 ((S, T) \times \bR^{2d})$,
the Cauchy problem \eqref{2.1} with $P = P_0$, $b \equiv 0$, and $c \equiv 0$ has a unique solution
$u \in
S_2((S, T) \times \bR^{2d})$.
In addition,
\begin{align*}
&    \|u\|_{ L_2 ( (S, T) \times \bR^{2d}) }
		 +\|D_v u\|_{ L_2 ((S, T) \times \bR^{2d}) }
	    +    \|D^2_v u\|_{ L_2 ((S, T) \times \bR^{2d}) }\\
 &\quad
   +\|(-\Delta_x)^{1/3}  u\|_{ L_2 ((S, T) \times \bR^{2d}) }
   +\|D_v (-\Delta_x)^{1/6}  u\|_{ L_2 ((S, T) \times \bR^{2d}) }\\
& \quad + \|\partial_t u - v \cdot D_x u\|_{ L_2 ((S, T) \times \bR^{2d}) }  \leq N (d, T-S) \delta^{-1} \|f\|_{ L_2 ((S, T)\times \bR^{2d}) }.
\end{align*}
\end{theorem}
						
\begin{lemma}
			\label{lemma 10.2}
Let $\lambda > 0$ be finite,
$T \in (-\infty, \infty]$,
and
$h\in C_b (\bR^{1+2d}_T)$, and
$f \in L_{\infty} ((-\infty, T), C_b (\bR^{2d})) \cap L_2 (\bR^{1+2d}_T)$
 be functions satisfying $D_{\xi} h \in C_b (\bR^{1+2d})$, $\partial_t h\in L_{\infty} ((-\infty, T), C_b (\bR^{2d})) \cap L_2 (\bR^{1+2d}_T)$, and
the equation
\begin{equation*}
	 \partial_t h + a^{i j} (t) \xi_i \xi_j h + k_i D_{\xi_i} h + \lambda h  = f.
\end{equation*}
Then,
one has
\begin{align*}
&	 \lambda \|h\|_{L_2 (\bR^{1+2d}_T) } +  \||\xi|^2  h\|_{ L_2 (\bR^{1+2d}_T) } +  \| |k|^{2/3} h \|_{L_2 (\bR^{1+2d}_T) }\\
&\quad	+ \||k|^{1/3} \xi h\|_{L_2 (\bR^{1+2d}_T) }	
	\leq N  (d) \delta^{-1}  \|f\|_{ L_2 (\bR^{1+2d}_T) }.
\end{align*}
\end{lemma}

\begin{proof}
In this proof, we take $N = N (d)$.
By the method of characteristics, we have
\begin{align*}
	&h (t, k, \xi)\notag\\
 &=  \int_{-\infty}^t e^{-\lambda (t- t')} e^{ - \int_{t'}^t a^{j l} (s)  (k_j (s - t) + \xi_j) (k_l (s - t) + \xi_l) \, ds     } f (t', k, k (t' - t) + \xi) \,  dt'.
\end{align*}
Note that by the parabolicity condition and the Minkowski inequality,
$$
	\|h (t, \cdot, \cdot)\|_{ L_2 (\bR^{2d}) } \leq  \int_{-\infty}^t e^{-\lambda (t- t')}  \|f (t', \cdot, \cdot)\|_{ L_2 (\bR^{2d}) } \,  dt'.
$$
Then, by Young's inequality, we get
$$
	\lambda \|h\|_{  L_2 (\bR^{1+2d}_T) } \leq  \|f\|_{  L_2 (\bR^{1+2d}_T) }.
$$
 Furthermore, the parabolicity  condition gives
\begin{equation}
			\label{10.2.4}
\begin{aligned}
    & \int_{t'}^t a^{j l} (s)  (k_j (s - t) + \xi_j) (k_l (s - t) + \xi_l) \, ds
    \ge \delta \int_{t'}^t   |k (s - t) + \xi|^2  \, ds\\
 	& =  \delta (t-t') (|k|^2 (t-t')^2/3 -  k \cdot \xi (t-t') + |\xi|^2)\\
 	&\ge  (\delta/24)  (t-t') (|k|^2 (t-t')^2 + |\xi|^2).
\end{aligned}
\end{equation}
By  this and the Minkowski inequality, we get
\begin{align*}
&	   \||k|^{2/3}   h (t, k, \cdot)\|_{ L_2 (\bR^{d}) }\\
&	\leq \int_{-\infty}^t |k|^{2/3}  e^{ - (\delta/24)  (t-t')^3 |k|^2    }  \|f (t', k, \cdot)\|_{ L_2 (\bR^d) } \, dt'.
\end{align*}
Furthermore, by Young's inequality, and the change of variables  $t=\delta s^3 |k|^2$,
\begin{align*}
&\||k|^{2/3}   h (\cdot, k, \cdot) \|_{ L_2 (\bR^{1+d}_T) }\\
& 	\leq   |k|^{2/3}  \big(\int_0^{\infty}   e^{ - (\delta/24) s^3 |k|^2    }  \, ds\big)  \|f (\cdot, k, \cdot)\|_{ L_2 (\bR^{1+d}_T) }\\
&	\leq  N \delta^{-1/3} \big(\int_0^{\infty} t^{-2/3}  e^{-t/24}  \, dt\big)  \|f (\cdot, k, \cdot)\|_{ L_2 (\bR^{1+d}_T) }.
\end{align*}
Integrating the above inequality over $k \in \bR^d$, we prove the estimate of $|k|^{2/3} h$.

Next, by the Cauchy-Schwartz inequality and \eqref{10.2.4},
$$
 		\||\xi|^2   h (t, \cdot, \cdot)\|^2_{ L_2 (\bR^{2d}_T) }
	\leq   \int_{\bR^{1+2d}_T}   I_1 (z) I_2 (z)\, dz,
 $$
where
\begin{align*}
	I_1 (z) & =	  \int_{-\infty}^t |\xi|^2   e^{ -  (\delta/24)(t-t') |\xi|^2   } \, dt'  \leq N \delta^{-1},\\
	I_2 (z) &= \int_{-\infty}^t |\xi|^2   e^{ - (\delta/24) (t-t') (|k|^2 (t-t')^2 + |\xi|^2)  }  f^2 (t', k,  k (t'-t) + \xi) \, dt'.
\end{align*}
Furthermore, by the change of variables $\xi \to k (t'-t) + \xi$
and the Fubini theorem, we get
  \begin{align*}
	&\int_{\bR^{1+2d}_T} I_2 (z) \, dz \\
&\leq  \int_{\bR^{1+2d}_T} \int_{-\infty}^t |\xi - k (t'-t)|^2   e^{ - \frac \delta {72} (t-t') (|k|^2 (t-t')^2 + |\xi|^2)  } f^2 (t', k , \xi) \, dt' \, dz\\
&	\leq 2 \int_{\bR^{2d}}  \bigg(\int_0^{\infty}    (|\xi|^2 + |k|^2 t^2)   e^{ - \frac \delta {72}  (|k|^2 t^3  + |\xi|^2 t)  } \, dt\bigg)   \bigg(\int_{-\infty}^T f^2 (t, k, \xi) \, dt\bigg) \,dkd\xi\\
&	\leq N \delta^{-1} \int_{\bR^{1+2d}_T} f^2 (z) \, dz.
  \end{align*}
Finally, the  estimate of $|k|^{1/3} \xi h (z)$ follows from the estimates of
$|k|^{2/3} h (z)$
and
$|\xi|^2 h (z)$, and the Cauchy-Schwartz inequality.
\end{proof}

\begin{corollary}
			\label{corollary 10.4}
Let $u (z)$ be a function such that
\begin{itemize}
\item[--] there exists some $r > 0$ such that for any $t \in (-\infty, T)$, $u (t, \cdot, \cdot)$ is compactly supported on $B_r \times B_r$,

\item[--]  $u$, $\partial_t u$, $D_x u$, $D_v u$, $D^2_v u \in C_b (\bR^{1+2d}_T) \cap L_2 (\bR^{1+2d}_T)$.

\end{itemize}
Then, \eqref{4.1.0} holds.

\end{corollary}

\begin{proof}
Let $f = P_0 u + \lambda u$.
For a smooth integrable function $\zeta (x, v)$, by $\hat \zeta (k, \xi)$, we denote its Fourier
transform in the $x, v$ variables.
Then, one has
$$
\partial_t \hat u + a^{i j} (t) \xi_i \xi_j \hat u + k_i D_{\xi_i} \hat u + \lambda \hat u  = \hat f.
$$
Combining Lemma \ref{lemma 10.2} with Parseval's identity, we prove the assertion.
\end{proof}

To generalize the estimate \eqref{4.1.0} for the class of functions $S_2 (\bR^{1+2d}_T)$, we prove the following approximation result.

\begin{lemma}
            \label{lemma 10.3}
For any $u \in S_2 (\bR^{1+2d}_T)$, there exists a sequence of functions
$u_n, n \ge 1,$ such that
\begin{enumerate}[a)]
\item
for any $t \in (-\infty, T)$, $u_n (t, \cdot)$ is compactly supported on $B_{2 n^{3}} \times B_{2 n}$,
\item
for any  $j,k, l \in \{0, 1, 2, \ldots\}$,
$\partial_t^j D_x^k D_v^l u_n \in C_b (\bR^{1+2d}_T) \cap L_2 (\bR^{1+2d}_T)$,

\item
$
  \lim_{n \to \infty}  \|u_n - u\|_{ S_2 (\bR^{1+2d}_T) } =0.
$
\end{enumerate}
\end{lemma}

\begin{proof}
Let $\eta \in C^{\infty}_0 ((0, \infty)  \times \bR^{2d})$
be a function with unit integral.
For $\varepsilon > 0$ and $h \in L_{1, \text{loc}} (\bR^{1+2d})$,
we denote
\begin{equation*}
	h_{(\varepsilon)} (t, x, v) =  \int  h(t-\varepsilon^2 t',x-\varepsilon^{1/2}x',v-\varepsilon v')\eta(t', x', v') \, dx'dv'dt'.
\end{equation*}
Furthermore, let $\phi \in C^{\infty}_0 (B_2\times B_2)$ be a function such that $\phi = 1$ on $B_1\times B_1$ and denote
$$
	\phi_r (z) = \phi (x/r^3, v/r),
	 \quad u_{\varepsilon, r}(z) =  u_{(\varepsilon)}  (z)  \phi_r  (x, v),
\quad u_n =  u_{n^{-1}, n},
$$
so that $a)$ and $b)$ are clearly satisfied.

We now prove $c)$. Clearly,  $Au_{n} \to Au$ for $A = I, D_v, D^2_v$ in $L_2 (\bR^{1+2d}_T)$ as $n \to \infty$. To prove the convergence of the transport term, we
first, note that for any $z \in \bR^{1+2d}_T$,
\begin{equation}
                \label{10.3.1}
\begin{aligned}
&	(\partial_t  - v \cdot D_x) (u_{(n^{-1})} - u)
= h_n  + g_n,
\end{aligned}
\end{equation}
where
$$
	h_n (z) =  (\partial_t u  -  v \cdot D_x u)_{(n^{-1})} (z) -  (\partial_t u -  v \cdot D_x u) (z)
$$
and
\begin{equation*}
	g_n (z) = - n^{-1/2} \int u (t-n^{-2}t', x - n^{-1/2}x', v - n^{-1} v') v' \cdot D_x \eta (t',x', v') \, dx'dv'dt'.
\end{equation*}
Furthermore, $h_n \to 0$ in $L_2 (\bR^{1+2d}_T)$
as $n \to \infty$ and by the Minkowski inequality,
\begin{equation}
                \label{10.3.2}
	\|g_n\|_{ L_2 (\bR^{1+2d}_T) } \le N  n^{-1/2}  \|u\|_{ L_2 (\bR^{1+2d}_T) }.
\end{equation}
Therefore, from \eqref{10.3.1} we obtain
 $
	\|(\partial_t  - v \cdot D_x) (u_{(n^{-1})}  - u)\|_{    L_2 (\bR^{1+2d}_T) } \to 0
 $
as $n \to \infty$.
To prove the desired convergence, we write
\begin{align*}
&	(\partial_t  - v \cdot D_x) (u_{n} - u_{(n^{-1})}) = (\phi_n  -  1)  (\partial_t u - v \cdot D_x u)_{(n^{-1})}\\
&\quad 	+ (\phi_n - 1) g_n  - u_{(n^{-1})} v \cdot D_x \phi_n =: A_n + B_n +C_n.
\end{align*}
We have
$$
	\|C_n\|_{  L_2 (\bR^{1+2d}_T)  } \le N n^{-2} \|u\|_{ L_2 (\bR^{1+2d}_T)},
$$
and by  \eqref{10.3.2},
 $$
	\|B_n\|_{  L_2 (\bR^{1+2d}_T)  } \le N    n^{-1/2}  \|u\|_{ L_2 (\bR^{1+2d}_T) }.
 $$
Furthermore, note that
$$
	|A_n|  \le N (d) |\phi_n  - 1| M_t  M_x M_v  (|\partial_t u - v \cdot D_x u|),
$$
where
$M_t$, $M_x$, and $M_v$ are  the Hardy-Littlewood maximal function with  respect to the $t$, $x$, and $v$ variables.
Then, by the Hardy-Littlewood inequality and the dominated convergence theorem,
$$
    \|A_n\|_{ L_2 (\bR^{1+2d}_T) } \to 0
$$
as
$n \to \infty$.
Thus,  $(\partial_t  - v \cdot D_x) (u_{n} - u) \to 0$
 in $L_2 (\bR^{1+2d}_T)$ as $n \to \infty$.
The lemma is proved.

\end{proof}

\begin{proof}[Proof of Theorem \ref{theorem 4.1} (i)]
Let $u_n, n \ge 1,$ be a sequence from Lemma \ref{lemma 10.3}. Then, by  Corollary \ref{corollary 10.4},
\begin{equation}
			\label{eq4.1}
\begin{aligned}
	&\lambda \|u_n\|_{L_2 (\bR^{1+2d}_T)} + \lambda^{1/2} \|D_v u_n\|_{L_2 (\bR^{1+2d}_T)} + \|D^2_v u_n\|_{L_2 (\bR^{1+2d}_T)} \\
&\quad + \|(-\Delta_x)^{1/3} u_n\|_{L_2 (\bR^{1+2d}_T)} + \|D_v (-\Delta_x)^{1/6} u_n\|_{L_2 (\bR^{1+2d}_T)} \\
&\le N \delta^{-1} \|P_0 u_n\|_{L_2 (\bR^{1+2d}_T)} 
\end{aligned}
\end{equation}
Passing to the limit in the above inequality as $n \to \infty$ and using Lemma \ref{lemma 10.3}, we prove \eqref{4.1.0} for $u, D_v u$, and $D^2_v u$.

Next, we fix  any $\phi \in C^{\infty}_0 (\bR^{1+2d}_T)$.  Since $u_n \to u$ in $L_2 (\bR^{1+2d}_T)$, for $A = (-\Delta_x)^{1/3}$, $D_v (-\Delta_x)^{1/6}$, we have
$$
	\bigg|\int_{  \bR^{1+2d}_T } (A u) \phi \, dz\bigg|   \le \|\phi\|_{  L_2 (\bR^{1+2d}_T) }  \nlimsup_{n \to \infty}  \| A u_n\|_{  L_2 (\bR^{1+2d}_T)  }.
$$
Now the desired estimates for $(-\Delta_x)^{1/3} u$ and $D_v(-\Delta_x)^{1/6}$ follow from this, \eqref{eq4.1}, and Lemma \ref{lemma 10.3}.
\end{proof}

Next we prove the existence part by using a density argument. We first show localized $L_2$ estimates.
\begin{lemma}
					\label{lemma 4.2}
Let  $\lambda \geq 0$ and $r_1, r_2, R_1, R_2 > 0$
 be  numbers such that
$r_1 < r_2$, and $R_1 < R_2$.
Let $u \in S_{2,\text{loc}} (\bR^{1+2d}_0)$ and denote
$
	f = P_0 u + \lambda u.
$
Then, the following assertions hold.
\begin{equation}
			\label{4.2.0}
\begin{aligned}
	(i)& \,
	 \delta^{-2 }(r_2 - r_1)^{-1} \|D_v u\|_{ L_2 (Q_{r_1,  R_1}) } + 	\| D^2_v u\|_{ L_2 (Q_{r_1,  R_1}) }\\
&	\leq N (d) \delta^{-1}  \| f \|_{ L_2 (Q_{r_2,  R_2}) }
	+  N (d) \delta^{-4}((r_2 - r_1)^{-2} \\
&\quad + r_2 (R_2 - R_1)^{-3})  \|u\|_{ L_2 (Q_{r_2,   R_2}) }.
\end{aligned}	
\end{equation}

$(ii)$ Denote
$C_r = (-r^2, 0) \times \bR^d \times B_r$.
Then, we have
\begin{align*}
&\delta^{- 2 }(r_2-r_1)^{-1}	\|D_v u\|_{ L_2  (C_{r_1} )}
	+\| D^2_v  u\|_{ L_2  ( C_{r_1} )   }\\
&	\leq N (d) \delta^{-1} \| f \|_{ L_2   ( C_{r_2})    }
	+ N (d) \delta^{-  4 }	 (r_2 - r_1)^{-2} \|u\|_{ L_2 (C_{r_2} ) }.
\end{align*}
\end{lemma}

\begin{proof}
The method is standard and can be found in Lemma 2.4.4 of \cite{Kr_08}.
In this proof, $N$ is a constant depending only on $d$.

$(i)$ Let
$\zeta \in C^{\infty}_{\text{loc}} (\bR)$
be a function
such that
$\zeta = 0$ if $t \geq 1$, and $\zeta = 1$ if $t \leq 0$.
Denote $\hat r_0 = r_1, \hat R_0 = R_1$,
$$
	\hat r_n  = r_1 + (r_2 - r_1) \sum_{k = 1}^n 2^{-k},
\quad
	\hat R_n = R_1 + (R_2 - R_1)    \sum_{k = 1}^n 2^{-k},
$$
 $$
	\eta_n (t, v) = \zeta \big(2^{2(n+1)} (r_2 - r_1)^{-2} (-\hat r^2_{n} - t)\big)   \, \,
	   \zeta \big(2^{(n+1)} (r_2 - r_1)^{-1} (|v| - \hat r_{n})\big),
 $$
and
  $$
	\phi_n (z) = \eta_n (t, v) \, \zeta \big(2^{3(n+1)}  (R_2 - R_1)^{-3} (|x| -    \hat R^{3}_{n} )\big).
  $$
Note that $\eta_n$ and $\phi_n$ are smooth functions.

Denote
$$
	Q (n) = Q_{\hat r_n, \hat R_n}
$$
and observe that in $\bR^{1+2d}_0$, $\phi_n$ vanishes outside  $Q (n+1)$,
and $\phi_n = 1$ on $Q (n)$.

Furthermore, in $\bR^{1+2d}_0$, we have
$$
	(P_0 + \lambda) (u \phi_n)  = f \phi_n + u (P_0 \phi_n) - 2 (a D_v \phi_n) \cdot D_v u.
$$
Then by \eqref{4.1.0},
\begin{align*}
	&\| D^2_v u \|_{ L_2 (Q(n)) }  \leq  \| D^2_v (u \phi_n) \|_{ L_2 (Q (n+1)) }\leq N  \delta^{-1} \|f\|_{ L_2 (Q_{r_2, R_2}) }\notag\\
&	\quad + N \delta^{-1} (\delta^{-1}  2^{2n} (r_2 -r_1)^{-2}
+ 2^{3n} r_2 (R_2 - R_1)^{-3})  \|u \|_{L_2 (Q (n+1)) }\notag\\
&\quad	+ N \delta^{-2} 2^n (r_2 - r_1)^{-1}  \| D_v u \|_{ L_2 (Q (n+1)) }.
\end{align*}
By the interpolation inequality,
\begin{align*}
	&\| D^2 _v u \|_{ L_2 (Q(n)) }+ (1+ \delta^{- 2 }(r_2 - r_1)^{-1})  \| D_v u \|_{ L_2 (Q (n)) }\\
&	\leq N \delta^{-1}\|f\|_{ L_2 (Q_{r_2, R_2}) }	+ 	N \delta^{-4} (2^{2n} (r_2 - r_1)^{-2} \\
 &\quad	+ 2^{3n}r_2 (R_2 - R_1)^{-3})  \|u \|_{ L_2 (Q (n+1))  } + 2^{-4} \|D^2_v u\|_{  L_2 (Q (n+1)) }.
\end{align*}
We multiply the above inequality by
$2^{-4n}$ and   sum up with respect to $n =0, 1, 2, \ldots$.
We get
\begin{align*}
	&\| D^2_v u \|_{ L_2 (Q_{r_1, R_1}) } + \sum_{n = 1}^{\infty} 2^{-4n}  \| D^2_v u \|_{ L_2 (Q (n)) } + \delta^{  -2  } (r_2 - r_1)^{-1}  \| D_v u \|_{ L_2 (Q_{r_1,R_2}) }\\
	& \leq N \delta^{-1}  \| f \|_{ L_2 (Q_{r_2, R_2}) }  + N \delta^{-4 } ((r_2 - r_1)^{-2}+ r_2(R_2 - R_1)^{-3}) \| u \|_{  L_2  (Q_{r_2, R_2) } }\\
	&\quad +  \sum_{n = 1}^{\infty} 2^{-4n} \| D^2_v u \|_{ L_2 (Q (n)) }.
 \end{align*}
Canceling the same sum on both sides of the above inequality, we prove the lemma.

$(ii)$
To prove the claim
we substitute $R_2 =  2 R_1$ in
\eqref{4.2.0}
and
pass to the limit as $R_1 \to \infty$.
This assertion is proved.
\end{proof}

\begin{lemma}
            \label{lemma 10.1}
For any number $\lambda \ge 0$, the set
$(P_0 + \lambda)C^{\infty}_0 (\bR^{1+2d})$
is dense in $L_2 (\bR^{1+2d})$.
\end{lemma}

\begin{proof}
Proof by contradiction. If the claim does not hold, then there exists a  function
$u \in L_2 (\bR^{1+2d})$
that is not identically zero and
such that
for any $\zeta \in C^{\infty}_0 (\bR^{1+2d})$,
$$
	\int (P_0 \zeta + \lambda \zeta) u \, dz = 0.
$$
Hence,
$$
-\partial_t u   + v \cdot D_x u - a^{i j} (t) D_{v_i v_j} u+ \lambda u \equiv 0
$$
in the sense of distributions. Let $w(t,x,v)=u(-t,-x,v)$. We regularize $w$ by using a mollification argument from Lemma \ref{lemma 10.3}.
For $\varepsilon > 0$ and  a function $h \in L_{1, \text{loc}} (\bR^{1+2d})$, we denote
$$
		h_{(\varepsilon)} (z) =  \int h (t, x - \varepsilon^{1/2} x', v -  \varepsilon v') \eta (x', v') \, dx'dv',
$$
where $\eta \in C^{\infty}_0 (\bR^{2d})$ and $\int \eta \, dxdv= 1$.
Then, the function  $w_{(\varepsilon)}$ satisfies the equation
\begin{equation}
                \label{3.1.1}
	 \partial_t  w_{(\varepsilon)}  - v \cdot D_x  w_{(\varepsilon)}   - a^{i j} (t) D_{v_i v_j}  w_{(\varepsilon)}  + \lambda w_{(\varepsilon)}= g_{\varepsilon},
\end{equation}
where
$$
	g_{\varepsilon} = - \varepsilon^{1/2} \int u (t, x - \varepsilon^{1/2}x', v - \varepsilon v') v' \cdot D_x \eta (x', v') \, dx'dv',
$$
which satisfies
\begin{equation}
			\label{3.1.2}
	\|g_{\varepsilon}\|_{ L_2 (\bR^{1+2d}) } \le N \varepsilon^{1/2} \|u\|_{ L_2 (\bR^{1+2d}) }.
\end{equation}
Furthermore, by Eq. \eqref{3.1.1} and \eqref{3.1.2}, $\partial_t w_{(\varepsilon)} \in L_{2, \text{loc}}(\bR^{1+2d})$, and, hence,
$w_{(\varepsilon)} \in S_{2, \text{loc}} (\bR^{1+2d})$.
Then, by  Lemma \ref{lemma 4.2},  \eqref{3.1.1}, and \eqref{3.1.2}, for any $r > 0$,
\begin{equation}
    \label{3.1.3}
    \begin{aligned}
   & \| D_v w_{(\varepsilon)}  \|_{L_2 (Q_r ) }
    \le N (d, \delta) (r \| g_{\varepsilon} \|_{L_2 (Q_{2r}  ) } + r^{-1} \| w_{(\varepsilon)}  \|_{L_2 (Q_{2r} ) })\\
  &  \le N (\varepsilon^{1/2} r +  r^{-1}) \|w\|_{L_2 (\bR^{1+2d}) }.
     \end{aligned}
 \end{equation}
 Therefore, $D_v w \in L_{2, \text{loc}} (\bR^{1+2d}_0)$,
 and, passing to the limit as $\varepsilon \to 0$ in
 \eqref{3.1.3}, we get
 $$
    \|D_v   w \|_{L_2 ( Q_{r})} \le N r^{-1} \| w \|_{L_2 (\bR^{1+2d}) }.
 $$
 Finally, passing to the limit as $r \to \infty$, we conclude that $D_v w \equiv 0$, and hence,
$ w \equiv 0$  in $\bR^{1+2d}_0$. By a translation in the time coordinate, we see that $w$ and thus $u$ are identically equal to $0$ in $\bR^{1+2d}$, which gives a contradiction. The lemma is proved.
\end{proof}

\begin{proof}[Proof of Theorem \ref{theorem 4.1} (ii)]
First, we consider the case when $T = \infty$.
The assertion follows from the a priori estimate $(i)$
and the denseness of
$     (P_0 + \lambda)  C^{\infty}_0 (\bR^{1+2d})$
in $L_2 (\bR^{1+2d})$ (see Lemma \ref{lemma 10.1}).

When $T < \infty$, we note that the uniqueness
holds by the a priori estimate of the assertion $(i)$ of this theorem.
To prove the existence, note that
the equation
$$
    P_0 u + \lambda u = f 1_{t < T}
$$
has a unique solution  $\widetilde u \in S_2 (\bR^{1+2d})$.
 We conclude that $u :=\widetilde u 1_{t < T}$ is a $S_2 (\bR^{1+2d}_T)$  solution  of Eq. \eqref{2.3}.
\end{proof}

\section{\texorpdfstring{$S_p$}{}-estimate for the model equation}
                \label{section 4}
Here we generalize Theorem \ref{theorem 4.1} for $p \in (1, \infty)$.
We follow the argument  in Chapter 4 of \cite{Kr_08}.
To derive an estimate of the sharp functions of $(-\Delta_x)^{1/3} u$, $D_v (-\Delta_x)^{1/6} u$, and $D^2_v u$ for $u \in S_{p} (\bR^{1+2d}_T)$ (see Proposition \ref{theorem 4.7.2}),
we split $u$ into the $P_0$-caloric part and the remainder.
The latter is estimated in Lemma \ref{lem4.7} and the former - in  Proposition \ref{theorem 4.7}.
Throughout this section, the matrix-valued function $a$ is independent of $x, v$.

\begin{theorem}
			\label{theorem 4.10}
Let $p > 1$ be a number.
The following assertions hold.

$(i)$
For any number $\lambda \ge 0$ and $u \in S_p(\bR^{1+2d}_T)$,
one has
\begin{equation}
                \label{eq2.38}
 \begin{aligned}
&	\lambda \|u \|+\lambda^{1/2}\|D_v u\|+\|D^2_v u\|
+ \|(- \Delta_x)^{1/3} u \|
\\
&\quad + \|D_v (-\Delta_x)^{1/6} u\|	
\leq N (d, p) \delta^{-\theta} \| P_0 u   + \lambda  u \|,
\end{aligned}	
 \end{equation}
where $\theta = \theta (d) > 0$ and $\|\cdot\|=\|\cdot\|_{ L_p  (\bR^{1+2d}_T) }$.

$(ii)$
For any $\lambda > 0$,
$T \in (-\infty, \infty]$
and $f \in  L_p (\bR^{1+2d}_T)$,
Eq. \eqref{2.3} has a unique solution
$u \in S_p (\bR^{1+2d}_T)$.

$(iii)$ For any finite numbers $S < T$   and
  $  f \in L_p ((S, T) \times \bR^{2d})$,
the Cauchy problem
\eqref{2.1} with $P = P_0$, $b \equiv 0$, and $c \equiv 0 $
has a unique solution
$u \in
S_p((S, T) \times \bR^{2d})$.
In addition,
\begin{align*}
&    \|u\|_{ L_p ( (S, T) \times \bR^{2d}) }
+ \|D_v u\|_{ L_p ((S, T) \times \bR^{2d}) }
+  \|D^2_v u\|_{ L_p ((S, T) \times \bR^{2d}) }
\\
& \quad       +\|(-\Delta_x)^{1/3}  u\|_{ L_p ((S, T) \times \bR^{2d}) }    + \|D_v (-\Delta_x)^{1/6} u\|_{ L_p  (\bR^{1+2d}_T) }\\
& \leq N (d, p, T-S)  \delta^{-\theta} \|f\|_{ L_p ((S, T)\times \bR^{2d}) },
\end{align*}
where $\theta = \theta (d) > 0$.
\end{theorem}

We note that in the next lemma, $f$ is assumed to be compactly supported only in the $t$ and $v$ variables. The reason for such a choice will be clear when we estimate the $P_0$-caloric part. See the proof of Proposition \ref{theorem 4.7}.
\begin{lemma}
            \label{lem4.7}
Let $R\ge 1$
be a number and $f\in L_2(\bR^{1+2d})$ vanish outside $(-1,0)\times \bR^d\times B_1$. Let $u\in S_2((-1,0)\times \bR^{2d})$ be the unique solution to the Cauchy problem (see Definition \ref{definition 2.1} and Theorem \ref{theorem 4.1} (iii))
\begin{equation}
            \label{eq7.31}
	P_0 u= f,
	\quad
	u (-1, \cdot) = 0.
\end{equation}
Then, there exists a number $\theta  = \theta (d) > 0$ such that
\begin{align}
                \label{eq7.25}
&\||u|+|D_vu|+|D_v^2 u|\|_{L_2((-1,0)\times B_{R^3}\times B_R)}\notag\\
&\le N (d)  \delta^{-1} \sum_{k=0}^\infty  2^{-k(k-1)/4}R^{-k}
\|f\|_{L_2(Q_{1,2^{k+1}R/\delta^2})}
,\\
   \label{eq7.26}
&
(|(-\Delta_x)^{1/3} u|^2)^{1/2}_{Q_{1,R}}
\le N (d) \delta^{-\theta} R^{-2}\sum_{k=0}^\infty  2^{-2k}(f^2)^{1/2}_{Q_{1,2^{k}R/\delta^2}},
\end{align}
and
$$
    (|D_v (-\Delta_x)^{1/6} u|^2)^{1/2}_{Q_{1, R}}
   \leq N (d)  \delta^{-\theta} R^{-1} \sum_{k=0}^\infty
   2^{- k }(f^2)^{1/2}_{Q_{1, 2^k R/\delta^2}}.
 $$
\end{lemma}

\begin{proof}
It is possible to obtain the estimates of this lemma  by using the fast decay of
the fundamental solution of the operator $P_0$, which can be written down explicitly (see, for example, \cite{BCLP_13},\cite{CZ_19}).
Instead of invoking the integral representation of the solution to Eq. \eqref{eq7.31},
we decompose it into a sum of functions supported on dyadic shells and exploit the global $L_2$-estimate of Theorem \ref{theorem 4.1} and the scaling property of the operator $P_0$.

In this proof, $N$ is a constant depending only on $d$.

\textit{Estimate of $u$, $D_v u$, and $D^2_v u$.}
Denote
\begin{equation*}
f=f_0+\sum_{k=1}^\infty f_k:=f 1_{\{x\in B_{(2R/\delta^2)^3}\}}+\sum_{k=1}^\infty f 1_{\{x\in B_{(2^{k+1}R/\delta^2)^3}\setminus B_{(2^{k}R/\delta^2)^3}\}}.
\end{equation*}
By Theorem \ref{theorem 4.1}, for $k=0,1,2,\ldots$, there exists a unique  solution $u_k \in S_2((-1,0)\times \bR^{2d})$  to \eqref{eq7.31} with $f_k$ in place of $f$, which satisfies
\begin{equation}
    \label{eq7.41}
\||u_{k}|+|D_v u_{k}|+|D_v^2 u_{k}|\|_{L_2((-1,0)\times \bR^{2d})}
\le N \delta^{-1}\|f_k\|_{L_2((-1,0)\times \bR^{2d})}.
\end{equation}
By the same theorem, for $A = I$, $D_v$, and $D^2_v$, we have
$$
 \lim_{n \to \infty} \sum_{k=0}^n A u_k = A u \quad \text{in} \, \, L_2((-1,0)\times \bR^{2d}).
$$
Now we take a sequence of functions $\zeta_j= \zeta_j (x,v) \in C_0^\infty(B_{(2^{j+1} R/\delta^2)^3}\times B_{2^{j+1} R/\delta^2})$, $j=0,1,2,\ldots,$ such that $\zeta_j=1$ on $B_{(2^{j+1/2} R/\delta^2)^3}\times B_{2^{j+1/2} R/\delta^2}$ and
$$
|\zeta_j|\le 1,\quad |D_v \zeta_j|\le N 2^{-j}R^{-1}  \delta^2,
$$
$$
|D^2_v \zeta_j|\le N 2^{-2j}R^{-2}  \delta^{4},
\quad |D_x \zeta_j|\le N  2^{-3j}R^{-3} \delta^{6}.
$$
For $k\ge 1$ and $j=0,1,\ldots,k-1$, we set $u_{k,j}=u_k\zeta_j$, which satisfies
$$
P_0 u_{k,j}= u_k P_0 \zeta_j-2(aD_v \zeta_j)\cdot D_v u_k
$$
because $f_k\zeta_j\equiv 0$.
By Theorem \ref{theorem 4.1},
\begin{align*}
&\||u_{k,j}|+|D_v u_{k,j}|+|D_v^2 u_{k,j}|\|_{L_2((-1,0)\times \bR^{2d})}\notag\\
&\le N \delta^{-1}\||u_k P_0 \zeta_j|  + |a||D_v \zeta_j| |D_v u_k|\|_{L_2((-1,0)\times \bR^{2d})},
\end{align*}
which by the properties of $\zeta_j$ and the fact that $|a^{i j}| < \delta^{-1}$ implies that
\begin{align}
    \label{eq8.04}
&\||u_{k}|+|D_v u_{k}|+|D_v^2 u_{k}|\|_{L_2((-1,0)\times B_{(2^j R/\delta^2)^3}\times B_{2^j R/\delta^2})}\notag\\
&\le N2^{-j}R^{-1}\||u_{k}|+|D_v u_{k}|\|_{L_2((-1,0)\times B_{(2^{j+1} R/\delta^2)^3}\times B_{2^{j+1} R/\delta^2})}.
\end{align}
By  an induction argument, \eqref{eq7.41}, and \eqref{eq8.04},  for $k\ge 1$ we get
\begin{align}
    \label{eq8.11}
&\||u_{k}|+|D_v u_{k}|+|D_v^2 u_{k}|\|_{L_2((-1,0)\times B_{R^3}\times B_R)}\notag\\
&\le N^k 2^{-k(k-1)/2}R^{-k}\delta^{-1}\|f_k\|_{L_2((-1,0)\times \bR^{2d})}\notag\\
&\le N 2^{-k(k-1)/4}R^{-k}\delta^{-1}
\|f\|_{L_2(Q_{1,2^{k+1}R/\delta^2})}.
\end{align}
To conclude \eqref{eq7.25}, we  use \eqref{eq7.41} with $k=0$, \eqref{eq8.11}, and the triangle inequality.

\textit{Estimate of $(-\Delta_x)^{1/3} u$}.
Note that $u\zeta_0$ satisfies
$$
P_0 (u\zeta_0)=f\zeta_0 + u (P_0 \zeta_0)-2(a D_v\zeta_0)\cdot D_v u.
$$
By Theorem \ref{theorem 4.1} and \eqref{eq7.25},
\begin{align}
		\label{eq11.58}
&\|(-\Delta_x)^{1/3} (u\zeta_0)\|_{L_2((-1,0)\times \bR^{2d})}
+ \|D_v (-\Delta_x)^{1/6} (u\zeta_0)\|_{L_2((-1,0)\times \bR^{2d})}\\
&\le N \delta^{-1}\sum_{k=0}^\infty 2^{-k(k-1)/4 }R^{-k}
\|f\|_{L_2(Q_{1,2^{k+1}R/\delta^2})}.\notag
\end{align}
It remains to estimate the commutator term
\begin{equation}
    \label{eq11.56}
\|(-\Delta_x)^{1/3} (u\zeta_0)-\zeta_0(-\Delta_x)^{1/3} u\|_{L_2(Q_{1,R})}.
\end{equation}
Since $\zeta_0 = 1$ in $B_{(2^{1/2}R/\delta^2)^3}\times B_{2^{1/2}R/\delta^2}$,
for any $z \in Q_{1,R}$,
we have
\begin{align*}
&    |(-\Delta_x)^{1/3} (u\zeta_0)-\zeta_0(-\Delta_x)^{1/3}u|(z)\\
&   \le N  \int_{ |y| >  (2^{3/2}-1)R^3   } |u| (t, x+y, v) |y|^{-d-2/3} \, dy.
\end{align*}
By Lemma \ref{lemma 4.7},
\eqref{eq11.56} is bounded by
$$
NR^{  - 2 } \sum_{j=0}^\infty 2^{-2j-3dj/2} \|u\|_{ L_2(Q_{1,2^jR})}.
$$
By \eqref{eq7.25}, the above sum is further bounded by
$$
  N \delta^{-1} R^{-2}\sum_{j=0}^\infty 2^{-2j-3dj/2}\sum_{k=0}^\infty 2^{-k(k-1)/4}(2^jR)^{-k}  \|f\|_{L_2(Q_{1,2^{k+j+1}R/\delta^2 })},
$$
which gives \eqref{eq7.26} by a straightforward computation with a change of order of summations.

\textit{Estimate of $D_v (-\Delta_x)^{1/6} u$.}
By \eqref{eq11.58}, it suffices to estimate the $L_2(Q_{1,R})$ norm
of
$$
   A = (-\Delta_x)^{1/6} D_v (u\zeta_0)-\zeta_0(-\Delta_x)^{1/6} D_v u,
$$
which is bounded above by
$$
   A_1 + A_2: =   |(-\Delta_x)^{1/6} (u D_v \zeta_0)|
+|(-\Delta_x)^{1/6} (\zeta_0 D_v u)
      - \zeta_0 (-\Delta_x)^{1/6} D_v u|.
$$
For any $z \in Q_{1,R}$,
$$
   A_1   (z)
    \le
    N\delta^2  \int_{ |y| > (2^{3/2}-1) R^3   } |u| (t, x+y, v) |y|^{-d-1/3} \, dy.
$$
Then, arguing as above, we get
 \begin{equation}
                \label{eq11.57}
    \|A_1\|_{L_2(Q_{1,R})}
    \le  N (d)  R^{-1 } \delta \sum_{k=0}^\infty 2^{-k-3dk/2} \|f\|_{L_2(Q_{1,2^{k}R/\delta^2  })}.
 \end{equation}
Furthermore, by  Lemma \ref{lemma 4.7},
\begin{align*}
&\|(-\Delta_x)^{1/6} (\zeta_0 D_v u)
      - \zeta_0 (-\Delta_x)^{1/6} D_v u\|_{L_2(Q_{1,R})}\\
&\le
N R^{-1 }\sum_{j=0}^\infty 2^{-j-3dj/2}\|D_v u\|_{L_2(Q_{1,2^jR})}.
\end{align*}
By a computation similar to the one for \eqref{eq11.56}, we get
$$
    \|A_2\|_{L_2(Q_{1,R})}
    \le   N (d) R^{-1  } \delta^{-1} \sum_{k=0}^\infty 2^{-k-3dk/2}\|f\|_{L_2(Q_{1,2^{k}R/\delta^2})}.
$$
Combining this with \eqref{eq11.57},
we prove the desired estimate.
\end{proof}

\begin{proposition}
			\label{theorem 4.7}
Let $r > 0$, $\nu \geq 2$ be numbers,
$T \in (-\infty, \infty]$,
 $z_0 \in \overline{\bR^{1+2d}_T}$,
and
$u \in S_{2,\text{loc}} (\bR^{1+2d}_T) $.
 Assume that
$P_0 u= 0$ in
$(t_0-\nu^2 r^2, t_0) \times \bR^d \times B_{\nu r} (v_0)$.
Then, there exist constants $N (d)> 0$ and $\theta = \theta (d) > 0$ such that
\begin{align*}
    J_1 &:=
     \bigg(|(-\Delta_x)^{\frac 1 3} u -  ((-\Delta_x)^{\frac 1 3}  u )_{ Q_r (z_0) }|^2\bigg)^{1/2}_{Q_{r} (z_0)}\\
&\leq N  \nu^{-1} \delta^{-\theta}
  (|(-\Delta_x)^{\frac 1 3} u|^2)^{1/2}_{ Q_{\nu r } (z_0)},
\end{align*}
 \begin{align*}
    J_2  &:=  \bigg(|D_v (-\Delta_x)^{\frac 1 6} u -  (D_v(-\Delta_x)^{\frac 1 6} u)_{ Q_r (z_0) }|^2\bigg)^{1/2}_{ Q_r (z_0) }\\
&\leq
   N \nu^{-1}  \delta^{-\theta} (|D_v (-\Delta_x)^{\frac 1 6} u|^2)^{1/2}_{Q_{\nu r} (z_0)}\\
&
\quad +  N \nu^{-1} \delta^{-\theta} \sum_{k = 0}^{\infty} 2^{-2 k }
     (|(-\Delta_x)^{\frac 1 3} u|^2)^{1/2}_{ Q_{\nu r, 2^k \nu r}(z_0) },
  \end{align*}
 \begin{align*}
 J_3 &:=
   \bigg(|D^2_v u -  (D^2_v u)_{ Q_r (z_0) }|^2\bigg)^{1/2}_{ Q_r (z_0) }\\
&   \leq N \nu^{-1} \delta^{-\theta} (|D^2_v u|^2)^{1/2}_{ Q_{\nu r} (z_0)  }
    +   N \nu^{-1}  \delta^{-\theta} \sum_{k = 0}^{\infty} 2^{- k }
     (|(-\Delta_x)^{\frac 1 3} u|^2)^{1/2}_{ Q_{\nu r, 2^k \nu r}(z_0) }.
 \end{align*}
\end{proposition}

\subsection{Proof of Proposition \ref{theorem 4.7}}
            \label{section 5}

We follow the scheme of Chapter 4 of \cite{Kr_08}.
Thanks to the Poincar\'e inequality, to handle the $P_0$-caloric function, it suffices to estimate its H\"older norm. We do this by using Caccioppoli type estimate combined with the localized version of Theorem \ref{theorem 4.1}. See Lemma \ref{lemma 4.2}.
The proof of Proposition \ref{theorem 4.7} is given at the end of this subsection
after a series of lemmas.

The following lemma contains one of the key estimates of the proof.

\begin{lemma}[Caccioppoli type estimate]
			\label{lemma 4.3}
Let
$u \in S_{2,\text{loc}} (\bR^{1+2d}_0)
$
be a function
such that
 $$
	P_0 u  = 0  \quad \text{in} \,\,  Q_1.
 $$
Then, for any numbers $0 < r < R \leq 1$,
 $$
	 \| D_x  u \|_{ L_2 (Q_r) } 	 \leq N (d, r, R) \delta^{- 7  }   \| u \|_{ L_2 (Q_R)  } .
 $$
\end{lemma}

\begin{proof}
By modifying $u$ outside $Q_1$, we may certainly assume that $u$ is compactly supported. By taking the standard mollification with respect to $x$ and then taking the limit, we may assume without loss of generality that $(-\Delta_x)^\beta u\in S_{2}(\bR^{1+2d}_0)$ for any $\beta\ge 0$.
Let
$r < r_1 < r_2 < R$ be  numbers,
$\psi  \in C^{\infty}_0 (\bR^{d})$  vanish outside $B_1$     and $\eta \in C^{\infty}_0 (\bR^{1+d})$
vanish outside $(-1, 1) \times B_1$.
We set
$\phi_1 (t, v) = \eta (t/r_1^2, v/r_1),$
$\phi_2 (x) = \psi (x/r_2^3)$
and denote
$$
	\phi (z) = \phi_1 (t,v) \phi_2 (x).
$$
Throughout the proof, a constant $N$ depends only on $d$, $r$, and $R$.

Note that $u\phi$ satisfies the equation
\begin{equation}
			\label{4.3.1}
	P_0 (u \phi)     = u  P_0 \phi - 2  (a D_v \phi) \cdot D_v u
\end{equation}
on $\bR^{1+2d}_0$  because $\phi P_0 u \equiv  0$.
Then,  by Theorem \ref{theorem 4.1} applied to \eqref{4.3.1}, one has
\begin{align*}
&	\| (-\Delta_x)^{1/3} (u \phi)\|_{ L_2 (\bR^{1+2d}_0) }\\
& 	\leq N  \delta^{-1}\| u P_0 \phi \|_{ L_2 (\bR^{1+2d}_0)  }
+ N \delta^{-1} \| (a D_v \phi)\cdot D_v u \|_{ L_2 (\bR^{1+2d}_0) }.
\end{align*}
Applying  Lemma \ref{lemma 4.2} $(i)$ with $r_1$ in place of $r_1$, $R_1 = r_2$, and with  $R$ in place of $r_2$ and $R_2$ and using the fact that $|a^{i j}| < \delta^{-1}$, from the above inequality we get
\begin{equation}
			\label{4.3.3}
	 \| \phi_1 (-\Delta_x)^{1/3} (u \phi_2) \|_{ L_2 (\bR^{1+2d}_0) } \leq N \delta^{- 4}\| u  \|_{ L_2 (Q_R) }.
\end{equation}

Furthermore, the function
$w  = (-\Delta_x)^{1/3} (u \phi)$
solves the equation
\begin{equation}
			\label{4.3.2}
P_0 w   =     (-\Delta_x)^{1/3} [u  (P_0 \phi)] -   2   (a D_v \phi_1) \cdot  D_v  (-\Delta_x)^{1/3} (u \phi_2)
\end{equation}
on $\bR^{1+2d}_0$.
Due to Theorem \ref{theorem 4.1} $(i)$ applied to \eqref{4.3.2},
\begin{equation}
			\label{4.3.4}
	\begin{aligned}
		&\| (-\Delta_x)^{2/3} (u \phi) \|_{ L_2  (\bR^{1+2d}_0 ) }
\le
	N \delta^{-1} (\|(-\Delta_x)^{1/3} (u  P_0 \phi)\|_{ L_2  (\bR^{1+2d}_0) }\\
&\quad 	+\| (a D_v \phi_1)   \cdot   D_v (-\Delta_x)^{1/3}  (u \phi_2)\|_{ L_2   (\bR^{1+2d}_0) }) =: J_1 + J_2.
	\end{aligned}		
\end{equation}
By \eqref{4.3.3},
\begin{equation}
            \label{4.3.5}
    J_1  \le  N\delta^{- 6  } \| u  \|_{ L_2 (Q_R) }.
\end{equation}

Next we consider the term $J_2$.
Observe that
$f = (-\Delta_x)^{1/3} (u\phi_2)$ satisfies the equation
$$
	P_0 f = -  (-\Delta_x)^{1/3} [(v \cdot D_x \phi_2) u]
	\quad \text{on}\,\,  (-1, 0) \times \bR^d \times B_1.
$$
By  Lemma \ref{lemma 4.2} $(ii)$ and   the assumption $|a^{i j}| < \delta^{-1}$,
\begin{equation*}
\begin{aligned}
	J_2  \leq  N\delta^{-3} \| \eta  (-\Delta_x)^{1/3} (u\phi_2) \|_{ L_2  (\bR^{1+2d}_0) }
	+N \| \eta v_i  (-\Delta_x)^{1/3} (u  D_{x_i} \phi_2) \|_{ L_2  (\bR^{1+2d}_0) }.
\end{aligned}
\end{equation*}
Here
$	\eta  \in C^{\infty} (\bR^{1+d})$
is a function of $(t, v)$ such that
$\eta  = 1$ on $(-r_2^2, r_2^2) \times B_{r_2}$  and $\eta = 0$ outside
 $(-r^2_3, r^2_3) \times B_{r_3}$,
where $r_2 < r_3 < R$.
Then, using \eqref{4.3.3}, we get
\begin{equation}
            \label{4.3.6}
	J_2 \leq N \delta^{  -  7 }\| u \|_{ L_2 (Q_R) }.
\end{equation}

By \eqref{4.3.4} - \eqref{4.3.6},
we obtain
\begin{align*}
&	\| (1 - \Delta_x)^{2/3} (u \phi) \|_{ L_2 (\bR^{1+2d}_0) } \\
&\leq N (d) \| u \phi \|_{    L_2 (\bR^{1+2d}_0) } + N (d) \| (-\Delta_x)^{2/3} (u \phi) \|_{ L_2  (\bR^{1+2d}_0) }\\
&	\leq N \delta^{- 7 }  \| u \|_{ L_2 (Q_R) }.
\end{align*}
We assume that
$\phi_1  = 1$ on $(-r^2, r^2) \times B_{r}$ and
$\phi_2  = 1$ on  $B_{r^3}$.
We conclude the proof as follows:
\begin{align*}
&\|   D_x u  \|_{ L_2 (Q_r)  } \leq  \|  D_x (u \phi) \|_{ L_2 (\bR^{1+2d}_0) } \\
&    \leq  N (d) \| (1 - \Delta_x)^{2/3} (u \phi) \|_{ L_2 (\bR^{1+2d}_0) }
    \leq N \delta^{- 7 }  \| u \|_{ L_2 (Q_R)}.
\end{align*}
The lemma is proved.
\end{proof}

We also need the following nonlocal estimates.
 \begin{lemma}
			\label{lemma 4.6}
   Let  $r \in (0, 1)$ be a number and
  $
	u \in S_{2,\text{loc}}(\bR^{1+2d}_0).
  $
We denote
$f = P_0 u$
 and assume that $f = 0$
  in
  $(-1, 0) \times \bR^d \times B_1$.
   Then,
   \begin{equation}
				\label{4.6.0}
    (i)\, \, 	\| D_x u \|_{ L_2 (Q_r)} \leq
	 N \delta^{- 4  } \sum_{k = 0}^{\infty}
	  2^{-k}	  (|(-\Delta_x)^{1/3} u|^2)_{ Q_{1, 2^k}}^{ 1/2   },
   \end{equation}
   $$
   (ii)\, 	\| D_x u \|_{ L_2 (Q_r)} \leq
    N \delta^{- 4 } \sum_{k = 0}^{\infty}
	  2^{- 2 k }	  (|(-\Delta_x)^{1/6} u|^2)_{ Q_{1, 2^k}}^{ 1/2   },
   $$
where $N = N (d, r)$.
   \end{lemma}

\begin{proof}
$(i)$
By mollifying $u$ in the $x$ variable, we may assume that $u$ is sufficiently regular in $x$.
Fix some number $R \in (r, 1)$.
Let $\eta \in C^{\infty}_0 (\bR^{1+2d})$
be a function such that
$\eta = 1$ on $\widetilde{Q}_r$ and $\eta$ vanishes outside  $\widetilde{Q}_R$.
In this proof, we assume that $N = N (d, r, R)$.

We decompose
$$
	\eta^2 D_x u = \eta^2 \cR_x (-\Delta_x)^{1/2} u = \eta  (L u + \text{Comm}),
$$
where  $\cR_x = D_x (-\Delta_x)^{-1/2}$ is the Riesz transform,
\begin{align*}
	L u &= \cR_x  (-\Delta_x)^{1/6} (\eta (-\Delta_x)^{1/3} u),\\
	\text{Comm} &=  \eta \cR_x (-\Delta_x)^{1/2} u - \cR_x  (-\Delta_x)^{1/6} (\eta (-\Delta_x)^{1/3} u).
\end{align*}
Thus, to prove the claim we only need to show that \eqref{4.6.0} holds
if we replace the left-hand side  with
$
	\| L u\|_{ L_2 (Q_R) } + \|  \text{Comm} \|_{ L_2 (Q_R) }.
$

\textit{Estimate of $L u$}.
Denote
$$
	w = (-\Delta_x)^{1/3} u.
$$
Then, since $P_0 u = 0$
in $(-1, 0) \times \bR^d \times B_1,$
 the function $\eta w $ satisfies the equation
\begin{equation}
            \label{4.6.1}
P_0 (\eta w) =  w P_0 \eta - 2 (a D_v \eta) \cdot D_v w \quad \text{on} \, \bR^{1+2d}_0.
\end{equation}
By the fact that $\cR_x$ is an isometry in $L_2 (\bR^d)$,
the interpolation inequality, and Theorem \ref{theorem 4.1},  we have
\begin{align*}
	\| L u\|_{ L_2 (\bR^{1+2d}_0) }&\leq
  N(d)  \|\eta w \|_{ L_2 (\bR^{1+2d}_0) }
  + N(d) \|  (-\Delta_x)^{1/3} (\eta w)\|_{ L_2 (\bR^{1+2d}_0) }\\
&	\leq N  \delta^{- 2 } \|  w \|_{ L_2 (Q_R) } +  N  \delta^{-1} \|   (a D_v \eta) \cdot D_v w \|_{ L_2 (Q_R) }.
\end{align*}
By the fact that $P_0 w = 0$ in $(-1, 0) \times \bR^d \times B_1$, Lemma \ref{lemma 4.2} $(i)$, and Assumption \ref{assumption 2.1}, the last term is bounded by
$$
	N  (d, r, R) \delta^{- 4 }  \| w \|_{ L_2 (Q_{1}) }.
$$

\textit{Estimate of $\text{Comm}$.}
Denote
$$
	A= \cR_x (-\Delta_x)^{1/6} = D_x (-\Delta_x)^{-1/3}
$$ and observe that
$$
	\text{Comm} = \eta A w - A (\eta w).
$$
It is well known that for any Schwartz function $\phi$ on $\bR^d$,
$$
	(-\Delta_x)^{-1/3} \phi  = \cK \ast \phi,
$$
where
$$
	\cK (x) = N   \int_0^{\infty} t^{1/3}  p (t, x)  \frac{dt}{t}
$$
and
$$
	 p (t, x) = (4\pi t)^{-d/2} e^{-|x|^2/(4t)}.
$$
Furthermore, by the change of variables $t \to \frac{t}{|x|^2}$ for any $x \in \bR^d \setminus \{0\}$, we have
\begin{align*}
&    D_x \cK (x) = N x \int_0^{\infty} t^{- 2/3 - d/2}  e^{-|x|^2/(4t)}  \frac{dt}{t}\\
&     =  N x |x|^{  - (d + 4/3) } \int_0^{\infty} t^{-2/3 - d/2}  e^{-1/(4t)}  \frac{dt}{t}
    = N x |x|^{  - (d + 4/3) }.
\end{align*}
Therefore, for any $z \in \bR^{1+2d}_0$,
$$
	|\text{Comm} (z)| \leq  N \int_{\bR^d} |w (t, x - y, v)|  \ |\eta (t, x, v) -  \eta (t, x-y, v)| \, |y|^{-(d+1/3) } \, dy.
$$
 We split the above integral into two parts. The first part, $I_1$, is the integral over $|y| \leq 2$,
and the remainder is denoted by $I_2$.
First, by the mean-value theorem,
$$
	 |I_1 (z) | \leq N \int_{|y|   < 2}  | w (t, x  - y, v)| \, \,  |y|^{-d + 2/3} \, dy, \, \, z \in \bR^{1+2d}_0.
$$
Then by the Minkowski inequality,
\begin{align*}
&	\| I_1 \|_{ L_2 (Q_R) }
	\leq N  \int_{|y|   < 2}
	\| w (\cdot, \cdot - y, \cdot)\|_{  L_2 (Q_R)  } \, \,  |y|^{-d + 2/3} \, dy\\
&	\leq N  \| w \|_{ L_2 (Q_{1,2})  }  \int_{|y|   < 2}
	|y|^{-d + 2/3} \, dy
	\leq N 	\| w \|_{ L_2 (Q_{1,2})  }.
\end{align*}
Furthermore,
since $\eta$ vanishes outside  $Q_R$, for any $z \in Q_R$,
$$
		I_2  (z) = N |\eta (z)|   \int_{|y| > 2}   |w (t, x - y, v)|  \,  |y|^{-(d+1/3) } \, dy.
$$
By virtue of  Lemma \ref{lemma 4.7},
$$
	\| I_2 \|_{ L_2 (Q_R) }
	\leq N
	   \sum_{k = 0}^{\infty}
	   2^{-k}
	  (w^2)^{1/2}_{ Q_{1, 2^k}  }.
$$
Thus, the commutator term is less than the right-hand side of \eqref{4.6.0}.

$(ii)$ The proof is almost the same as of $(i)$.
Let us point out some minor differences.
This time, we denote
 $$
	 \widetilde L u = \cR_x  (-\Delta_x)^{1/3} (\eta (-\Delta_x)^{1/6} u),
 $$
 $$
     \widetilde A = \cR_x (-\Delta_x)^{1/3} = D_x (-\Delta_x)^{-1/6},
    \quad \widetilde w = (-\Delta_x)^{1/6} u,
 $$
 $$
	\textit{Comm} =  \eta \cR_x (-\Delta_x)^{1/2} u - \cR_x  (-\Delta_x)^{1/3} (\eta (-\Delta_x)^{1/6} u) =    \eta	\widetilde A \widetilde w - \widetilde A  (\eta \widetilde w).
 $$
 Note that $\eta  \widetilde w$ satisfies Eq. \eqref{4.6.1} with $w$ replaced with $\widetilde w$.
 Then, by Theorem \ref{theorem 4.1}
 and Lemma \ref{lemma 4.2} $(i)$,
\begin{align*}
&  \|\widetilde L u\|_{ L_2 (\bR^{1+2d}_0)}\\
&    \leq N (d) \delta^{-1} ( \| \widetilde w P_0\eta\|_{ L_2 (\bR^{1+2d}_0)}
    +   \|   (a D_v \eta) \cdot D_v \widetilde w \|_{ L_2 (\bR^{1+2d}_0)})
    \leq N  \delta^{- 4} \|\widetilde w\|_{  L_2 (Q_{1})  }.
\end{align*}
  Furthermore, as above
 $$
    |\textit{Comm}|
    \leq  N \int_{\bR^d} |\widetilde w (t, x - y, v)|  \ |\eta (t, x, v) -  \eta (t, x-y, v)| \, |y|^{-(d+2/3) } \, dy.
 $$
Finally, we repeat the last paragraph of the proof of $(i)$.
The lemma is proved.
\end{proof}

\begin{lemma}
			\label{lemma 4.4}
Let $R \in (1/2, 1)$ be a number and
$
  u   \in
  S_{2,\text{loc}} (\bR^{1+2d}_0)
$
be a function such that $P_0 u  = 0$ in $(-1, 0) \times \bR^d \times B_1$.
Then,  for any $l, m = \{0, 1, 2,\ldots\}$, there exists a constant $\theta = \theta (d, l, m) > 0$ such that
$$
		(i) \, 	\sup_{ Q_{1/2} } |D_x^l D_v^m u|+
		\sup_{ Q_{1/2} } |\partial_t D_x^l D_v^m u|
		\leq N (d, l, m, R) \delta^{-\theta} \| u \|_{ L_2 (Q_R) },
$$
\begin{align*}
   (ii) \, &\sup_{ Q_{1/2}}
    (|D^l_x D_v^{m+1} (-\Delta_x)^{1/6} u|
    + |\partial_t D^l_x D_v^{m+1}  (-\Delta_x)^{1/6} u|)\\
    &\leq N (d, l, m)  \delta^{-\theta}(\|D_v (-\Delta_x)^{1/6} u\|_{L_2 (Q_1)}
    +
	 \sum_{k = 0}^{\infty}
	 2^{-2k}	 \big(|(-\Delta_x)^{1/3} u|^2)^{1/2}_{ Q_{1, 2^k} }\big),
\end{align*}
\begin{align*}
			(iii) \, &\sup_{ Q_{1/2} } |D_x^l D_v^{m+2} u|+
			\sup_{ Q_{1/2} } |\partial_t D_x^l D_v^{m+2}  u| 	
			\leq  N (d, l, m) \delta^{-\theta} \big(\| D^2_v u \|_{ L_2 (Q_1)}\\
 &\quad+ \sum_{k = 0}^{\infty}
	 2^{-k}	  (|(-\Delta_x)^{1/3} u|^2)^{1/2}_{ Q_{1, 2^k} }\big).
\end{align*}
\end{lemma}

\begin{proof}
In this proof, $N$ is a constant independent of $\delta$.

$(i)$
Let us fix some number $r \in (1/2, R)$.
First, we prove that for any
$m \in \{0, 1, 2, \ldots\}$,
\begin{equation}
			\label{4.4.1}
\| D^{m+1}_v u \|_{ L_2 (Q_r) }
	\leq N (d, r, R, m) \delta^{-\theta} \|u\|_{ L_2 (Q_R) },
\end{equation}
where $\theta = \theta (d, m).$

Proof by induction.
First, for $m  = 0$ the assertion holds due to Lemma \ref{lemma 4.2} $(i)$.
For $m > 0$,
let  $\alpha = (\alpha_1, \ldots \alpha_d)$ be a multi-index
of order $m$.
Then, by the product rule, formally we have
\begin{equation}
			\label{4.4.0}
	P_0 (D^{\alpha}_v u)
	= \sum_{  \widetilde \alpha:  \, \widetilde \alpha < \alpha, |\widetilde \alpha| = m-1 } c_{\widetilde \alpha}
	D^{\widetilde \alpha}_v D^{\alpha - \widetilde \alpha}_{x} u .
\end{equation}
We fix some numbers
$r_1, r_2$ such that
$r < r_1 < r_2  < R$.
By Lemma \ref{lemma 4.2} $(i)$,
\begin{align}
			\label{4.4.4}
\| D^{m+1}_v u \|_{ L_2 (Q_r) } \leq N  \delta \| D^{m-1}_v D_x u \|_{ L_2 (Q_{r_1}) }
+N \delta^{  -2 }  \| D^{m}_v u \|_{ L_2 (Q_{r_1}) }.
\end{align}
Observe that $D_x u$ satisfies
 $	P_0 (D_x u) = 0$
in  $(-1, 0)\times \bR^d \times B_1$. Hence, by the induction hypothesis and Lemma \ref{lemma 4.3},
\begin{equation}
			\label{4.4.5}
	\| D^{m-1}_v D_x u \|_{ L_2 (Q_{r_1}) } \leq  N  \delta^{-\theta} \| D_x u \|_{ L_2 (Q_{r_2}) } \leq N  \delta^{-\theta} \| u \|_{ L_2 (Q_R) }.
\end{equation}
To make the argument above rigorous, we need to use the method of finite-difference quotient. Thus, by induction \eqref{4.4.1} holds.

Next, observe that for any multi-index
$\alpha$, the function $D^{\alpha}_x u$ satisfies
$$
P_0 (D^{\alpha}_x u)= 0\quad
\text{in}\,\,
(-1, 0)\times \bR^d \times B_1.
$$
Then, by \eqref{4.4.1},
$$
	\|D^m_v D^l_x u \|_{ L_2 (Q_{1/2})   }
	 \leq  N \delta^{-\theta}  \| D^l_x u \|_{ L_2 (Q_r) }.
$$
Iterating the estimate of Lemma \ref{lemma 4.3}, we get
\begin{equation}
			\label{4.6.2}
	\|D^m_v D^l_x u \|_{ L_2 (Q_{1/2})   }
		 \leq N \delta^{-\theta} \| u \|_{ L_2 (Q_R) }.
\end{equation}
Here again, we need to use the method of finite-difference quotient.

Furthermore, the fact that $|a^{i j}| < \delta^{-1},$
\begin{equation}
                                \label{eq7.41bb}
	\partial_t u = a^{i j} D_{    v_i v_j } + v \cdot D_x u
\end{equation}
 in $Q_1$ combined with \eqref{4.6.2} yields the estimate
$$
	 \| \partial_t  D^m_v D^l_x u  \|_{ L_2 (Q_{1/2})   }
	  \leq N  \delta^{-\theta}  \| u \|_{ L_2 (Q_R) }.
$$
By this, \eqref{4.6.2}, and the Sobolev embedding theorem, we prove the inequality for the sup-norm of $D_x^l D_v^m u$.
 The estimate of the second term follows from \eqref{eq7.41bb}.

$(ii)$
First, since $u- (u)_{Q_r}$ satisfies the same equation as $u$,
by the assertion $(i)$ and the Poincar\'e  inequality,
\begin{align*}
   J:& = \sup_{Q_{1/2}}
    (|D^l_x D_v^{m+1}  u|
    + |\partial_t D^l_x D_v^{m+1} u|)
    \le N  \delta^{-\theta} \|u - (u)_{Q_r}\|_{ L_2 (Q_r)  }\\
&    \le N \delta^{-\theta}  (\|D_v u\|_{L_2 (Q_r)} + \|D_x u\|_{L_2 (Q_r)}
   + \|\partial_t u\|_{L_2 (Q_r)}).
\end{align*}
Furthermore,  by \eqref{eq7.41bb} and \eqref{4.4.4}, we get
\begin{align*}
    \|\partial_t u\|_{L_2 (Q_r)}
  &  \leq   N   \delta^{-1} \|D^2_v u\|_{L_2 (Q_r)} +    N \|D_x u\|_{L_2 (Q_r)}\\
   &     \leq N \delta^{-\theta}(\|D_x u\|_{L_2 (Q_{R})}+ \|D_v u\|_{L_2 (Q_{R})}),
\end{align*}
where $R \in (r, 1)$.
By the above,
\begin{equation}
                \label{4.4.1.1}
    J \leq N \delta^{-\theta}(\|D_v u\|_{L_2 (Q_R)} +  \|D_x u\|_{L_2 (Q_R)}).
\end{equation}
Finally, note that
$
    P_0 ((-\Delta_x)^{1/6}u) = 0
$
in $(-1, 0) \times \bR^d \times B_1$.
Substituting $(-\Delta_x)^{1/6} u$ in \eqref{4.4.1.1} and using
Lemma \ref{lemma 4.6} $(ii)$, we prove the assertion.

$(iii)$ Denote
	$$
		u_1 (z)  = u (z) - (u)_{Q_r}  - v^i  (D_{  v_i } u)_{Q_r}
	$$
	and observe that
	$
		P_0  u_1 = 0
	$
	in
	$(-1, 0) \times \bR^d \times B_1$.
	Then, by  $(i)$,
	\begin{equation}
	                \label{4.4.9}
	    \sup_{  Q_{1/2} } |D_x^l D_v^m D^2_v u|
	+\sup_{  Q_{1/2} } |\partial_t D_x^l D_v^m D^2_v u| \leq N  \delta^{-\theta}  \| u_1 \|_{ L_2 (Q_r) }.
	  \end{equation}

Applying the Poincar\'e inequality twice and using \eqref{eq7.41bb}, we get
	\begin{align}
							\label{4.4.3}
		& \|  u_1 \|_{ L_2 (Q_r) }\notag\\
	& \leq N (\| D_x u \|_{ L_2 (Q_r) } +  \| \partial_t u \|_{ L_2 (Q_r) }
		+  \| D_v u -  (D_v u)_{Q_r}\|_{ L_2 (Q_r) }) \\
		&\leq N  (\| D_x u \|_{ L_2 (Q_r) } + \delta^{-1}  \| D^2_{v} u \|_{ L_2 (Q_r) }  + \|  D_{v x} u \|_{ L_2 (Q_r) } +  \| \partial_t D_{v}  u \|_{ L_2 (Q_r) })\notag.
	\end{align}
	By \eqref{4.4.5} with $m = 2$,
	\begin{equation}
				\label{4.4.6}
		 \|  D_{v x} u \|_{ L_2 (Q_r) } \leq N  \delta^{-\theta}  \| D_x u \|_{ L_2 (Q_R) }.
	\end{equation}
	Furthermore, by \eqref{4.4.0} with $|\alpha| = 1$,
	\begin{equation}
	\begin{aligned}
	                \label{4.4.8}
			&\| \partial_t D_{v}  u \|_{ L_2 (Q_r) } \\
		&\leq N \delta^{-1} \| D_v^3 u \|_{ L_2 (Q_r) }
			 + N \|D_{v x}  u\|_{ L_2 (Q_r) }
			+ N \| D_x u\|_{ L_2 (Q_r) }.
	\end{aligned}
	\end{equation}
	By \eqref{4.4.4}   with $m = 2$ and \eqref{4.4.6},
	we have
	\begin{equation}
				\label{4.4.7}
		\| D_v^3 u \|_{ L_2 (Q_r) } \leq   N \delta^{-\theta}(\| D_x u \|_{ L_2 (Q_R) }+\| D_v^2 u \|_{ L_2 (Q_R)}).
	\end{equation}
	Combining \eqref{4.4.3} - \eqref{4.4.7}, we obtain
	$$
		 \|  u_1 \|_{ L_2 (Q_r) } \leq N  \delta^{-\theta}  (\| D^2_{v} u \|_{ L_2 (Q_{ R  }) }
 +   \| D_x u \|_{ L_2 (Q_R) }).
	$$
	Now the  assertion follows from \eqref{4.4.9}, the above inequality, and Lemma \ref{lemma 4.6} $(i)$.
	\end{proof}

\begin{proof}[Proof of  Proposition \ref{theorem 4.7}]
    Let $\widetilde u$ and $\widetilde P_0$ be the function
    and the operator from Lemma \ref{lemma 4.1} defined  with $\nu r$ in place of $r$.
Then, by the aforementioned lemma, we have
$$
   \widetilde P_0 \widetilde  u = 0 \quad \text{in} \ (-1, 0) \times \bR^d \times B_1,
$$
and for any $c > 0$,
$$
	 \fint_{Q_{\nu r, c \nu r} (z_0)} |(-\Delta_x)^{1/3} u|^2 \, dz
	 = (\nu r)^{-4}	 \fint_{Q_{1, c} } |(-\Delta_x)^{1/3} \widetilde u|^2 \, dz,
$$
\begin{align*}
&\fint_{Q_r (z_0)} \bigg|(-\Delta_x)^{1/3} u -  ((-\Delta_x)^{1/3} u)_{Q_r (z_0)} \bigg|^2 \, dz\\
&	= (\nu r)^{-4}
	\fint_{ Q_{1/\nu}} \bigg|(-\Delta_x)^{1/3} \widetilde u - ((-\Delta_x)^{1/3}\widetilde u)_{ Q_{1/\nu}} \bigg|^2 \, dz.
 \end{align*}	
Similar identities hold for $D_v (-\Delta_x)^{1/6} u$
and $D^2_v u$.
Hence, we may assume that $r =1/\nu$ and $z_0 = 0$.

Next, the fact that
$P_0 ((-\Delta_x)^{1/3} u) = 0$
in $(-1, 0) \times \bR^d \times B_1$
combined with
Lemma \ref{lemma 4.4} $(i)$ gives
\begin{align*}
 	J_1 & \leq  \sup_{ z_1, z_2 \in Q_{1/\nu} } |(-\Delta_x)^{1/3} u (z_1)  - (-\Delta_x)^{1/3} u (z_2)|\\
	 	&\leq N(d) \nu^{-1} \delta^{-\theta}  \bigg(\fint_{Q_1} |(-\Delta_x)^{1/3} u|^2 \, dz\bigg)^{1/2}.
\end{align*}
Similarly, by Lemma \ref{lemma 4.4} $(ii)$, we prove the estimate of $J_2$.
Finally, Lemma \ref{lemma 4.4} $(iii)$
implies the desired estimate of $J_3$.
\end{proof}

\subsection{Proof of Theorem \ref{theorem 4.10}}

Let us give an outline of the proof.
First, we prove a mean oscillation estimate (see Proposition \ref{theorem 4.7.2}), and, as a result, we obtain Theorem \ref{theorem 4.10} $(i)$
with $p > 2$ and $\lambda = 0$.
Then, by using Agmon's method, we derive the a priori estimate
of $\lambda \|u\|_{L_p (\bR^{1+2d}_T)}$ for the same range of $p$.
Furthermore, we show that  $(P_0  + \lambda) C^{\infty}_0 (\bR^{1+2d})$ is dense in $L_p (\bR^{1+2d})$ for $p > 2$ and $\lambda \ge 0$.
Thus, we prove the unique solvability of Eq. \eqref{2.3} for $p > 2$.
The results for $p \in (1, 2)$ are obtained by using the duality method.

\begin{proposition}
                \label{theorem 4.7.2}
Let $r > 0$, $\nu \geq 2$ be numbers,
 $z_0 \in \overline{\bR^{1+2d}_T}$,
and
$u \in S_{2} (\bR^{1+2d}_T)$.
Suppose that $P_0u=f$ in $\bR^{1+2d}_T$.
Then, there exists constants $\theta = \theta (d) >0$ and $N =  N (d) > 0$ such that the following assertions hold.
\begin{equation*}
\begin{aligned}
     (i)\,
    &   \bigg(\big|(-\Delta_x)^{1/3} u -  ((-\Delta_x)^{1/3} u)_{ Q_r (z_0) }\big|^2\bigg)^{1/2}_{Q_r (z_0)} \\
&
    \leq N \nu^{-1} \delta^{-\theta}   (|(-\Delta_x)^{1/3} u|^2)^{1/2}_{ Q_{\nu r} (z_0) }\\
&\quad    + N  \nu^{1+2d}  \delta^{-\theta}  \sum_{k=0}^\infty 2^{-2k}(f^2)^{1/2}_{Q_{ 2\nu r ,2^{k+1} \nu r/\delta^2} (z_0)},
\end{aligned}
\end{equation*}
  \begin{align*}
    (ii)\, &   \bigg(\big|D_v(-\Delta_x)^{1/6} u -  (D_v (-\Delta_x)^{1/6} u)_{ Q_r (z_0) }\big|^2\bigg)^{1/2}_{Q_r (z_0)}\\
&    \leq N \nu^{-1}  \delta^{-\theta}   (|D_v (-\Delta_x)^{1/6} u|^2)^{1/2}_{ Q_{\nu r} (z_0) }\\
&\quad  + N \nu^{-1} \delta^{-\theta}  \sum_{k=0}^\infty 2^{-2k}(|(-\Delta_x)^{1/3}u|^2)^{1/2}_{Q_{\nu r,2^{k} \nu r } (z_0)}\\
&\quad +  N   \nu^{1+2d} \delta^{-\theta}   \sum_{k=0}^\infty 2^{-k}(f^2)^{1/2}_{Q_{2\nu r,2^{k+1} \nu r/\delta^2}  (z_0)},
 \end{align*}
\begin{align*}
	(iii)\, &    \bigg(\big|D^2_v u -  (D^2_v u)_{ Q_r (z_0) }\big|^2\bigg)^{1/2}_{Q_r (z_0)}\\
&	 \leq
	    N \nu^{ - 1} \delta^{-\theta}  (|D_v^2  u|^2)^{1/2}_{ Q_{\nu r} (z_0) }
   + N \nu^{ - 1 } \delta^{-\theta}
   \sum_{k=0}^\infty 2^{-k}(|(-\Delta_x)^{1/3}u|^2)^{1/2}_{Q_{\nu r,2^{k} \nu r } (z_0)}\notag\\
&\quad    + N   \nu^{ 1 + 2 d }   \delta^{-\theta} \sum_{k=0}^\infty 2^{-k}(f^2)^{1/2}_{Q_{2\nu r,2^{k+1} \nu r/\delta^2} (z_0)}\notag.
\end{align*}
\end{proposition}

\begin{proof}
Here we assume that $N$ depends only on $d$.

Let $\phi \in C^{\infty}_0 ((t_0 - (2\nu r)^2, t_0+ (2\nu r)^2) \times B_{2\nu r} (v_0))$
be a function of $(t, v)$
such that $\phi = 1$
on
$
(t_0 - (\nu r)^2, t_0) \times B_{\nu r} (v_0).
$
Then, by Theorem \ref{theorem 4.1}
there exists a unique solution
$g \in S_2 ((t_0 - (2\nu r)^2, t_0) \times \bR^{2d})$
to the Cauchy problem
$$
    P_0 g = f \phi, \quad g (t_0 - (2\nu r)^2, \cdot) \equiv 0,
$$
and, by Lemma \ref{lem4.7} and the scaling argument (see Lemma \ref{lemma 4.1}),
\begin{equation}
                \label{4.10.7.1}
   \fint_{Q_{\nu r} (z_0)}|(-\Delta_x)^{1/3} g|^2 \, dz
\le N  \delta^{-\theta} \bigg(\sum_{k=0}^\infty 2^{-2k}(f^2)^{1/2}_{Q_{2\nu r,2^{k +1 } \nu r/\delta^2 } (z_0)}\bigg)^2,
\end{equation}
\begin{equation}
                \label{4.10.7}
\begin{aligned}
 &  \fint_{Q_r (z_0)} |(-\Delta_x)^{1/3} g|^2 \, dz
\le N  \nu^{2+4d} \fint_{Q_{ 2  \nu r} (z_0)}|(-\Delta_x)^{1/3} g|^2 \, dz\\
&\le N \nu^{2+4d} \delta^{-\theta} \bigg(\sum_{k=0}^\infty 2^{-2k}(f^2)^{1/2}_{Q_{ 2 \nu r,2^{k +1 } \nu r /\delta^2} (z_0)}\bigg)^2.
\end{aligned}
\end{equation}
Furthermore, note that the function
$
    h = u - g \in S_2 ((t_0 - (2\nu r)^2, t_0) \times \bR^{2d})
$
satisfies
$$
   P_0 h = f (1 - \phi)  \quad \text{in} \,\, (t_0 - (2\nu r)^2, t_0) \times \bR^{2d}.
$$
Then, by  Proposition \ref{theorem 4.7} and  \eqref{4.10.7.1},
\begin{align*}
&     \fint_{ Q_r (z_0) } \bigg|(-\Delta_x)^{1/3} h -  ((-\Delta_x)^{1/3} h)_{ Q_r (z_0) }\bigg|^2 \, dz\\
&     \leq N \nu^{-2} \delta^{-\theta}  (|(-\Delta_x)^{1/3} u|^2)_{ Q_{\nu r} (z_0) }
     + N \nu^{-2}  \delta^{-\theta}  (|(-\Delta_x)^{1/3} g|^2)_{ Q_{\nu r} (z_0) }\\
&    \leq N \nu^{-2} \delta^{-\theta}   (|(-\Delta_x)^{1/3} u|^2)_{ Q_{\nu r} (z_0) }\\
&\quad
    + N  \nu^{ -2 }  \delta^{-\theta} \bigg(\sum_{k=0}^\infty 2^{-2k}(f^2)^{1/2}_{Q_{2\nu r,2^{k+1} \nu r/\delta^2 } (z_0)}\bigg)^2.
\end{align*}
Combining this with \eqref{4.10.7}, we prove the assertion $(i)$.

$(ii)$
 By Lemmas \ref{lem4.7} and \ref{lemma 4.1},
 \begin{align}
            \label{4.10.12}
&\bigg(\fint_{Q_{\nu r} (z_0)} |D_v (-\Delta_x)^{1/6} g|^2 \, dz\bigg)^{1/2} \\
&\le N  \delta^{-\theta} \sum_{k=0}^\infty 2^{- k } (f^2)^{1/2}_{Q_{2\nu r,2^{k+1} \nu r/\delta^2 } (z_0)}\notag,
\end{align}
 \begin{align}
            \label{4.10.11}
&  \bigg(\fint_{Q_r (z_0)}|D_v (-\Delta_x)^{1/6} g|^2  \, dz\bigg)^{1/2}\\
&\le N    \delta^{-\theta} \nu^{1+2d}\sum_{k=0}^\infty 2^{- k }(f^2)^{1/2}_{Q_{2\nu r,2^{k+1} \nu r/\delta^2 } (z_0)}\notag.
\end{align}
Furthermore,  by  Proposition \ref{theorem 4.7} and the triangle inequality,
\begin{align*}
&  \bigg( \fint_{ Q_r (z_0) } \bigg|D_v (-\Delta_x)^{1/6} h -  (D_v(-\Delta_x)^{1/6} h)_{ Q_r (z_0) }\bigg|^2 \, dz\bigg)^{1/2}     \\
& \leq N \nu^{-1  } \delta^{-\theta}   (|D_v (-\Delta_x)^{1/6} u|^2)^{1/2}_{ Q_{\nu r} (z_0) }   +  N \nu^{-1} \delta^{-\theta}   (|D_v (-\Delta_x)^{1/6} g|^2)^{1/2}_{ Q_{\nu r} (z_0) }\\
 &  \quad + N \nu^{- 1} \delta^{-\theta}  \sum_{k=0}^\infty 2^{-2k} (|(-\Delta_x)^{1/3}u|^2)^{1/2}_{Q_{\nu r,2^{k} \nu r } (z_0)} \\
  &\quad +   N \nu^{- 1} \delta^{-\theta}  \sum_{k=0}^\infty 2^{-2k} (|(-\Delta_x)^{1/3}g|^2)^{1/2}_{Q_{\nu r,2^{k} \nu r } (z_0)}.
\end{align*}
 By  \eqref{4.10.7.1} and \eqref{4.10.12},  we estimate the terms containing $g$ on the right-hand side of the above inequality.
We obtain
\begin{align*}
&  \bigg( \fint_{ Q_r (z_0) } \bigg|D_v (-\Delta_x)^{\frac 1 6} h -  (D_v(-\Delta_x)^{\frac 1 6} h)_{ Q_r (z_0) }\bigg|^2 \, dz\bigg)^{1/2}     \\
& \leq N \nu^{- 1 } \delta^{-\theta}   (|D_v (-\Delta_x)^{\frac 1 6} u|^2)^{1/2}_{ Q_{\nu r} (z_0) }   \\
&\quad  + N \nu^{-1 } \delta^{-\theta}  \sum_{k=0}^\infty 2^{-2k} (|(-\Delta_x)^{\frac 1 3}u|^2)^{1/2}_{Q_{\nu r,2^{k} \nu r } (z_0)} \\
&\quad +    N  \nu^{ -1 } \delta^{-\theta}   \sum_{l=0}^\infty 2^{-l}(f^2)^{1/2}_{Q_{2\nu r,2^{l+1} \nu r/\delta^2 } (z_0)}\\
  &\quad + N  \nu^{ -1 } \delta^{-\theta}   \sum_{k, l=0}^\infty 2^{-2k-l}(f^2)^{1/2}_{Q_{2\nu r,2^{k+l+1} \nu r/\delta^2} (z_0)}.
\end{align*}
Changing the index of summation $l \to k+l$, we replace the last term with
$$
	 N  \nu^{ -1 } \delta^{-\theta}   \sum_{l=0}^\infty 2^{-l}(f^2)^{1/2}_{Q_{2\nu r,2^{l+1} \nu r/\delta^2} (z_0)}.
$$
Combining this inequality with \eqref{4.10.11}, we prove the assertion $(ii)$.

$(iii)$
This time, by Lemma \ref{lem4.7} we have
 \begin{equation}
            \label{4.10.13}
\bigg(\fint_{Q_{\nu r} (z_0)} |D_v^2  g|^2 \, dz\bigg)^{1/2}
\le N  \delta^{-\theta} \sum_{k=0}^\infty 2^{- k^2/8 } (f^2)^{1/2}_{Q_{2\nu r,2^{k+1} \nu r/\delta^2 } (z_0)},
\end{equation}
 \begin{equation}
            \label{4.10.14}
  \bigg(\fint_{Q_r (z_0)}|D_v^2  g|^2  \, dz\bigg)^{1/2}
\le N    \delta^{-\theta} \nu^{ 1+2d }\sum_{k=0}^\infty 2^{- k^2/8 }(f^2)^{1/2}_{Q_{2\nu r,2^{k+1} \nu r/\delta^2} (z_0)}.
\end{equation}

Next, by using Proposition  \ref{theorem 4.7} and \eqref{4.10.13} and arguing as above, we get
\begin{align*}
 &\bigg(\fint_{ Q_r (z_0) } \bigg|D^2_v  h -  (D_v^2 h)_{ Q_r (z_0) }\bigg|^2 \, dz\bigg)^{1/2}\\
&\le
	N \nu^{-1} \delta^{-\theta}   (|D_v^2  u|^2)^{1/2}_{ Q_{\nu r} (z_0) } + N \nu^{- 1 } \delta^{-\theta} (|D_v^2  g|^2)^{1/2}_{ Q_{\nu r} (z_0) }\\
&
	\,\, + N \nu^{-1} \delta^{-\theta} \sum_{k=0}^\infty 2^{-k}\Big((|(-\Delta_x)^{1/3}u|^2)^{1/2}_{Q_{\nu r,2^{k} \nu r } (z_0)}+
(|(-\Delta_x)^{1/3} g|^2)^{1/2}_{Q_{\nu r,2^{k} \nu r } (z_0)} \Big)\\
&\le  N \nu^{-1} \delta^{-\theta}   (|D_v^2  u|^2)^{1/2}_{ Q_{\nu r} (z_0) } + N \nu^{-1} \delta^{-\theta} \sum_{k=0}^\infty 2^{-k}(|(-\Delta_x)^{1/3}u|^2)^{1/2}_{Q_{\nu r,2^{k} \nu r } (z_0)} \\
&\,\, +   N \nu^{-1 } \delta^{-\theta} \sum_{l =0}^\infty 2^{-l^2/8} (f^2)^{1/2}_{Q_{2\nu r,2^{l+1} \nu r/\delta^2 } (z_0)}\\
 &\,\, + N \nu^{-1 } \delta^{-\theta} \sum_{k, l =0}^\infty 2^{  -l^2/8-k  }(f^2)^{1/2}_{Q_{2\nu r,2^{k+l+1} \nu r/\delta^2 } (z_0)}.
\end{align*}
As above, we may replace the double sum with the term
$$
	 N \nu^{-1} \delta^{-\theta} \sum_{k=0}^\infty 2^{-k} (f^2)^{1/2}_{Q_{2\nu r,2^{k+1} \nu r/\delta^2 } (z_0)}.
$$
As before, this inequality and \eqref{4.10.14} imply the estimate of the mean-square oscillation of $D^2_v u$.
\end{proof}

\begin{proposition}
            \label{theorem 4.12}
For any $p \in (2, \infty)$,
$T \in (-\infty, \infty]$,
and $u \in S_p (\bR^{1+2d}_T)$,
the estimate \eqref{eq2.38} holds
with $\lambda = 0$.
\end{proposition}

\begin{proof}
By  Proposition \ref{theorem 4.7.2},
there exist constants $N  = N (d)$ and $\theta  = \theta (d) > 0$ such that
for any $z \in \bR^{1+2d}_T$,
\begin{align*}
((-\Delta_x)^{1/3} u)^{\#}_{T} (z)
&\leq  N  \nu^{-1} \delta^{-\theta}  \cM^{1/2}_{ T} |(-\Delta_x)^{1/3} u|^2 (z)\\
&\quad     + N   \nu^{1+2d}   \delta^{-\theta} \sum_{k=0}^\infty 2^{-2k} \bM^{1/2}_{2^k, T} f^2 (z),
\end{align*}
\begin{align*}
&     (D_v (-\Delta_x)^{1/6} u)^{\#}_{T} (z)
    \leq N \nu^{-1}   \delta^{-\theta}  \cM^{1/2}_{T} |D_v (-\Delta_x)^{1/6} u|^2 (z)\\
&\quad   +  N   \nu^{1+2d}   \delta^{-\theta} \sum_{k=0}^\infty 2^{-k } \bM^{1/2}_{2^{k }/\delta^2, T} f^2 (z)\\
&\quad  + N \nu^{- 1 }   \delta^{-\theta}  \sum_{k=0}^\infty 2^{- 2 k } \bM^{1/2}_{2^k, T} |(-\Delta_x)^{1/3} u|^2 (z),
\end{align*}
\begin{align*}
  (D^2_v u)^{\#}_{ T} (z)
&	 \leq  N \nu^{-1}   \delta^{-\theta}\cM^{1/2}_{ T} |D^2_v u|^2 (z)
    +	 N 	\nu^{1 + 2d}  \delta^{-\theta} \sum_{k=0}^\infty 2^{-k  } \bM^{1/2}_{  2^k/\delta^2, T} f^2 (z)\\
&\quad  + N \nu^{-1}  \delta^{-\theta}  \sum_{k=0}^\infty 2^{-  k } \bM^{1/2}_{2^k, T} |(-\Delta_x)^{1/3} u|^2 (z).
\end{align*}
We raise the above inequalities
to the $p$-th power, integrate over $\bR^{1+2d}_T$,
and use the Minkowski inequality.
 Furthermore, we apply Corollary \ref{corollary 4.9} $(i)$
 and   $(ii)$ with $p/2 > 1$.
 We obtain
 \begin{align}
                \label{4.12.1}
  &  \|(-\Delta_x)^{1/3} u\|_{ L_p (\bR^{1+2d}_T) }\\
& \leq  N  \nu^{-1}   \delta^{-\theta} \|(-\Delta_x)^{1/3} u\|_{ L_p (\bR^{1+2d}_T) }
    + N \nu^{1+2d}  \delta^{-\theta} \|f\|_{ L_p (\bR^{1+2d}_T) }\notag,
 \end{align}
\begin{align*}
      &\|D_v (-\Delta_x)^{1/6} u\|_{ L_p (\bR^{1+2d}_T) }
    \le N  \nu^{-1}   \delta^{-\theta} \|D_v (-\Delta_x)^{1/6} u\|_{ L_p (\bR^{1+2d}_T) }\\
&\quad     + N \nu^{-1}  \delta^{-\theta} \|(-\Delta_x)^{1/3} u\|_{ L_p (\bR^{1+2d}_T) }
    + N \nu^{1+2d}  \delta^{-\theta} \|f\|_{ L_p (\bR^{1+2d}_T) },
\end{align*}
\begin{align*}
&     \|D^2_v u\|_{ L_p (\bR^{1+2d}_T) }\le N  \nu^{-1}  \delta^{-\theta}
     \|D^2_v u\|_{ L_p (\bR^{1+2d}_T) }\\
  &\quad   + N \nu^{-1 }   \delta^{-\theta} \|(-\Delta_x)^{1/3} u\|_{ L_p (\bR^{1+2d}_T) }
    + N \nu^{1+2d}   \delta^{-\theta}\|f\|_{ L_p (\bR^{1+2d}_T) }.
\end{align*}
By setting $\nu  =  2 (1+N  \delta^{-\theta})$, we  cancel the term containing $(-\Delta_x)^{1/3}u$
   on the right-hand side of \eqref{4.12.1}.
   Using this and our choice of $\nu$, we prove the estimates for $D_v (-\Delta_x)^{1/6} u$
   and $D^2_v u$.
The theorem is proved.
\end{proof}

\begin{lemma}
            \label{lemma 4.13}
Under the assumptions of Proposition \ref{theorem 4.12},
for any $\lambda > 0$,
$$
    \lambda \|u\|_{ L_p (\bR^{1+2d}_T) }
    \le N  \delta^{-\theta} \|P_0 u + \lambda u\|_{ L_p (\bR^{1+2d}_T) },
$$
where $N = N (d, p)$ and $\theta  = \theta (d) > 0$.
\end{lemma}

\begin{proof}
We use S. Agmon's method (see, for example, the proof of Lemma 6.3.8 of \cite{Kr_08}).

Denote
 $$
	 \hat x = (x_1, \ldots, x_{d+1}),
	 \quad  \hat v = (v_1, \ldots,  v_{d+1}),
	 \quad  \hat z = (t, \hat x, \hat v),
$$
$$
	\hat P_0 (\hat z) = \partial_t - \sum_{i = 1}^{d+1} v_i D_{x_i} - \sum_{i, j = 1}^{d} a^{i j} D_{v_i v_j}  - D_{v_{d+1} v_{d+1}}.
$$
Let $\zeta$ be a smooth cutoff function on $\bR$ such that  $\zeta \not \equiv 0$ and
 denote
$$
	\hat u (\hat z) = u (z) \zeta (v_{d+1}) \cos (\lambda^{1/2} v_{d+1}).
$$

By direct calculations,
\begin{equation}
            \label{4.13.1}
\begin{aligned}
&	\lambda u (z) \zeta (v_{d+1}) \cos (\lambda^{1/2} v_{d+1})
	 = - D_{v_{d+1} v_{d+1}} \hat u (\hat z) \\
 &
+ u (z) (\zeta'' (v_{d+1}) \cos (\lambda^{1/2} v_{d+1}) - 2 \lambda^{1/2}  \zeta' (v_{d+1}) \sin (\lambda^{1/2} v_{d+1})),
\end{aligned}	
\end{equation}
 \begin{equation}
            \label{4.13.2}
\begin{aligned}
&   \hat P_0 \hat u (\hat z)  = \zeta (v_{d+1}) \cos (\lambda^{1/2} v_{d+1})
(P_0 u (z) + \lambda u (z)) \\
	&-  u (z) (\zeta'' (v_{d+1}) \cos (\lambda^{1/2} v_{d+1}) - 2 \lambda^{1/2} \zeta' (v_{d+1}) \sin (\lambda^{1/2} v_{d+1})).
   \end{aligned}
    \end{equation}
Note that for any $p > 0$  and $\lambda > 1$
$$
	\int_{\bR} |\zeta (t) \cos (\lambda^{1/2} t)|^p \, dt \ge N_1 (p) > 0.
$$
This combined with \eqref{4.13.1} gives
$$
	\lambda \|u\|_{ L_p (\bR^{1+2d}_T)  } \le  N(p) \|D_{v_{d+1} v_{d+1}} \hat u\|_{ L_p (\bR^{3+2d}_T)  } + N(p)  (1+\lambda^{1/2}) \| u\|_{ L_p (\bR^{1+2d}_T)  }.
$$
Furthermore, by Proposition \ref{theorem 4.12} and \eqref{4.13.2},
\begin{align*}
&	 \|D_{v_{d+1} v_{d+1}} \hat u\|_{ L_p (\bR^{3+2d}_T)  }
	 \le N\delta^{-\theta}  \|\hat P_0 \hat u\|_{ L_p (\bR^{3+2d}_T)  }\\
&	 \le N  \delta^{-\theta} \|P_0 u+ \lambda u\|_{ L_p (\bR^{1+2d}_T)  } + N \delta^{-\theta}  (1+\lambda^{1/2})  \|u\|_{ L_p (\bR^{1+2d}_T)  }.
\end{align*}
Thus, by the above,
$$
     \lambda \|u\|_{ L_p (\bR^{1+2d}_T)  }
     \le  N  \delta^{-\theta} \|P_0 u + \lambda u\|_{ L_p (\bR^{1+2d}_T)  } +  N  \delta^{-\theta} (1+\lambda^{1/2})  \|u\|_{ L_p (\bR^{1+2d}_T)  }.
$$
We  note that
for any $\lambda \ge  \lambda_0 = 16 N^2  \delta^{-2\theta} + 1$, one has
$
	\lambda  - N  \delta^{-\theta} (1 + \lambda^{1/2}) > \lambda/2.
$
This gives the desired estimate for $\lambda \ge \lambda_0$.
This restriction is removed by using a scaling argument (see Lemma \ref{lemma 4.1}).
\end{proof}

Combining Proposition \ref{theorem 4.12} with Lemma \ref{lemma 4.13}, we prove the following result.
\begin{corollary}
                \label{corollary 4.14}
Under the assumptions of  Proposition \ref{theorem 4.12}, the estimate \eqref{eq2.38} holds.
\end{corollary}

To prove the following lemma, we repeat the argument of Lemma \ref{lemma 4.2} and replace Theorem \ref{theorem 4.1} with Corollary \ref{corollary 4.14}.
\begin{lemma}
            \label{lemma 5.1}
Let  $p > 2$, $\lambda \geq 0$, and $r_1, r_2, R_1, R_2 > 0$
 be  numbers such that
$r_1 < r_2$,
and $R_1 < R_2$.
Let
$u \in S_{p,\text{loc}} (\bR^{1+2d}_0)$ and
denote $f = P_0 u + \lambda u$.
Then, there exist constants $N = N (d, p)$ and $\theta = \theta (d) > 0$ such that
the following assertions hold.
\begin{align*}
	(i)& \,
	\delta^{-\theta/2}    (r_2 - r_1)^{-1} \|D_v u\|_{ L_p (Q_{r_1,  R_1}) } + \| D^2_v u\|_{ L_p ( Q_{r_1,  R_1}) }
	\leq N \delta^{-\theta}  \| f \|_{ L_p ( Q_{r_2,  R_2}) }\\
	&\quad	+  N  \delta^{-\theta} ((r_2 - r_1)^{-2} + r_2 (R_2 - R_1)^{-3})  \|u\|_{ L_p ( Q_{r_2,   R_2}) }.
\end{align*}	

$(ii)$ Denote
$C_r= (-r^2, 0) \times \bR^d \times B_r$.
Then we have
\begin{align*}
&   \delta^{-\theta/2}   (r_2-r_1)^{-1}	\|D_v u\|_{ L_p  ( C_{r_1})}
	+\| D^2_v  u\|_{ L_p  (C_{r_1})   }\\
&	\leq N \delta^{-\theta} \big(\| f \|_{ L_p   (C_{r_2})    }
	+(r_2 - r_1)^{-2}
	\|u\|_{ L_p ( C_{r_2} ) }\big).
\end{align*}
\end{lemma}

\begin{lemma}
            \label{lemma 4.15}
For any $\lambda \ge 0$ and $p >1$,
the set $(P_0 + \lambda) C^{\infty}_0 (\bR^{1+2d})$
is dense in $L_p (\bR^{1+2d})$.
\end{lemma}

\begin{proof}
We may assume that $p \neq 2$, since the case $p =2$ is covered in Lemma \ref{lemma 10.1}.

Proof by contradiction.
Denote $q = p/(p-1)$.
If the claim does not hold, there exists a  function $ u \in L_q (\bR^{1+2d}_T)$
that is not identically zero and
such that for any $\zeta \in C^{\infty}_0 (\bR^{1+2d})$,
$$
	\int (P_0 \zeta +  \lambda \zeta) u \, dz = 0.
$$

\textit{Case $p \in (1, 2)$.}
We repeat the argument of Lemma \ref{lemma 10.1} with appropriate modifications.
This time, instead of Lemma \ref{lemma 4.2}, we use  Lemma \ref{lemma 5.1} $(i)$   with $R_1 = r_1 = r$ and $R_2 = r_2 = 2r$.
By this lemma and \eqref{3.1.2} with $2$ replaced with $q$, we conclude
\begin{align*}
   & \|D_v  w_{(\varepsilon)}\|_{L_q ({Q}_r) }
    \le N (d, \delta, q) (r \|g_{\varepsilon}\|_{L_q ({Q}_{2r}) } + r^{-1} \|w_{(\varepsilon)}\|_{L_q ({Q}_{2r}) })\\
  &  \le N (\varepsilon^{1/2} r + r^{-1}) \|w\|_{L_{ q } (\bR^{1+2d}) }.
\end{align*}
As in the proof of Lemma \ref{lemma 10.1}, this implies
that $w\equiv 0$, which gives a contradiction.

  \textit{Case $p > 2$.}
Let $\eta_{\varepsilon} =\eta_{\varepsilon} (x) $ be a standard mollifier.
For an integer $k\ge 1$, we denote by $w_{\varepsilon, k}$ the  $k$-fold mollification  of the function $w (z) = u (-t, -x, v)$ in the $x$ variable with $\eta_{\varepsilon}$.
 The idea of the proof is to, first,  show that $w_{\varepsilon , k } \in L_2 (\bR^{1+2d}) \cap S_{2, \text{loc}} (\bR^{1+2d})$ for some large $k$, and then conclude that $w_{\varepsilon, k}  \equiv 0$ by using the localized $S_2$-estimate.

Step 1. For $s \in (1, \infty)$ and an open set $G \subset \bR^{1+2d}$,  denote
$$
	\|f\|_{ \cW_s ( G ) } = \| |f| + |\partial_t f| + |D_v f| + |D^2_v f|\|_{ L_s (G)  }.
$$
Note that $w_{\varepsilon , 1} :  = w \ast \eta_{\varepsilon} $ satisfies
\begin{equation}
			\label{4.15.3}
	\partial_t w_{\varepsilon, 1}   - a^{i j} D_{v_i v_j} w_{\varepsilon , 1} + \lambda  w_{\varepsilon , 1} = v \cdot D_x w_{\varepsilon , 1}
\end{equation}
with
$$
	\|v \cdot D_x w_{\varepsilon , 1}\|_{  L_q (\tQ_{ 1/2 }) } \le N (d)  \varepsilon^{-1} \|w\|_{  L_q (\tQ_{1}) }
$$
provided that $\varepsilon \in (0, 1/2)$.
Mollifying Eq. \eqref{4.15.3} in the $v$ variable with a standard mollifier and applying the interior estimate for nondegenerate equations  for fixed $x \in B_{(1/4)^3}$ (see, for example, Theorem 5.2.5 of \cite{Kr_08}),  we get
$$
	\| w_{\varepsilon, 1}  \|_{ \cW_q (\tQ_{  1/4  }) } \le N (d, \delta, q, \varepsilon) \|w\|_{  L_q  (\tQ_{1}  ) }.
$$
By this and the Sobolev embedding theorem for any $q_1 > q$ such that
$$
	\frac{2}{d+2} \ge \frac{1}{q} - \frac{1}{q_1},
$$
 we obtain
$$
	\| w_{\varepsilon , 1}  \|_{ L_{q_1} (\tQ_{1/4}) } \le N (d, \delta, q, q_1, \varepsilon) \|w\|_{  L_q  (\tQ_{1}) }.
$$

 Step 2. There exists $m = m (d, q) \in \{1, 2, \ldots\}$ and a sequence $\{q_k, k   = 0, 1, 2, \ldots\}$, such that $q_0 = q$, $q_m = 2$, and
$$
		\frac{2}{d+2} \ge \frac{1}{q_{k-1}} - \frac{1}{q_{k}  } > 0, \, \, k = 1, \ldots, m.
$$
 We set  $w_{\varepsilon, 0} = w$,
$$
	w_{\varepsilon, k} = w_{\varepsilon, k-1} \ast \eta_{\varepsilon},  k = 1, \ldots, m.
$$
We claim that for $k=1, \ldots, m$,
$$
	\|w_{\varepsilon, k  }\|_{L_{q_k} (\tQ_{2^{- 2k  }}) } \le  N (d, \delta,  q_{k-1}, q_{k}, k, \varepsilon) \|w_{\varepsilon, k-1}\|_{ L_{q_{k-1}} (\tQ_{2^{- 2 (k-1)}}) }.
$$
To prove this,  we repeat the argument of Step  1 with $w$ replaced with $w_{\varepsilon, k-1}$, $q$ with $q_{k-1}$, and $q_1$ with $q_k$.
Iterating the above estimate, we conclude that
$$
	\| w_{\varepsilon , m  }  \|_{L_2 (\tQ_{  2^{  -2m  }   } ) } \le N (d, \delta, q, \varepsilon) \|w\|_{  L_q  (\tQ_{1}  ) }.
$$
By this and the argument of Step 1 again,
\begin{equation}
			\label{4.15.1}
	\| w_{\varepsilon , m +1 }  \|_{ \cW_2 (\tQ_{  2^{-2 (m+1) }   } ) } \le N (d, \delta, q, \varepsilon) \|w\|_{  L_q  (\tQ_{1}  ) }.
\end{equation}
Shifting the center of coordinates and using Lemma \ref{lemma 4.1} give
\begin{equation}
			\label{4.15.2}
	\|w_{\varepsilon , m+1 }\|_{L_{2} (\tQ_{ 2^{-2(m + 1)}  } (z)) } \le N (d, q, \delta, \varepsilon) \|w\|_{  L_q (\tQ_{1} (z)) }, \, \, \forall z \in \bR^{1+2d}.
\end{equation}

Next, by the argument of Lemma 21 of \cite{BCLP_10}, there exists a sequence of points $z_n \in \bR^{1+2d}, n \ge 1$,
such that
$$
	\bigcup_{n = 1}^{\infty} \tQ_{ 2^{ -2 (m+1)}} (z_n) = \bR^{1+2d}, \quad \sum_{n = 1}^{\infty} 1_{ \tQ_{1} (z_n)  } \le N_0 (d, m).
$$
Then, by this,  \eqref{4.15.2},
and the inequality
$$
	\sum_{k = 1}^{\infty} b_k^{\alpha} \le \Big(\sum_{  k = 1}^{\infty} b_k\Big)^{\alpha}, \, \,  \alpha \ge 1, \, b_k \ge 0,\, k \ge 1
$$
with $\alpha =  2/q > 1$,
we obtain
$$
	\int_{\bR^{1+2d}} |w_{\varepsilon , m+1}|^2 \, dz \le \sum_{ n  = 1 }^{\infty}  \int_{\tQ_{2^{  - 2 (m+1)}} (z_n)} |w_{\varepsilon , m+1 }|^2 \, dz
$$
 $$
	\le N \sum_{ n =  1}^{\infty}  \Big(\int_{\tQ_{1} (z_n)} |w|^q \, dz\Big)^{2/q}
		\le  N   \| w\|_{ L_q (\bR^{1+2d})  }^{2} < \infty.
 $$
Furthermore,  generalizing \eqref{4.15.1} for double cylinders $\tQ_r$ and $\tQ_{2^{-2(m+1)} r}$, we show that   $w_{\varepsilon , m+1} \in S_{2, \text{loc}} (\bR^{1+2d})$.
Therefore, by Lemma \ref{lemma 4.2} $(i)$, for any $r > 0$,
$$
	\|D_v w_{\varepsilon , m+1  }\|_{ L_2 (Q_r)  } \le N (d, \delta) r^{-1} \|w_{\varepsilon , m+1 }\|_{ L_2 (\bR^{1+2d}) }.
$$
Passing to the limit as $r \to \infty$ gives $D_v w_{\varepsilon , m+1 } \equiv 0$ in $\bR^{1+2d}_0$. Shifting in the $t$ variable, we prove that $D_v w_{\varepsilon , m+1}  \equiv 0$ in $\bR^{1+2d}$, and hence, $w_{\varepsilon , m+1} = 0$, which implies $w \equiv 0$.
\end{proof}

\begin{proof}[Proof of Theorem \ref{theorem 4.10}]
 We  consider three cases:
  $$p > 2, \quad
 p \in  (1, 2), T =  \infty, \quad \text{and}\,\,
p \in (1, 2),\, T < \infty.
$$

\textit{Case 1: $p > 2$.} Thanks to Corollary \ref{corollary 4.14}, it remains to prove the assertion $(ii)$. The latter follows from the a priori estimate in Corollary \ref{corollary 4.14} and the denseness result (see Lemma \ref{lemma 4.15}).

\textit{Case 2: $p \in (1, 2)$, $T = \infty$.}
$(i)$ We use the standard duality argument (see, for example, Theorem 4.3.8 of \cite{Kr_08}).
Throughout the proof, we assume that $N = N (d, p)$.

Let $U \in C^{\infty}_0 (\bR^{1+2d})$ and $u \in S_{p} (\bR^{1+2d})$.
Denote
$$
	f = P_0 u + \lambda u.
$$
For $h \in L_{1, \text{loc}} (\bR^{1+2d})$,
by $h_\varepsilon$ we denote the mollification  in the $x$ variable with the standard mollifier.

\textit{Estimate of $(-\Delta_x)^{1/3} u$}.
It is well known that
\begin{align*}
&(-\Delta_x)^{1/3} U,\,(-\Delta_x)^{1/3} \partial_t U,\,(-\Delta_x)^{1/3} (v_iD_{x_i}U),\,
(-\Delta_x)^{1/3} D_{v}^2 U\\
& \in C^{\infty}_{\text{loc}} (\bR^{1+2d}) \cap L_1  (\bR^{1+2d}).
\end{align*}
Then, by using duality and integrating by parts, we get
\begin{align*}
	J &:= \int  ((-\Delta_x)^{1/3} u_{ \varepsilon })   (-\partial_t U + v_i D_{x_i} U - a^{i j} D_{v_i v_j} U  + \lambda U) \, dz\\
&	=
	 \int u_{\varepsilon} (-\partial_t + v_i D_{x_i}  - a^{i j} D_{v_i v_j}   + \lambda) ((-\Delta_x)^{1/3}  U) \, dz\\
&	 =  \int  ((-\Delta_x)^{1/3} U) (P_0 u_{ \varepsilon }  + \lambda u_{ \varepsilon }) \, dz.
\end{align*}
Furthermore, by H\"older's inequality, Corollary \ref{corollary 4.14}, and the change of variables $t \to -t, x \to -x$, one has
\begin{align*}
	|J| &\leq \| (-\Delta_x)^{1/3} U\|_{  L_{q} (\bR^{1+2d}) }   \|f_{ \varepsilon }\|_{ L_p (\bR^{1+2d}) }\\
& 	\leq N  \delta^{-\theta} \|-\partial_t U
	+ v_i D_{x_i} U
	- a^{i j} D_{v_i v_j} U + \lambda U\|_{  L_{q} (\bR^{1+2d}) }    \|f_{ \varepsilon }\|_{ L_p (\bR^{1+2d}) },
\end{align*}
 where $q = p/(p-1)$.
 By Lemma \ref{lemma 4.15} and the same change of variables, we conclude that   $(-\partial_t- a^{i j} D_{v_i v_j} + v_i D_{x_i} + \lambda ) C^{\infty}_0 (\bR^{1+2d})$ is dense   in $L_q (\bR^{1+2d})$. Thus, we obtain
$$
	\| (-\Delta_x)^{1/3} u_\varepsilon \|_{ L_p (\bR^{1+2d}) } \leq N  \delta^{-\theta} \|f_{\varepsilon }\|_{ L_p (\bR^{1+2d}) }.
$$
Taking the limit as $\varepsilon \to 0$, we prove the desired estimate.

\textit{Estimate of $D_v (-\Delta_x)^{1/6} u$}.
As above, by using duality and integration by parts, we get
\begin{align*}
    \cI &= \int  (D_v (-\Delta_x)^{1/6} u_{\varepsilon}) (-\partial_t U
     + v \cdot D_{x} U
     - a^{i j} (t) D_{v_i v_j} U   + \lambda U) \, dz\\
&     =  - \int (P_0 u_{ \varepsilon }   + \lambda u_{\varepsilon})
     D_v (-\Delta_x)^{1/6} U \, dz\\
&\quad - \int  ((-\Delta_x)^{1/6} u_{ \varepsilon }) \,  D_x U \, dz =: \cI_1 + \cI_2.
\end{align*}
 Furthermore, by Corollary \ref{corollary 4.14}
 with $q > 2$,
 $$
    \cI_1 \le N  \delta^{-\theta} \|-\partial_t U
    + v_i D_{x_i} U
    - a^{i j} D_{v_i v_j} U  + \lambda U\|_{  L_{q} (\bR^{1+2d}) }    \|f_{ \varepsilon }\|_{ L_p (\bR^{1+2d}) }.
 $$
 Next, by H\"older's inequality,
 $$
    \cI_2 \le N  \|(-\Delta_x)^{1/3} u_{ \varepsilon }\|_{  L_{p} (\bR^{1+2d}) }
    \| \cR_x  (-\Delta_x)^{1/3} U\|_{ L_{q} (\bR^{1+2d}) }
    = :\cI_{2, 1} \cI_{2, 2},
  $$
 where $\cR_x = D_x (-\Delta_x)^{-1/2}$ is the Riesz transform.
 Due to the $L_p$ estimate of $(-\Delta_x)^{1/3} u_{\varepsilon}$,
 $$
    \cI_{2, 1}
    \le N \delta^{-\theta}
    \|f_{ \varepsilon }\|_{ L_{ p } (\bR^{1+2d}) }.
 $$
 By the $L_{q}$ boundedness of the Riesz
 transform and Corollary \ref{corollary 4.14} combined with the change of variables as before, we get
 $$
    \cI_{2, 2} \le N \delta^{-\theta} \|-\partial_t U
    + v_i D_{x_i} U
    - a^{i j} D_{v_i v_j} U  + \lambda U\|_{  L_{q} (\bR^{1+2d}) }.
 $$
These estimates imply the desired inequality for $D_v (-\Delta_x)^{1/6} u$.

\textit{Estimate of $D^2_{v} u$.}
For any $k, l \in \{1,  \ldots, d\}$, we have
$$
	I = \int D_{v_k v_l} u_{ \varepsilon } (-\partial_t U
	+ v_i D_{x_i} U
	- a^{i j} D_{v_i v_j} U  + \lambda U) \, dz  = : I_1 + I_2,
$$
where
\begin{align*}
 	I_1 &=  	\int D_{v_k v_l} U (P_0 u_{ \varepsilon } + \lambda u_{ \varepsilon }) \, dz,\\
	I_2 &=  - \int (\delta_{ik}D_{v_l} U
+ \delta_{il} D_{v_k} U ) D_{x_i} u_{ \varepsilon } \, dz.
\end{align*}
 We only need to show that
\begin{equation}
                \label{4.10.8}
    \begin{aligned}
	&|I_1| + |I_2|\\
	&\leq   N \delta^{-\theta} \|-\partial_t U + v_i D_{x_i} U
	- a^{i j} D_{v_i v_j} U  + \lambda U\|_{  L_{q} (\bR^{1+2d})} \|f_{ \varepsilon }\|_{ L_p (\bR^{1+2d}) }.
 \end{aligned}	
\end{equation}
 The estimate of  $I_1$
 follows from Corollary \ref{corollary 4.14}.
 Furthermore,  H\"older's inequality yields
$$
	|I_2| \leq  \|(-\Delta_x)^{1/6} D_{v } U \|_{  L_{q} (\bR^{1+2d}) }   \| \cR_{x} (-\Delta_x)^{1/3}  u_{ \varepsilon } \|_{ L_p (\bR^{1+2d}) } =: I_{2, 1} I_{2, 2}.
$$
By Corollary \ref{corollary 4.14},
$$
	I_{2, 1}
	\leq N  \delta^{-\theta} \|-\partial_t U + v_i D_{x_i} U - a^{i j} D_{v_i v_j} U  + \lambda U\|_{  L_{q} (\bR^{1+2d}) }.
 $$
 Then,  by the $L_p$-boundedness of $\cR_{x}$ and the $L_p$ estimate
 of $(-\Delta_x)^{1/3} u_{\varepsilon}$,
   we obtain
$$
	I_{2, 2} \leq N  \delta^{-\theta} \|f_{ \varepsilon }\|_{ L_p (\bR^{1+2d}) }.
 $$
Combining these estimates, we prove \eqref{4.10.8}.

\textit{Estimate of $u$.}
For $\lambda > 0$,
integration by parts gives
$$
	 \int \lambda  u_{ \varepsilon }   (-\partial_t U + v_i D_{x_i} U - a^{i j} D_{v_i v_j} U  + \lambda U) \, dz
=  \int  \lambda U f_{ \varepsilon }  \, dz.
$$
Now the $L_p$ estimate of $\lambda u$ follows from H\"older's inequality and Corollary \ref{corollary 4.14}. Thus, the assertion $(i)$ is proved.
The claim $(ii)$ now follows from $(i)$ and Lemma \ref{lemma 4.15}.

\textit{Case 3: $p \in (1, 2)$, $T  < \infty$.} $(i)$ We fix an arbitrary function
$\phi \in  L_q (\bR^{1+2d}_T)$,  where $q = p/(p-1)$, and  extend it by zero for $t > T$.
By the assertion $(ii)$ in the case $p > 2$  and the change of variables $t \to -t, x \to -x$, the equation
$$
    -\partial_t U  + v \cdot D_x U -  a^{i j}  D_{v_i v_j} U
    + \lambda U = \phi
$$
has a unique solution $U$ such that $U$, $(-\partial_t + v \cdot D_x) U$, $D_v U$,  $D^2_{v} U \in L_q (\bR^{1+2d})$. Note that by Lemma \ref{lemma 4.13} and the aforementioned change of variables, $U = 0$ a.e. on $(T, \infty) \times \bR^{2d}$.

Furthermore,
for a measurable function $h$ on $\bR^{1+2d}$, we denote
$$
    T h (z) = h (t, x - v t, v).
$$
By  the change of variables $x \to x-vt$ and the chain rule for distributions, we have
\begin{align*}
   \int u (-\partial_t U + v \cdot D_x U) \, dz
    &= - \int T u \, T (\partial_t U  - v \cdot D_x U) \, dz\\
    &= - \int T u \,  \partial_t (T U) \, dz.
\end{align*}
Since $T u ,  \partial_t (T u) \in L_p (\bR^{1+2d})$,  $TU, \partial_t (T U) \in L_q (\bR^{1+2d})$, we may integrate by parts and obtain
\begin{align*}
&     \int u (-\partial_t U + v \cdot D_x U) \, dz =  \int (\partial_t T u) T U \, dz\\
&     =   \int T (\partial_t u - v \cdot D_x u) T U \, dz
         =  \int (\partial_t u - v \cdot D_x u)  U \, dz.
\end{align*}
By this and integration by parts, we get
\begin{align*}
   \mathfrak{I} : &=  \int_{\bR^{1+2d}_T} u \phi \, dz
   =   \int_{\bR^{1+2d}} u (-\partial_t U + v \cdot D_x U  - a^{i j} (t) D_{v_i v_j} U + \lambda U) \, dz\\
   & =  \int_{\bR^{1+2d}_T} U (P_0 u + \lambda u) \, dz.
\end{align*}
By H\"older's inequality and Lemma \ref{lemma 4.13},
\begin{align*}
&    |\mathfrak{I}|
    \le N \delta^{-\theta} \lambda^{-1}
    \|-\partial_t U + v \cdot D_x U  - a^{i j} (t) D_{v_i v_j} U + \lambda U\|_{ L_{q} ( \bR^{1+2d})}\\
&\quad   \times  \|P_0 u + \lambda u\|_{ L_p (\bR^{1+2d}_T)}
    =  N \lambda^{-1} \delta^{-\theta}\|\phi\|_{ L_{q} ( \bR^{1+2d}_T)}   \|P_0 u + \lambda u\|_{ L_p (\bR^{1+2d}_T)}.
\end{align*}
This gives
\begin{equation}
                \label{4.10.10}
    \lambda \|u\|_{ L_p (\bR^{1+2d}_T)}
    \le N \delta^{-\theta}  \|P_0 u + \lambda u\|_{ L_p (\bR^{1+2d}_T)}.
    \end{equation}

Next, by the solvability when $T=\infty$,   the equation
$$
   P_0 u_1  + \lambda u_1 = (P_0 u + \lambda u) 1_{t < T}
$$
has a unique solution $u_1 \in S_p (\bR^{1+2d})$.
Then, by \eqref{eq2.38} with $T = \infty$,
\begin{align*}
&   \lambda \|u_1\|_{ L_p (\bR^{1+2d}_T)} +
\lambda^{1/2} \|D_v u_1\|_{ L_p (\bR^{1+2d}_T)}
+\|D^2_v u_1\|_{ L_p (\bR^{1+2d}_T)} \\
&\quad + \|(-\Delta_x)^{1/3} u_1\|_{ L_p (\bR^{1+2d}_T)} + \|D_v(- \Delta_x)^{1/6} u_1 \|_{ L_p  (\bR^{1+2d}_T) }  \\
&\le N   \delta^{-\theta} \|P_0 u + \lambda u\|_{ L_p (\bR^{1+2d}_T)}.
\end{align*}
It follows from \eqref{4.10.10}
that $u_1 = u$ a.e. in  $\bR^{1+2d}_T$.
Thus, the desired estimate holds for $u$.

$(ii)$ Similar to the proof of Theorem \ref{theorem 4.1} $(ii)$, this claim follows from the a priori estimate in the assertion $(i)$  when $p \in (1, 2), T  < \infty$
and the solvability in the assertion $(ii)$  when $p \in (1, 2), T = \infty$.
\end{proof}

\section{Mixed-norm estimate for the model equation}
            \label{sec7}
In this section, we prove the following theorem, which is Theorem \ref{theorem 2.1} for the operator $P_0$.
We follow the argument of Theorem
\ref{theorem 4.10}
and make only minor adjustments.

\begin{theorem}
                \label{theorem 6.1}
Invoke the assumptions of Theorem \ref{theorem 2.1} and assume $b \equiv 0, c  \equiv 0$.
Then, for any number $\lambda \ge 0$
and
$u \in S_{p, r_1, \ldots, r_d, q} (\bR^{1+2d}_T, w)$, one has
\begin{align}
                \label{6.1.1}
  & \lambda \|u\|+
 \lambda^{1/2} \|D_v u\|
 + \|D^2_v u\|+\|(-\Delta_x)^{1/3} u\|
+\|D_v (-\Delta_x)^{1/6} u\|\\
&    \le N \delta^{-\theta} \|P_0 u+\lambda u\|\notag,
\end{align}
where
$\|\cdot\| = \|\cdot\|_{L_{p, r_1, \ldots, r_d, q} (\bR^{1+2d}_T, w)}$, $N  = N (d, p, r_1, \ldots, r_d, q,K)$, and $\theta = \theta (d, p , r_1, \ldots, r_d, q, K) > 0$.
Furthermore,
the part of the Theorem \ref{theorem 2.1} $(iii)$ concerning the $S_{p; r_1, \ldots, r_d} (\bR^{1+2d}_T, |x|^{\alpha} \prod_{i = 1}^d w_i (v_i))$  estimate is valid with $P_0$ in place of $P$ for any $\lambda \ge 0$.
\end{theorem}

Here is a generalization of Lemma \ref{lem4.7}.
\begin{lemma}
            \label{lemma 5.5}
Let $p > 1, R\ge 1$ be  numbers
 and
 $f\in L_p (\bR^{1+2d}_0)$ vanish outside $(-1,0)\times \bR^d\times B_1$.
 Let $u\in S_p((-1,0)\times \bR^{2d})$ be the unique solution to
$$
    P_0 u = f,
	\quad
	u (-1, \cdot) = 0.
$$
Then, there exist positive constants $\theta_0 = \theta_0 (d)$, $\theta = \theta (d)$, $N = N (d, p)$ such that for $c_k = 2^k/\delta^{\theta_0}, k \ge 0$,
\begin{align*}
&\||u|+|D_vu|+|D_v^2 u|\|_{L_p((-1,0)\times B_{R^3}\times B_R)}\notag\\
&\le N  \delta^{-\theta}   \sum_{k=0}^\infty 2^{-k(k-1)/4}R^{-k}
\|f\|_{  L_p ( Q_{1, 2 c_k R}      ) }
,\\
&
(|(-\Delta_x)^{1/3} u|^p)^{1/p}_{Q_{1,R}}
\le N   \delta^{-\theta}   R^{-2} \sum_{k=0}^\infty 2^{-2k}   (|f|^p)^{1/p}_{ Q_{1,  c_k R}    }  ,\\
&
    (|D_v (-\Delta_x)^{1/6} u|^p)^{1/p}_{Q_{1, R}}
   \leq N  \delta^{-\theta}   R^{-1} \sum_{k=0}^\infty
   2^{- k }  (|f|^p)^{1/p}_{  Q_{ 1,  c_k R}  }.
\end{align*}
\end{lemma}

\begin{proof}
We repeat the proof of Lemma \ref{lem4.7}. Let us point out two minor modifications that we make.
First, in the definition of cutoff functions $\xi_j (x, v)$, we replace $\delta^2$ with $\delta^{ \widetilde \theta + 1 }$,
where $\widetilde \theta = \widetilde \theta (d)  > 1$ is the constant from the a priori estimate in Theorem \ref{theorem 4.10} $(i)$.
Second, one needs to use Theorem  \ref{theorem 4.10}  instead of  Theorem  \ref{theorem 4.1}.
\end{proof}

\begin{proposition}
			\label{theorem 5.4}
Let $p > 1$, $r > 0$, $\nu \geq 2$ be numbers,
 $z_0 \in  \overline{\bR^{1+2d}_T}$,
and
$u \in S_{p,\text{loc}} (\bR^{1+2d}_{ T }) $
be a function such that
$P_0 u= 0$ in
$(t_0-\nu^2 r^2, t_0) \times \bR^d \times B_{\nu r} (v_0)$.
Then, there exist positive constants $N (d,p)$ and $\theta = \theta (d)$
 such that
$$
   \bigg(|(-\Delta_x)^{1/3} u -  ((-\Delta_x)^{1/3}u )_{ Q_r (z_0) }|^p\bigg)^{1/p}_{ Q_r (z_0) }
\leq N \nu^{-1} \delta^{-\theta}
     \big(|(-\Delta_x)^{1/3} u|^p\big)^{1/p}_{  Q_{\nu r} (z_0) },
$$
\begin{align*}
&	 \bigg(|D_v (-\Delta_x)^{1/6} u -  (D_v(-\Delta_x)^{1/6} u)_{ Q_r (z_0) }|^p\bigg)^{1/p}_{ Q_r (z_0) }\\
&	 \leq
   N \nu^{-1} \delta^{-\theta}  (|D_v (-\Delta_x)^{1/6} u|^p)^{1/p}_{Q_{\nu r} (z_0)}\\
&\quad  +  N \nu^{-1} \delta^{-\theta} \sum_{k = 0}^{\infty} 2^{-2 k }
     (|(-\Delta_x)^{1/3} u|^p)^{1/p}_{ Q_{\nu r, 2^k \nu r}(z_0) },
\end{align*}
\begin{align*}
&   \bigg(|D^2_v u -  (D^2_v u)_{ Q_r (z_0) }|^p\bigg)^{1/p}_{ Q_r (z_0) }
   \leq N \nu^{-1} \delta^{-\theta}
(|D^2_v u|^p)^{1/p}_{   Q_{\nu r} (z_0) }\\
&\quad     +   N \nu^{-1} \delta^{-\theta}
    \sum_{k = 0}^{\infty} 2^{- k }
     (|(-\Delta_x)^{1/3} u|^{p})^{1/p}_{ Q_{\nu r, 2^k \nu r}(z_0) }.
\end{align*}
\end{proposition}

\subsection{Proof of Proposition \ref{theorem 5.4}}
First, we need a localized $L_p$ estimate, which we prove by repeating the argument of Lemma \ref{lemma 4.2}
and replacing
 Theorem \ref{theorem 4.1} with Theorem \ref{theorem 4.10}.

\begin{lemma}
					\label{lemma 5.10}
Lemma \ref{lemma 5.1} holds for any $p\in (1,\infty)$.
\end{lemma}

The next two lemmas are generalizations of
Lemmas \ref{lemma 4.3} and \ref{lemma 4.6}, respectively.
Their proofs go along the same lines as in the lemmas
 in Section \ref{section 4}.
One minor adjustment one needs to make
is to replace
Theorem \ref{theorem 4.1} and Lemma \ref{lemma 4.2}
 with Theorem \ref{theorem 4.10} and Lemma \ref{lemma 5.10}, respectively.

\begin{lemma}
            \label{lemma 5.2}
Let
$0 < r < R \leq 1$ and $p\in (1,\infty)$ be numbers,
and
$u\in S_{p,\text{loc}}(\bR^{1+2d}_0)$ be a function
such that $P_0 u = 0$ in $Q_1$.
Then, there exist constants
$N = N (d, p,r, R)$ and
$\theta = \theta (d) > 0$ such that
$$
    \|D_x u\|_{L_p (Q_r)}
    \leq N \delta^{-\theta} \|u\|_{L_p (Q_R)}.
$$
\end{lemma}

\begin{lemma}
            \label{lemma 5.3}
Let $p\in (1,\infty)$ and $u\in S_{p,\text{loc}}(\bR^{1+2d}_0)$
be a function such that
$P_0 u = 0$
in $(-1, 0) \times \bR^d \times B_1$.
Then, for any $r \in (0, 1)$, we have
\begin{align*}
	\| D_x u \|_{ L_p (Q_r)} &\leq N \delta^{-\theta}
	\sum_{k = 0}^{\infty}
	 2^{-k}
	  (|(-\Delta_x)^{1/3} u|^p)^{1/p}_{Q_{1,2^k}},\\
  \| D_x u \|_{ L_p (Q_r)} &\leq
    N \delta^{-\theta} \sum_{k = 0}^{\infty}
	  2^{- 2 k }	  (|(-\Delta_x)^{1/6} u|^p)_{ Q_{1, 2^k}}^{ 1/p   },
\end{align*}
where $N = N (d, p, r)$ and $\theta = \theta (d) > 0$.
\end{lemma}

\begin{lemma}
            \label{lemma 5.4}
Under the assumptions of Lemma \ref{lemma 5.3},
for any $l, m \in \{0, 1, \ldots\}$,
there exists a constant $\theta = \theta (d, p,  l, m)  > 0$ such that
the following assertions hold.

$(i)$ For any $R \in (1/2, 1]$,
$$ 	
		\sup_{ Q_{1/2} } |D_x^l D_v^{m}  u|
		+ \sup_{ Q_{1/2} } |\partial_t D_x^l D_v^{m}  u|
		\leq N (d, p, l, m, R)  \delta^{-\theta}\| u \|_{ L_p (Q_R) },
$$
\begin{align*}
   (ii) \, &\sup_{ Q_{1/2}}
    (|D^l_x D_v^{m+1} (-\Delta_x)^{1/6} u|
    + |\partial_t D^l_x D_v^{m+1}  (-\Delta_x)^{1/6} u|)\\
  &  \leq N (d,p, l, m)   \delta^{-\theta} \big( \|D_v (-\Delta_x)^{1/6} u\|_{L_p (Q_1)}\\
&\quad    +
	 \sum_{k = 0}^{\infty}
	 2^{-2k}(|(-\Delta_x)^{1/3} u|^p)^{1/p}_{ Q_{1, 2^k} }\big),
\end{align*}
	 $$
			(iii) \, \sup_{ Q_{1/2} } |D_x^l D_v^{m+2}  u|
			+
			\sup_{ Q_{1/2} } |\partial_t D_x^l D_v^{m+2}  u|
			\leq  N (d,  p, l, m)\delta^{-\theta}\| D^2_v u \|_{ L_p (Q_1)}
	 $$
	   $$
+	N(d,p, l, m)  \delta^{-\theta}\sum_{k = 0}^{\infty}
	 2^{-k}
	  (|(-\Delta_x)^{1/3} u|^p)^{1/p}_{Q_{1,2^k}}.
	 $$
\end{lemma}

\begin{proof}
$(i)$
The proof is almost identical to that
of Lemma \ref{lemma 4.4} $(i)$.
By the induction argument, for any $j \in \{0, 1\}$,
we get
$$
\|   \partial_t^j   D_x^l D_v^m u\|_{ L_p (Q_{1/2}) }     \leq N (d,p,j, l, m, R)  \delta^{-\theta}   \|u\|_{ L_p (Q_R) } ,
 $$
where $\theta = \theta (d, j, l, m) > 0.$
Now the assertion follows
from the last inequality, \eqref{eq7.41bb} and the Sobolev embedding theorem.

$(ii)$, $(iii)$ The proof is almost the same as the one of Lemma \ref{lemma 4.4} $(ii)$, $(iii)$.
 One merely needs to replace Lemma \ref{lemma 4.6} with Lemma \ref{lemma 5.3}
 in this argument.
\end{proof}

\begin{proof}[Proof of Proposition \ref{theorem 5.4}]
The assertion follows from Lemma \ref{lemma 5.4} and the scaling argument (see Lemma \ref{lemma 4.1}).
\end{proof}

\subsection{Proof of Theorem \ref{theorem 6.1}}

\begin{proposition}
                \label{theorem 5.6}
Let $p > 1$, $r > 0$, $\nu \geq 2$ be numbers,
 $z_0 \in \overline{\bR^{1+2d}_T}$,
and
$u \in S_{p} (\bR^{1+2d}_{T})$.

Then, there exist positive constants $N = N (d, p)$, $\theta_0 = \theta_0 (d)$, and $\theta = \theta (d, p)$ such that  for $c_k = 2^k/\delta^{\theta_0}, k \ge 0,$ and  any  $z_0 \in \overline{\bR^{1+2d}_T}$,
\begin{equation*}
\begin{aligned}
     &   \bigg(\big|(-\Delta_x)^{1/3} u -  ((-\Delta_x)^{1/3} u)_{ Q_r (z_0) }\big|^p\bigg)^{1/p}_{Q_r (z_0)}\\
     &\leq
       N \nu^{-1}  \delta^{-\theta}  \big( |(-\Delta_x)^{1/3} u|^p \big)^{1/p}_{ Q_{\nu r} (z_0) }\\
 &
    \quad + N  \nu^{(4d+2)/p}  \delta^{-\theta}
    \sum_{k=0}^\infty 2^{-2k}
	  (|P_0 u|^p)^{1/p}_{  Q_{2 \nu r, 2 \nu r c_k} (z_0)  },
\end{aligned}
 \end{equation*}
 \begin{equation*}
  \begin{aligned}
  &   \bigg(\big|D_v(-\Delta_x)^{1/6} u -  (D_v (-\Delta_x)^{1/6} u)_{ Q_r (z_0) }\big|^p\bigg)^{1/p}_{Q_r (z_0)}\\
&    \leq N \nu^{-1}   \delta^{-\theta}    (|D_v (-\Delta_x)^{1/6} u|^p)^{1/p}_{ Q_{\nu r} (z_0) }\\
&\quad  + N \nu^{-1}  \delta^{-\theta}   \sum_{k=0}^\infty 2^{-2k}(|(-\Delta_x)^{1/3}u|^p)^{1/p}_{Q_{\nu r,2^{k} \nu r } (z_0)}\\
&\quad +  N   \nu^{(4d+2)/p}   \delta^{-\theta}    \sum_{k =0}^\infty 2^{-k }
	 (|P_0 u|^p)^{1/p}_{   Q_{2 \nu r, 2 \nu r c_k} (z_0)   },
 \end{aligned}
  \end{equation*}
  \begin{equation*}
    \begin{aligned}
	&		(|D^2_v u - (D^2_v u)_{Q_r (z_0)}|^{p})^{1/p}_{Q_r (z_0)}
	\leq
       N \nu^{-1}  \delta^{-\theta}   \big(|D^2_v u|^p\big)^{1/p}_{ Q_{\nu r} (z_0) }
		 \\
&  \quad +
	  N  \nu^{-1}  \delta^{-\theta}
	 \sum_{k = 0}^{\infty}
	  2^{-k}	    (|(-\Delta_x)^{1/3} u|^p)^{1/p}_{Q_{\nu r,2^k \nu r}(z_0)}\\
&\quad + N  \nu^{(4d+2)/p}   \delta^{-\theta}
    \sum_{k=0}^\infty 2^{-k }  (|P_0 u|^p)^{1/p}_{ Q_{2 \nu r, 2 \nu r c_k}  (z_0)}.
\end{aligned}
\end{equation*}
\end{proposition}

\begin{proof}
The assertion is derived from Lemma \ref{lemma 5.5}
and Proposition \ref{theorem 5.4}.
See the proof of Proposition  \ref{theorem 4.7.2} in Subsection \ref{section 5}.
\end{proof}

\begin{proof}[Proof of Theorem  \ref{theorem 6.1}]
First, we consider the case when the weight depends on $t, v$ variables only.
In this proof, $N$ is a constant depending only on $d$, $p$, $r_1, \ldots, r_d$, $q$, and $K$.

\textit{Step  1: Case $\lambda =  0$ and $u$ vanishing for large $|z|$ .}
By Lemma \ref{lemma 9.7} and the self-improving property of the $A_p$-weights (see, for instance, Corollary 7.2.6 of \cite{G_14}),
there exists a number
$$
    p_0 = p_0 (d, p, r_1, \ldots, r_d, q,  K),\quad
    1 < p_0 < \min\{p, r_1, \ldots, r_d, q\},
$$
such that $u \in S_{p_0, \text{loc}} (\bR^{1+2d}_T)$
\begin{equation}
			\label{eq6.1.5}
w_0\in A_{q/p_0} (\bR),\quad w_i\in A_{r_i/p_0} (\bR),\quad i=1,\ldots,d.
\end{equation}

Denote $f = P_0 u$.
By what was said above and the fact that $u$ is compactly supported, we have $u \in S_{p_0} (\bR^{1+2d}_T)$. Then, by Proposition  \ref{theorem 5.6} with $p$ replaced with $p_0$, for any $z_0 \in \overline{\bR^{1+2d}_T}$,
 we get
\begin{align*}
    ((-\Delta_x)^{1/3} u)^{\#}_T (z_0) &\le N \nu^{-1} \delta^{-\theta}  \cM^{1/p_0}_{T} |(-\Delta_x)^{1/3} u|^{p_0} (z_0)\\
&\quad     + N \nu^{(4d+2)/p_0} \delta^{-\theta}
    \sum_{k = 0}^{\infty} 2^{ - 2 k }
    \bM^{1/p_0}_{2^k /\delta^{\theta_0} , T} |f|^{p_0} (z_0),
\end{align*}
\begin{align*}
(D_v (-\Delta_x)^{1/6} u)^{\#}_T (z_0)&\le N \nu^{-1} \delta^{-\theta}  \cM^{1/p_0}_{ T}
     |D_v  (-\Delta_x)^{1/6} u|^{p_0} (z_0)\\
&\quad      + N \nu^{-1} \delta^{-\theta}  \sum_{k = 0}^{\infty} 2^{- 2 k } \bM^{1/p_0}_{2^k, T}  |(-\Delta_x)^{1/3} u|^{p_0} (z_0)\\
&\quad          + N \nu^{(4d+2)/p_0}  \delta^{-\theta}  \sum_{k = 0}^{\infty} 2^{-k  }
         \bM^{1/p_0}_{2^{k}/\delta^{\theta_0} , T} |f|^{p_0} (z_0),
\end{align*}
and
\begin{align*}
(D^2_v u)^{\#}_T (z_0)&\le N \nu^{-1}\delta^{-\theta} \cM^{1/p_0}_{ T}
     |D_v^2 u|^{p_0} (z_0)\\
&\quad      + N \nu^{-1} \delta^{-\theta}  \sum_{k = 0}^{\infty} 2^{-k} \bM^{1/p_0}_{2^k, T}  |(-\Delta_x)^{1/3} u|^{p_0} (z_0)\\
&\quad          + N \nu^{(4d+2)/p_0} \delta^{-\theta}  \sum_{k = 0}^{\infty} 2^{-k  }
         \bM^{1/p_0}_{2^k /\delta^{\theta_0} , T} |f|^{p_0} (z_0).
\end{align*}
We take the $\|\,\cdot\,\|$-norm on both sides of the above inequalities
and use the Minkowski inequality.
After that we apply
 Corollary \ref{corollary 4.9} $(i)$ with
$p/p_0$, $r_1/p_0, \ldots, r_d/p_0$, $q/p_0 > 1$ and \eqref{eq6.1.5} combined with Corollary \ref{corollary 4.9} $(ii)$.
We obtain
\begin{align*}
    \|(-\Delta_x)^{\frac 1 3} u\|
&    \le N \nu^{-1} \delta^{-\theta}  \|(-\Delta_x)^{\frac 1 3} u\|
    + N \nu^{(4d+2)/p_0} \delta^{-\theta}  \|f\|,\\
    \|D_v (-\Delta_x)^{\frac 1 6} u\| &\le
     N \nu^{-1} \delta^{-\theta}  (\|D_v (-\Delta_x)^{\frac 1 6} u\|
    +    \|(-\Delta_x)^{\frac 1 3} u\|)\\
&\quad    + N \nu^{(4d+2)/p_0}  \delta^{-\theta}  \|f\|,\\
    \|D^2_v u\| &\le
     N \nu^{-1}  \delta^{-\theta}  (\|D_v^2 u\|
    +
    \|(-\Delta_x)^{\frac 1 3} u\|)
    + N \nu^{(4d+2)/p_0} \delta^{-\theta}  \|f\|.
\end{align*}
 Taking $\nu = 2 + 2N \delta^{-\theta} $, we prove  the desired estimate.

\textit{ Step  2: Case $\lambda = 0$ and arbitrary $u \in S_{p, r_1, \ldots, r_d, q} (\bR^{1+2d}_T, w)$.}
Let $\phi \in C^{\infty}_0 (\bR^{1+2d})$ be a function such that $\phi = 1$ on $\tQ_{1 }$
and denote for $n \ge 1$,
$$
	\phi_n  (z)  = \phi (t/n^2, x/n^3, v/n), \quad u_n  = u \phi_n.
$$
 By the result of Step 1,
\begin{equation}
			\label{eq6.1.1}
	\|D^2_v u_n\| + \|(-\Delta_x)^{1/3} u_n\| + \|D_v (-\Delta_x)^{1/6} u_n\| \le N \delta^{-\theta} \|P_0 u_n\|.
\end{equation}

\textit{Estimate of $D^2_v u$.}
By \eqref{eq6.1.1},
$$
	\|D_v^2 u\|_{  L_{p, r_1, \ldots, r_d, q} (\tQ_n \cap \bR^{1+2d}_T, w)   }   \le N  \delta^{-\theta} (\|P_0 u\| + A_1 + A_2  + A_3),
$$
where
 $$
  	A_1  =  \|P_0 u  (\phi_n  -  1)\|, \quad A_2  =   \|(P_0 \phi_n) u\|, \quad A_3 =  2\|(a D_v u) \cdot D_v \phi_n\|.
$$
Since $u \in S_{p, r_1, \ldots, r_d, q} (\bR^{1+2d}, w )$, by the dominated convergence theorem,
\begin{equation}
			\label{eq6.1.6}
	 A_1    \to 0 \quad \text{as} \, \, n \to \infty.
\end{equation}
Furthermore,
\begin{equation}
			\label{eq6.1.7}
	A_2 + A_3 \le N \delta^{-1} n^{-1} (\|u\| + \|D_v u\|) \to 0 \quad \text{as} \, \, n \to \infty.
\end{equation}
Thus, the estimate for $D^2_v u$ holds.

\textit{Estimate of $(-\Delta_x)^{1/3 } u$ and $D_v (-\Delta_x)^{1/6 } u$.}
We use the duality argument as in the proof of Theorem \ref{theorem 4.1} $(i)$.
It follows from \eqref{eq6.1.1}-\eqref{eq6.1.7}  that
\begin{equation}
		\label{eq6.1.8}
	\|(-\Delta_x)^{1/3} u_n\| + \| D_v (-\Delta_x)^{1/6} u_n\| \le  N  \delta^{-\theta} (\|P_0 u\| + n^{-1} \||u| + |D_v u|\|).
\end{equation}
Next, for any $\eta \in C^{\infty}_0 (\bR^{1+2d}_T)$, we have
\begin{equation}
			\label{eq6.1.10}
	|(-\Delta_x)^{1/3} \eta| (z) \le N (d) (1+|x|)^{-d - 2/3},
\end{equation}
and by this  $(-\Delta_x)^{1/3} \eta \in L_{p^{*}, r_1^{*}, \ldots, r_d^{*}, q^{*}} (\bR^{1+2d}_T,  w^{*} )$ of $p,r_1,\ldots,r_d,q$ respectively,
where $p^{*}, r_{1}^{*} \ldots, r_{d}^{ *}, q^{*}$ are H\"older conjugates, and
$$
	w^{*} := w_0^{-1/(q-1)} (t) \prod_{i = 1}^d w_i^{-1/(r_i - 1) } (v_i).
$$
Then, since $u_n \to u$  in $L_{p, r_1, \ldots, r_d, q} (\bR^{1+2d}_T, w)$, we have
 $$
	\bigg| \int_{\bR^{1+2d}_T}	\eta (-\Delta_x)^{1/3} u   \, dz \bigg| \le  \|\eta\|_{L_{p^{*}, r_1^{*}, \ldots, r_d^{*}, q^{*}} (\bR^{1+2d}_T,  w^{*} ) }\nlimsup_{n \to \infty}   \|(-\Delta_x)^{1/3} u_n\|.
 $$
This combined with  \eqref{eq6.1.8} implies \eqref{6.1.1} for $(-\Delta_x)^{1/3} u$. In the same way, we prove the estimate for $D_v (-\Delta_x)^{1/6} u$.

\textit{ Step  3: Case $\lambda > 0$.}
We set $r_{d+1} = r_d$ and $w_{d+1}\equiv 1$.
Repeating the proof of Lemma \ref{lemma 4.13}
with
the spaces $L_p (\bR^{1+2d}_T)$ and $L_p (\bR^{3+2d}_T)$ replaced with
$$
L_{p, r_1, \ldots, r_d, q} (\bR^{1+2d}_T, w)\quad
\text{and}\quad
L_{p, r_1, \ldots, r_{d+1}, q} (\bR^{3+2d}_T, w),
$$
respectively,
we conclude that for $\lambda \ge \lambda_0 = 16 N^2  \delta^{-2\theta} + 1 > 0$,
$$
    \lambda \|u\|\leq N  \delta^{-\theta}  \|P_0 u + \lambda u\|.
$$
To show that the desired estimate holds for $\lambda \ge 0$,  we  use a scaling argument  (see Lemma \ref{lemma 4.1})
and the fact that the map $x \to \lambda x$ preserves the $A_p$ constant.
This combined with
\eqref{6.1.1} with $\lambda = 0$
proves the desired estimate.

 For the estimate in $S_{p; r_1, \ldots, r_d} (\bR^{1+2d}_T, |x|^\alpha\prod_{i = 1}^d w_i (v_i))$ (see p. \pageref{eq2.5}), the proof goes along the lines of the above argument. Let us point out one minor modification: due to \eqref{eq6.1.10} and the fact that $\alpha \in (-1, p-1)$,
 one has
$$
	(-\Delta_x)^{1/3} \eta \in   L_{p^{*}; r_1^{*}, \ldots, r_d^{*}} (\bR^{1+2d}_T, |x|^{-\alpha/(p-1)} \prod_{i = 1}^d w_i^{-1/(r_i - 1)} (v_i))
$$
for any $\eta \in C^{\infty}_0 (\bR^{1+2d}_T)$.
\end{proof}

\section{Proof of Theorem \ref{theorem 2.1}}
First, we prove a  key lemma analogous to Proposition \ref{theorem 5.6}, which will imply the a priori estimate of $D^2_v u$ for $u \in S_{p, r_1, \ldots, r_d, q} (\bR^{1+2d}_T, w).$

\begin{lemma}
            \label{lemma 6.1}
Let $\gamma_0 > 0$, $\nu \ge 2$,  $p_0 \in (1, \infty)$, $\alpha \in (1,  3/2  )$  be numbers, $T \in (-\infty, \infty]$, $R_0$ be the constant from
Assumption \ref{assumption 2.2} $(\gamma_0)$,
 and $u \in S_{p_0} (\bR^{1+2d}_T)$. Then, under Assumptions \ref{assumption 2.1} -  \ref{assumption 2.2} $(\gamma_0)$,
there exist positive constants $\theta_0 = \theta_0 (d)$, $\theta = \theta (d, p_0)$, $N  = N (d, p_0, \alpha)$,  and a sequence of positive numbers $\{a_{k  }, k   \ge  0 \}$ such  that
$$
	\sum_{k = 0 }^{\infty} a_k \leq N,
$$
and for  $c_k = 2^k/\delta^{\theta_0}$ and any $z \in \overline{\bR^{1+2d}_T}$,
and $r \in (0, R_0/(4\nu))$,
 \begin{equation}
                \label{6.1.2}
     \begin{aligned}
   &	(|D^2_v u - (D^2_v u)_{Q_r (z)}|^{p_0})^{1/p_0}_{Q_r (z)} \leq N \nu^{-1}  \delta^{-\theta}  (|D^2_v u|^{p_0})^{1/p_0}_{ Q_{\nu r} (z) } \\
&\quad
    + N  \nu^{-1}  \delta^{-\theta} \sum_{k=0}^\infty 2^{-k}(|(-\Delta_x)^{1/3} u|^{p_0})^{1/p_0}_{Q_{\nu r,2^{k} \nu r } (z) }\\
   &\quad + N  \nu^{(2+4d)/p_0} \delta^{-\theta} \sum_{k=0}^\infty 2^{-k}  (|P u|^{p_0})^{1/p_0}_{Q_{2\nu r, 2 \nu r c_k } (z) }   \\
	&\quad	+ N   \nu^{(2+4d)/p_0} \delta^{-\theta}
	\gamma_0^{(\alpha - 1)/(p_0 \alpha)}
	\sum_{ k =0 }^\infty a_{k }(|D^2_v u|^{p_0 \alpha})^{1/(p_0 \alpha)}_{Q_{2\nu r, 2 \nu r c_k} (z)} .
	\end{aligned}
\end{equation}
\end{lemma}

To prove the above lemma we need the following result.
\begin{lemma}
			\label{lemma 9.10}
Let $\gamma_0 > 0$ be a number and $R_0$ be the constant from Assumption \ref{assumption 2.2} $(\gamma_0)$. Let $r \in (0, R_0/2)$, $c > 0$ be numbers.
 Then, one has
$$
	I: = \fint_{  Q_{r, c r}  } |a (t, x, v) - (a (t, \cdot, \cdot))_{  B_{r^3} \times B_r  }| \, dz  \le N (d) c^3 \gamma_0.
$$
\end{lemma}

\begin{proof}
Let $X \subset  B_{ (c r)^3  }  $ be a finite set such that
  $\{B_{ r^3/2} (x), x \in X\}$
  is a maximal family of disjoint balls.
  Then, since
  $
   B_{(c r)^3} \subset \cup_{x \in X} B_{ r^3 } (x)
  $,   we have
\begin{align*}
    I &\le |Q_{ r, c r }  |^{-1}
    \sum_{x \in X}
     \int_{ Q_{r}(0,x,0)  }
      |a (z) -  (a (t, \cdot, \cdot))_{   B_{ r^3} \times B_{ r }   }| \, dz\\
&    \le |Q_{r, c r}|^{-1}  \sum_{x \in X} (B (x) + C (x)),
\end{align*}
where
\begin{align*}
    B (x)   & =
     \int_{Q_{r}(0,x,0)}
      |a (z) -  (a (t, \cdot, \cdot))_{   B_{ r^3} (x) \times B_{ r}   }| \,
      dz,\\
    C (x) & =
   \int_{Q_{ r}(0,x,0)}
      |(a (t, \cdot, \cdot))_{   B_{ r^3} (x) \times B_{ r}   } -  (a (t, \cdot, \cdot))_{   B_{ r^3}  \times B_{ r}   }| \,
      dz.
\end{align*}
By the facts that
$
	D_{r} ((0, x, 0), t) = B_{ r^3}  (x) \times B_{r}
$
and $r < R_0$,
and  Assumption \ref{assumption 2.2} $(\gamma_0)$, we have
 \begin{equation}
			\label{6.1.5}
    B (x) \le |Q_{ r}| \gamma_0.
 \end{equation}
Furthermore, for any $x \in X$
such that $x \neq 0$,
let
$x_j, j = 0, 1, \ldots, m,$
be a sequence of points
such that $x_0 = 0$, $x_m = x$, and $|x_j-x_{j+1}|\le r^3$ for $j=0,\ldots,m-1$, where $m\le N(d) c^3$.
Then, by the triangle inequality, we have
\begin{align*}
	C (x)   & \le \sum_{j = 0}^{m-1}
\int_{ Q_{ r}(0,x,0) }      |(a (t, \cdot, \cdot))_{   B_{ r^3} (x_{j+1}) \times B_{ r}   } -  (a (t, \cdot, \cdot))_{   B_{ r^3} (x_j) \times B_{ r}   }| \,dz\\
&      \le       \sum_{j = 0}^{m-1}
      	\int_{ Q_{ r}(0,x,0) }
       |(a (t, \cdot, \cdot))_{   B_{ r^3} (x_{j+1}) \times B_{ r}   } -  (a (t, \cdot, \cdot))_{   B_{8 r^3} (x_j) \times B_{2 r}   }|\, dz\\
&\quad     +
       \sum_{j = 0}^{m-1}
      	\int_{ Q_{ r}(0,x,0)}
       |(a (t, \cdot, \cdot))_{   B_{ 8  r^3} (x_j) \times B_{ 2  r}   } -  (a (t, \cdot, \cdot))_{   B_{ r^3} (x_j) \times B_{ r}   }| \, dz.
\end{align*}
It is easy to check that for any two sets $A \subset A' \subset \bR^d$ of positive finite Lebesgue measure  and any
$f \in L_{1, \text{loc}} (\bR^d)$, one has
$$
	|(f)_{A}  - (f)_{A'}| \leq  \frac{|A'|}{|A|} (|f - (f)_{A'}|)_{A'}.
$$
By this, the fact that $2  r < R_0$,  and Assumption \ref{assumption 2.2} $(\gamma_0)$,
\begin{align*}
    C (x) &\le |Q_{ r}| N (d)
  \sum_{j=0}^{m-1}
    	\fint_{ Q_{2 r}(0,x,0) }
    	|a (z)  - (a (t, \cdot, \cdot))_{   B_{ 8  r^3} (x_j) \times B_{2 r}   }| \, dz\\
&    \le  N (d) |Q_{ r}|  m
    \gamma_0
    \le N (d) |Q_{ r}| c^{3}    \gamma_0.
\end{align*}
Combining  this  with \eqref{6.1.5} and using the fact that $|X| \le N (d) c^{3 d }$,
 we prove the lemma.
\end{proof}

\begin{proof}[Proof of Lemma \ref{lemma 6.1}]
Clearly, we may assume that for any $k\ge 0$, $D^2_v u \in L_{p_0 \alpha} ( Q_{2\nu r,2 \nu r c_k} )$
because otherwise the sum on the  right-hand side of \eqref{6.1.2} is infinite, and the inequality holds trivially.
Furthermore, by Lemma \ref{lemma 4.1}, it suffices to prove \eqref{6.1.2} for $z = 0$.
We denote
$$
\bar a (t)
 = (a (t, \cdot, \cdot))_{  B_{\nu^3 r^3} \times B_{\nu r} }
\quad\text{and}\quad
 \bar P = \partial_t -  v \cdot D_x - \bar a^{i j} D_{v_i v_j}.
 $$

By Proposition \ref{theorem 5.6}, there exist positive numbers $\theta_0 = \theta_0 (d)$, $\theta = \theta (d, p)$, and $N = N (d, p_0)$ such that
\begin{equation}
			               \label{6.1.4}
  \begin{aligned}
	&		(|D^2_v u - (D^2_v u)_{Q_r}|^{p_0})^{1/p_0}_{Q_r }
	\leq
       N \nu^{-1}  \delta^{-\theta}  \big(|D^2_v u|^{p_0}\big)^{1/p_0}_{ Q_{\nu r}  }
		 \\
&  \quad +
	  N  \nu^{-1}  \delta^{-\theta}
	 \sum_{k = 0}^{\infty}
	  2^{-k}	    (|(-\Delta_x)^{1/3} u|^{p_0})^{1/{p_0}}_{Q_{\nu r,2^k \nu r}}\\
&\quad + N  \nu^{(4d+2)/p_0}   \delta^{-\theta}
    \sum_{k=0}^\infty 2^{-k}   (|P u|^{p_0})^{1/{p_0}}_{ Q_{2 \nu r, 2 \nu r c_k} }\\
&\quad + N  \nu^{(4d+2)/p_0}   \delta^{-\theta}
    \sum_{k=0}^\infty 2^{-k}   (|a - \bar a|^{p_0} |D^2_v u|^{p_0})^{1/{p_0}}_{  Q_{2 \nu r, 2 \nu r c_k}},
\end{aligned}
\end{equation}
where $c_k = 2^k/\delta^{\theta_0}$.

Fix any $k  \in \{0, 1, 2, \ldots\}$
and denote $\alpha_1 = \alpha/(\alpha-1) (>3)$.
Then,  by H\"older's inequality,
\begin{equation}
                \label{6.1.3}
\begin{aligned}
&	(|a - \bar a|^{p_0}|D^2_v u|^{p_0})^{1/p_0}_{ Q_{2 \nu r, 2 \nu r c_k}}\\
&	\leq  (|a -  \bar a|^{ p_0 \alpha_1}   )^{1/(p_0 \alpha_1)}_{    Q_{ 2 \nu r, 2 \nu r c_k}}  \,  (|D^2_v u|^{ p_0 \alpha})_{    Q_{ 2 \nu r, 2 \nu r c_k}}^{1/(p_0 \alpha)  }\\
&	 =: A_1^{1/(p_0 \alpha_1)} A_2^{1/(p_0 \alpha)}.
\end{aligned}
\end{equation}
Since $a$ is a bounded function, we have
$$
	A_1 \leq N  \delta^{-\theta}   \fint_{  Q_{ 2 \nu r, 2 \nu r c_k }      } |a -  \bar a|  \, dz.
$$
Then, by Lemma \ref{lemma 9.10} with $r$ replaced with $2\nu r$ and $c =  2^{k}/\delta^{\theta_0}$,
$$
	A_1 \le N (d) 2^{3 k } \delta^{- 3 \theta_0}  \gamma_0.
$$
 By this and \eqref{6.1.3}, we conclude
\begin{align*}
&	  2^{-k} (|a - \bar a|^{p_0} |D^2_v u|^{p_0})^{1/p_0}_{    Q_{2 \nu r, 2 \nu r c_k}}\\
& 	\leq  N  \delta^{-\theta}\gamma_0^{1/(p_0 \alpha_1)}
 	2^{- k  + 3 k/(p_0 \alpha_1) }   (|D_v^2 u|^{p_0 \alpha})^{1/(p_0 \alpha)}_{Q_{2 \nu r, 2 \nu r c_k}}.
\end{align*}
With $a_{k}:=2^{- k  + 3 k/(p_0 \alpha_1) }$, the series converges since $p_0  \alpha_1  > 3$.
Now the assertion follows from this and \eqref{6.1.4}.
\end{proof}

\begin{proof}[Proof of Theorem \ref{theorem 2.1} (i)]
We will focus on the case when the weight depends only on $t$ and $v$. The estimate in $S_{p; r_1, \ldots, r_d} (\bR^{1+2d}_T, |x|^{\alpha} \prod_{i = 1}^d w_i (v_i))$ is proved in the same way.

First, we consider the case $T = \infty$. We follow the idea of Section 5.1 of \cite{DK_18}.
 In the first step, we prove  a priori estimate for a function $u$ with a sufficiently small support
in the temporal variable. Then, we use partition of unity to handle an arbitrary function $u \in S_{p, r_1, \ldots, r_d, q} (\bR^{1+2d}_T, w)$.
Throughout the proof, we assume that $N  = N (d, p, r_1, \ldots, r_d, q, K, L)$.

\textit{Step 1}.
We show that there exists $\beta = \beta (d, p, r_1, \ldots, r_d, q, K)> 0$,
\begin{equation}
			\label{6.5.8}
R_1 =  \delta^{\beta} \widetilde R_1 (d, p, r_1, \ldots, r_d, q, K) > 0,\quad
\gamma_0 =   \delta^{\beta} \widetilde \gamma_0 (d, p, r_1, \ldots, r_d, q, K) > 0,
\end{equation}
$$
	\text{and} \quad	\theta = \theta (d, p, r_1, \ldots, r_d, q, K) >0,
$$
such that for any $t_0 \in \bR$ and $u \in S_{p, r_1, \ldots, r_d, q} (\bR^{1+2d}_T, w)$ vanishing outside $(t_0 - (R_0 R_1)^2, t_0) \times \bR^{2d}$,
$$
	\|D^2_v u\| + \|(-\Delta_x)^{1/3} u\| + \| D_v (-\Delta_x)^{1/6} u\| \le N (d, p, r_1, \ldots, r_d, q, K) \delta^{-\theta}    \|P u\|.
$$

By Lemma \ref{lemma 9.7}, there exists a number
$$
    p_0 = p_0 (d, p, r_1, \ldots, r_d, q,  K)
$$
such that
$$
    1 < p_0 < \min\{ p, r_1, \ldots, r_d, q\}
$$
and
$
    u \in S_{p_0, \text{loc}} (\bR^{1+2d}_T).
$
Since $u$ is assumed to be compactly supported, this gives $u \in S_{p_0} (\bR^{1+2d}_T)$.
By the self-improving property of the $A_p$-weights (see, for instance, Corollary 7.2.6 of \cite{G_14}),
we also fix a number $\alpha$ and $p_0$ further smaller such that
$$
   1 <	\alpha   <  \min\big\{    \frac{3}{2} ,
    \frac{p}{p_0}, \frac{r_1}{p_0}, \ldots, \frac{r_d}{p_0}, \frac{q}{p_0}\big\}
$$
and
\begin{equation}
                    \label{eq12.31}
w_0\in A_{\frac q {\alpha p_0}} (\bR) ,\quad w_i\in A_{\frac {r_i}{\alpha p_0}} (\bR),\,\, i=1,\ldots,d.
\end{equation}
Let $\nu \ge 2$, $\gamma_0$, and $R_1$ be some numbers which  will be chosen later.

If  $4 \nu r \ge R_0$, then by H\"older's inequality with $\alpha$ and $\alpha_1 = \alpha/(\alpha-1)$, for any $z \in \overline{\bR^{1+2d}_T}$  we have
\begin{align*}
		&(|D^2_v u - (D^2_v u)_{Q_r (z) }|^{p_0})^{1/p_0}_{Q_r (z) }
		\le 2 (|D^2_v u|^{p_0})^{1/p_0}_{Q_r (z) }\\
&		\le 2  (I_{ (t_0 - (R_0 R_1)^2, t_0)  })_{Q_r (z)}^{1/(p_0\alpha_1)} (|D^2_v u|^{p_0 \alpha})^{1/(p_0 \alpha)}_{Q_r (z) }\\
&		\le  2  (R_1  R_0 r^{-1})^{2/(p_0\alpha_1)}  \cM^{1/(p_0 \alpha)}_T |D^2_v  u|^{p_0 \alpha}(z)\\
&		\le N \nu^{2/(p_0\alpha_1)} R_1^{2/(p_0\alpha_1)}   \cM^{1/(p_0 \alpha)}_T |D^2_v  u|^{p_0 \alpha} (z).
\end{align*}
In the case when $ 4 \nu r < R_0$, we use Lemma \ref{lemma 6.1} with $p$ replaced with $p_0$. Combining these cases, we get in $\overline{\bR^{1+2d}_T}$,
   \begin{align*}
     (D^2_v u)^{\#}_T
	& \leq     N (\nu^{-1} \delta^{-\theta}  + \nu^{2/(p_0\alpha_1)} R_1^{2/(p_0\alpha_1)}) \cM^{1/p_0}_{T} |D^2_v u|^{p_0}\\
	&\quad + N  \nu^{-1 } \delta^{-\theta}  \sum_{k = 0}^{\infty}
	  2^{-k} \bM^{1/p_0}_{2^k, T} |(-\Delta_x)^{1/3} u|^{p_0}\\
	&\quad + N \nu^{(4d+2)/p_0} \delta^{-\theta}  \gamma_0^{ 1/(p_0 \alpha_1)}  \sum_{k = 0}^{\infty} a_{k  }
	\bM^{1/(p_0 \alpha)}_{2^{k  }/\delta^{\theta_0}  , T}  |D^2_v u|^{p_0 \alpha}\\
	&\quad + N  \nu^{(4d+2)/p_0} \delta^{-\theta}   \sum_{k = 0}^{\infty} 2^{-k  }    %
    \bM^{1/p_0}_{2^{k  }/\delta^{\theta_0}, T} |P u|^{p_0}.
   \end{align*}
We  take the $\|\cdot\|$-norm of both sides of this inequality, use Corollary \ref{corollary 4.9} $(ii)$ with $p$, $r_1, \ldots, r_d$, $q > 1$
and Corollary \ref{corollary 4.9} $(i)$ with
$$
p/(p_0 \alpha), r_1/(p_0 \alpha), \ldots, r_d/(p_0 \alpha), q/(p_0 \alpha) > 1
$$
and \eqref{eq12.31}.
An application of the Minkowski inequality gives
\begin{align}
		                    \label{6.4.1}
	 \|D^2_v u \| & \leq
	 N (\nu^{-1} \delta^{-\theta}+ \nu^{2/(p_0\alpha_1)} R_1^{2/(p_0\alpha_1)}   +  \nu^{(4d+2)/p_0}  \delta^{-\theta}  \gamma_0^{ 1/(p_0\alpha_1)}    ) \|D^2_v u \| \notag \\
&
    +  N \nu^{-1} \delta^{-\theta}  \|(-\Delta_x)^{1/3} u \|
+  N   \nu^{(4d+2)/p_0}  \delta^{-\theta} \|Pu\|.
	\end{align}
Furthermore, note that $u$ solves the equation
\begin{equation}
                \label{6.4.6}
	\partial_t u- v \cdot D_x u	- \Delta_v u   = f,
\end{equation}
where $f  = Pu + (a^{i j} - \delta_{i j}) D_{v_i v_j}$.
Then, by Theorem \ref{theorem 6.1} applied to \eqref{6.4.6},
\begin{equation}
            \label{6.4.2}
	\|(-\Delta_x)^{1/3} u\|+\|D_v (-\Delta_x)^{1/6} u\|
	\leq N  \delta^{-\theta} (\|Pu\|	+ \|D^2_v u\|).
\end{equation}
Combining this with \eqref{6.4.1}, we get
\begin{align}
                \label{6.4.3}
     \|D^2_v u \|
     &\le
	 N (\nu^{-1}  + \nu^{(4d+2)/p_0}
	 \gamma_0^{ 1/(p_0\alpha_1) }) \delta^{-\theta} \|D^2_v u \|   \notag  \\
 &
	 \quad  +  N \nu^{2/(p_0\alpha_1)} R_1^{2/(p_0\alpha_1)} \|D_v^2 u\| + N (\nu^{-1} + \nu^{(4d+2)/p_0}) \delta^{-\theta}\| P u \|.
\end{align}	
Furthermore, we set $\nu=2 + 4 N \delta^{-\theta}$.
Then we choose $\gamma_0 > 0$ and $R_1 > 0$ such that \eqref{6.5.8} holds and
$$
    N\nu^{(4d+2)/p_0} \delta^{-\theta}
	 \gamma_0^{ 1/(p_0\alpha_1 )} \le 1/4,
				\quad	N  \nu^{2/(p_0\alpha_1)} R_1^{2/(p_0\alpha_1)}   \le 1/4.
$$
Thus, we can cancel the term containing $D^2_v u$ on the
right-hand side of \eqref{6.4.3}.
By \eqref{6.4.2}, we also obtain the estimate of $(-\Delta_x)^{1/3} u$ and $D_v (-\Delta_x)^{1/6} u$.

\textit{Step 2}.
Let $\zeta \in C^{\infty}_0 ((-(R_0 R_1)^2, 0))$ be a nonnegative function such that
\begin{equation}
				\label{6.4.7}
	\int \zeta^q (t) \, dt = 1, \quad |\zeta'| \le N (R_0 R_1)^{-2-2/q}.
\end{equation}
Observe that for any $t \in \bR$,
\begin{align*}
&	\|D^2_v u (t, \cdot)\|^q_{  L_{p, r_1, \ldots, r_d, q} (\bR^{2d}, \prod_{i=1}^d w_i) }\\
&=  \int_{\bR} \|D^2_v u (t, \cdot)\|^q_{  L_{p, r_1, \ldots, r_d, q} (\bR^{2d}, \prod_{i=1}^d w_i) } \zeta^q (t-s) \, ds.
\end{align*}
Multiplying the above inequality by $w_0$ and integrating over $\bR$ give
$$
		\|D^2_v u \|^q   = \int_{\bR} \|D^2_v \big(u   \zeta (\cdot-s)\big)\|^q  \, ds.
 $$
A similar identity holds for $\|(-\Delta_x)^{1/3} u\|$ and  $\|D_v (-\Delta_x)^{1/6} u\|$.
Furthermore, we fix arbitrary $s \in \bR$ and note that
$u_s (z) := u (z)  \zeta (t-s)$ vanishes outside $(s - (R_0 R_1)^2, s)$  and  satisfies the equation
$$
	P u_s (z) = \zeta (t-s) P u (z) + u \zeta'(t-s).
$$
Then, by the conclusion of Step 1 and \eqref{6.4.7},
\begin{align*}
&	\|D^2_v u_s \|
	+ \|(-\Delta_x)^{1/3} u_s \|
	+\|D_v (-\Delta_x)^{1/6} u_s\|\\
&	\le
	N  \delta^{-\theta} \|(P u)  \zeta (\cdot-s)\| + N  \delta^{-\theta} (R_0 R_1)^{-2-2/q}  \|u \phi (\cdot - s)  \|,
\end{align*}
where $\phi \in C^{\infty}_0 ((-(R_0 R_1)^2, 0))$ is such that $\phi  = 1$ on the support of $\zeta$ and $\int |\phi|^q \, dt = N (R_0 R_1)^2$.
Raising the above inequality to the $q$-th power and integrating with respect to $s$, we get
\begin{equation*}
\begin{aligned}
	&\|D^2_v u \|
	+ \|(-\Delta_x)^{1/3} u\|
	+\|D_v (-\Delta_x)^{1/6} u\|\\
	&\le
	N  \delta^{-\theta} \|P u\| + N  \delta^{-\theta}  (R_0 R_1)^{-2}    \|u\|.	
\end{aligned}
\end{equation*}
Due to \eqref{6.5.8}, we may replace the last term with $N \delta^{-\theta} R_0^{-2} \|u\|$.
 By using  Agmon's method (see the proof of Theorem \ref{theorem 6.1}),
 we conclude that for any $\lambda \ge  1$,
\begin{align*}
&	\lambda \|u\|\	+	\|D^2_v u\|
	+ \|(-\Delta_x)^{1/3} u\|
	+ \|D_v (-\Delta_x)^{1/6} u\|\\
&	\le N  \delta^{-\theta} \|P u + \lambda u\| +  N \delta^{-\theta}  (R_0^{-2} + \lambda^{1/2}) \|u\|.
\end{align*}
Setting $\lambda_0  =   16 (N \delta^{-\theta} R_0^{-2})^2 + 1  $  and canceling the term containing $\|u\|$ on the right-hand side, we prove the desired a priori estimate.

Finally, in the case when $b$ and $c$ are not identically zero,
we use the a priori estimate that we just proved. We get
\begin{align*}
&	\lambda \|u\|\	+	\|D^2_v u\|
	+ \|(-\Delta_x)^{1/3} u\|+ \|D_v (-\Delta_x)^{1/6} u\|\\
&	\le
	 N \delta^{-\theta} (\|P u + b^i D_{v_i} u + c u + \lambda u\| +  \|D_v u\| +  \|u\|).
\end{align*}
By using the interpolation inequality in the weighted mixed-norm Sobolev spaces (see Lemma \ref{lemma A.3})
and further increasing $\lambda_0$, we prove the assertion.
\end{proof}

Next, we show the existence part.
\begin{proposition}
                   \label{theorem 7.2}
Theorem \ref{theorem 2.1}
holds if $p = q = r_i, i = 1, \ldots, d,$ and $w = 1$.
\end{proposition}

\begin{proof}
By the method of continuity,
the assertion follows from Theorems \ref{theorem 2.1} $(i)$ and  \ref{theorem 4.10} $(ii)$.
\end{proof}

\begin{lemma}
                \label{lemma 9.8}
Invoke the assumptions of Theorem \ref{theorem 2.1} $(ii)$
and assume, additionally,
that $f$ vanishes outside $\tQ_R$ for some $R \ge 1$.
Let $\lambda_0  > 1$  be the constant from Theorem \ref{theorem 2.1} $(i)$.
Fix any $\lambda \ge  \lambda_0$ and let $u \in S_p (\bR^{1+2d})$ be the unique solution to Eq. \eqref{2.1.1} (see Proposition \ref{theorem 7.2}).
Then, for any $j \in \{0, 1, 2, \ldots\}$ one has
$$
	\lambda \|u\|_{ L_p (\tQ_{2^{j+1} R} \setminus  \tQ_{2^{j} R})  }
	+	\lambda^{1/2} \|D_v u\|_{ L_p (\tQ_{2^{j+1} R} \setminus  \tQ_{2^{j} R})  }
	+  \|D^2_v u\|_{ L_p (\tQ_{2^{j+1} R} \setminus  \tQ_{2^{j} R})  }
$$
$$
	\le N 2^{-j (j-1)/4} R^{-j} \|f\|_{ L_p (\bR^{1+2d}) },
$$
where $N = N (d, \delta, p)$.
\end{lemma}

\begin{proof}
We follow the argument of Section 8 of \cite{DK_18}, which is somewhat similar to the one of
Lemma \ref{lem4.7}.

First, by Theorem \ref{theorem 2.1} $(i)$,
\begin{equation}
\begin{aligned}
                \label{9.8.1}
	&\lambda \|u\|_{  L_p (\bR^{1+2d})
	}
    + \lambda^{1/2} \|D_v u\|_{  L_p (\bR^{1+2d})}
    + \|D^2_v u\|_{  L_p (\bR^{1+2d})}\\
   & \le N \|f\|_{ L_p (\bR^{1+2d}) }.
\end{aligned}
\end{equation}
Furthermore, let $\{\eta_j ,j = 0,1,2,\ldots\}$ be a sequence of smooth functions such that $\eta_j = 0$ in $\tQ_{2^{j} R}$, $\eta_j  = 1$ outside $\tQ_{2^{j+1} R}$,
\begin{equation}
\begin{aligned}
                \label{9.8.2}
&|\eta_j|\le 1,\quad |D_v \eta_j|\le N2^{-j}R^{-1}, \quad |D^2_v \eta_j|\le N2^{-2j}R^{-2},\\
	& |D_x \eta_j|\le N2^{-3j}R^{-3}, \quad |\partial_t \eta_j| \le N2^{-2j}R^{-2}.
\end{aligned}	
\end{equation}
Then, since $f = 0$ outside $\tQ_R$,  the function $u_j = u \eta_j$ satisfies the equation
$$
P u_{j} +  b^i D_{v_i} u_j + c u_j +   \lambda u_j
=  u (P \eta_j +b^i D_{v_i} \eta_j)
- 2 (a D_v \eta_j)\cdot D_v u.
$$
By Theorem \ref{theorem 2.1} $(i)$, interpolation inequality, and
\eqref{9.8.2}, we get
\begin{align*}
&	\|\lambda |u|+\lambda^{1/2}|D_v u|+|D^2_v u|\|_{  L_{p}  (\tQ_{2^{j+2} R} \setminus  \tQ_{2^{j+1} R}) }\\
&	\le     \lambda \|u_j\|_{  L_{p}  (\bR^{1+2d})   }
    + \lambda^{1/2} \|D_v u_j\|_{  L_{p}  (\bR^{1+2d})   }
    + \|D^2_v u_j\|_{  L_{p}  (\bR^{1+2d})   }\\
& 	\le  N 2^{-j} R^{-1} \||u| + |D_v u|\|_{  L_{p} (\tQ_{2^{j+1} R} \setminus  \tQ_{2^{j} R}) }\\
&   \le  N 2^{-j} R^{-1} (\lambda \|u\|_{  L_{p} (\tQ_{2^{j+1} R} \setminus  \tQ_{2^{j} R}) }
   +  \lambda^{1/2} \|D_v u\|_{  L_{p} (\tQ_{2^{j+1} R} \setminus  \tQ_{2^{j} R}) }).
\end{align*}
Iterating this estimate and using \eqref{9.8.1}, we prove the assertion.
\end{proof}

\begin{proof}[Proof of Theorem \ref{theorem 2.1} (ii)]
 We will only consider the case when the weight is a function of $t, v$  because the result in
 $$
 S_{p; r_1, \ldots, r_d} (\bR^{1+2d}_T, |x|^{\alpha} \prod_{i = 1}^d w_i (v_i))
 $$
 is established in the same way.
The uniqueness follows from
 Theorem \ref{theorem 2.1} $(i)$.

To prove the existence part, we first consider the case when $T = \infty$.
By the reverse H\"older inequality
for $A_p$-weights    (see, for example, Theorem 7.2.5 of \cite{G_14})
and the scaling property of $A_p$-weights (see Proposition 7.1.5 $(9)$ of \cite{G_14}),
there exist a large  constant
$p_1  > 1$ and small constants
$\varepsilon_i > 0, i = 0, \ldots, d$
depending only on
$d, \delta, p, r_1, \ldots, r_d, q$, and $K$
such that for any $R > 1$,
$$
     \frac{p_1}{q} = \frac{1+\varepsilon_0}{\varepsilon_0},
     \quad
      \frac{p_1}{r_i} = \frac{1+\varepsilon_i}{\varepsilon_i},\quad i = 1, \ldots, d,
$$
$$
    \big(\fint_{-R^2}^0 w_0^{ 1 + \varepsilon_0 } \, dt\big)^{\frac{1}{1+\varepsilon_0}}  \le N (q, K) \fint_{-R^2}^0 w_0 \, dt   \le N (q, K)  R^{2q - 2 }    \int_{-1}^1  w_0 \, dt  ,
$$
\begin{align*}
    \big(\fint_{-R}^{R} w_i^{ 1 + \varepsilon_i } \, dv_i\big)^{\frac{1}{1+\varepsilon_i}}  &\le N (r_i, K)  \fint_{-R}^R w_i \, dv_i  \\
    &\le N (r_i, K) R^{r_{i}  -  1}   \int_{-1}^1  w_i \, dv_i      ,\quad i = 1, \ldots, d.
\end{align*}
 Then, applying H\"older's inequality
 repeatedly, we prove that
 for any $R > 1$ and $h \in L_{p_1} (\tQ_R)$,
\begin{equation}
                \label{9.6.1}
    \|h\|_{ L_{p, r_1, \ldots, r_d, q} (\tQ_R, w) }
    \le
    N R^{\kappa} \|h\|_{ L_{p_1} (\tQ_R) },
\end{equation}
where $\kappa = \kappa (d, p, r_1, \ldots, r_d, q, K)  > 0$ and
$$
N  = N (d, \delta, p, r_1, \ldots, r_d, q,w_0, w_1, \ldots, w_d)  > 0.
$$
Note that \eqref{9.6.1} also holds if we replace $\tQ_R$ with   $\tQ_{2R} \setminus \tQ_R$.

Next, let $f_n \in C^{\infty}_0 (\bR^{1+2d})$ be a sequence of functions such that
$f_n \to f$
in $L_{p, r_1, \ldots, r_d, q} (\bR^{1+2d}, w)$ as $n \to \infty$.
Now we fix some $n \in \{1, 2, \ldots\}$.
We may assume that $f_n$ vanishes outside $\tQ_R$ for some $R > 1$ depending on $n$.
Let us consider the equation
\begin{equation}
                \label{9.6.3}
    P u_n + b^i D_{v_i} u_n + c u_n + \lambda u_n = f_n.
\end{equation}
By Proposition \ref{theorem 7.2}, this equation has a unique solution $u_n \in L_{p_1} (\bR^{1+2d})$.
Then, by \eqref{9.6.1},
$u_n \in S_{p, r_1, \ldots, r_d, q, \text{loc}} (\bR^{1+2d},w)$.
Furthermore, denote
$$
    X_j = L_{p, r_1, \ldots, r_d, q} (\tQ_{2^{j+1} R} \setminus \tQ_{2^{j} R}, w),\quad j \geq  0.
$$
Then, by  Lemma \ref{lemma 9.8}
and \eqref{9.6.1} with $R$ replaced with $2^j R$,
we have
\begin{align*}
&    \lambda \|u_n\|_{ X_{j}  }
	 + \lambda^{1/2} \|D_v u_n\|_{ X_{j}  }
   + \|D^2_v u_n\|_{ X_{j}  }\\
&    \le
    N 2^{j \kappa} R^{\kappa} \big(\lambda \|u_n\|_{ L_{p_1} (\tQ_{2^{j+1} R} \setminus \tQ_{2^{j} R}) }
+ \lambda^{1/2} \|D_v u_n\|_{ L_{p_1} (\tQ_{2^{j+1} R} \setminus \tQ_{2^{j} R}) }\\
&\quad   + \|D^2_v u_n\|_{ L_{p_1} (\tQ_{2^{j+1} R} \setminus \tQ_{2^{j} R}) }\big)\\
&    \le  N  2^{- j (j-1)/4 + j \kappa} R^{\kappa - j} \|f_n\|_{  L_{p_1} (\bR^{1+2d}) }.
\end{align*}
Summing up the above inequality with respect to $j$,
 we conclude
 $$
 u_n \in S_{p, r_1, \ldots, r_d, q} (\bR^{1+2d}, w).
 $$
 Then, by the a priori estimate of Theorem \ref{theorem 2.1} $(i)$,
 $\{u_n, n \ge 1\}$
 is a Cauchy sequence in
 $S_{p, r_1, \ldots, r_d, q}  (\bR^{1+2d}, w)$.
 Therefore, this sequence has  a limit
 $$
 u \in S_{p, r_1, \ldots, r_d, q} (\bR^{1+2d}, w).
 $$
Finally, passing to the limit in \eqref{9.6.3}, we prove that $u$ is the solution to Eq. \eqref{2.1.1}.

 In the case when $T < \infty$,
 we consider the equation
 $$
   P u_1 + b^i D_{v_i} u_1 +  c u_1 + \lambda u_1 = f 1_{t < T}.
 $$
 By the above, this equation has
 a unique solution
 $u_1 \in S_{p, r_1, \ldots, r_d, q} (\bR^{1+2d}, w)$,
 which is the solution  to Eq. \eqref{2.1.1}.
The theorem is proved.
\end{proof}

\appendix
\section{}

\begin{lemma}
			\label{lemma 4.7}
Let $\sigma > 0$,  $R > 0$, $p \ge 1$ be  numbers, and
$f \in L_{p, \text{loc} } (\bR^d)$. Denote
$$
	g ( x)
	=
	\int_{|y| > R^3 }
	f ( x+y)	|y|^{- (d + \sigma)}   \, dy.
$$
Then,
$$
   (|g|^p)^{1/p}_{  B_{R^3 }  }\leq N (d , \sigma) R^{- 3 \sigma}
	 \sum_{k = 0}^{\infty}
	 2^{- 3 k\sigma}
	(|f|^p)^{1/p}_{ B_{ (2^kR)^3 } }.
$$
\end{lemma}

\begin{proof}
By H\"older's inequality for any $x \in B_{R^3 }$, we have
\begin{align*}
&	|g ( x )| \leq	\sum_{k=0}^\infty   \int_{2^{  3  k} R^3 <|y|<2^{  3  (k+1)} R^{3} }
	|f|(x+y)  |y|^{- (d + \sigma)} \, dy\\
&	\le  N (d) \sum_{k=0}^\infty 2^{-  3  k\sigma} R^{- 3 \sigma}
	\bigg(\fint_{2^{ 3 k} R^3 <|y|<2^{3 (k+1)} R^3 }
	|f|^p( x+y) \, dy\bigg)^{1/p}.
\end{align*}

Taking the $L_p$-average of the both sides of the last inequality over $B_{R^3}$
and using the Minkowski inequality, we prove the assertion of this lemma.
\end{proof}

\begin{lemma}
            \label{lemma 9.7}
Let $p > 1$  be a number,
 $w \in A_p (\bR^d)$,
and $f \in L_p (\bR^d, w)$.
Then, there exists a number $p_0 > 1$ depending only on $d$, $p$, and $[w]_{A_p}$ such that
$f \in L_{p_0, \text{loc}} (\bR^d).$
\end{lemma}

\begin{proof}
By Corollary 7.2.6 of \cite{G_14}, there exists $q \in (1, p)$ depending
only on $p$, $d$, and $[w]_{A_p (\bR^d)}$
such that
$
    w \in A_{q} (\bR^d).
$
Let $p_0 = p/q$.
Then, by this and H\"older's inequality
for any cube $C$,
$$
    \int_C |f|^{p_0} \, dx
    \le
    \Big(\int_C |f|^p w \, dx\Big)^{1/q}
     \Big(\int_C w^{-1/(q-1)} \, dx\Big)^{(q-1)/q} < \infty.
$$
The lemma is proved.
\end{proof}

For numbers $p_1, \ldots, p_d \in (1, \infty)$,
by $L_{p_1, \ldots, p_d} (w_1, \ldots, w_d)$ we denote the space of measurable functions with the finite norm
\begin{align*}
	\|f\|_{ L_{p_1,   \ldots, p_d } (w_1, \ldots, w_d) }
=
    \big|\int_{\bR}
    \big| \ldots \big|\int_{\bR} \big|\int_{\bR} |f|^{p_1} (x) \, w_1 (x_1) dx_1\big|^{\frac {p_2} {p_1} } \ldots w_d (x_d) dx_d\big|^{\frac {1} {p_d}}.
\end{align*}
Furthermore, by $W^2_{p_1, \ldots, p_d} (w_1, \ldots, w_d)$ we mean  the Sobolev space
of all functions $u \in L_{p_1, \ldots, p_d} (w_1, \ldots, w_d)$  such that
$D_x u, D^2_x u \in L_{p_1, \ldots, p_d} (w_1, \ldots, w_d)$.

\begin{lemma}[Interpolation inequality]
				\label{lemma A.3}
Let  $p_1, \ldots, p_d \in (1, \infty)$ be arbitrary numbers and
 $w_i \in A_{p_i} (\bR), i = 1, \ldots, d$,
such that $[w_i]_{ A_{p_{i}} (\bR) } \le K, i = 1, \ldots, d,$  for some $K \ge 1$.
Then, for any $u \in W^2_{p_1, \ldots, p_d} (w_1, \ldots, w_d)$
and $\varepsilon > 0$,
we have
$$
	\|D_x u\| \le  \varepsilon \|D^2_x u\| +  N \varepsilon^{-1} \|u\|,
$$
where
$
	\|\cdot\| = \|\cdot\|_{ L_{p_1,   \ldots, p_d } (w_1, \ldots, w_d) }
$
and
$
	N  = N (d,  p_1, \ldots, p_d, K).
$
\end{lemma}

\begin{proof}
First, by Lemma 3.8 $(iii)$ of \cite{DK19}, for any $w \in A_{p_1} (\bR^d)$ and $\varepsilon>0$, one has
$$
	\int |D_x u|^{p_1} w (x) \, dx \le N \int |g|^{p_1} w (x) \, dx,
$$
where
$$
	g (x) = \varepsilon |D^2_x u| +  \varepsilon^{-1} |u|
$$
and $N = N (d, p_1, [w]_{A_{p_1} (\bR^d) })$.
Applying a variant of the Rubio de Francia extrapolation theorem (see, for example, Theorem 7.11 of \cite{DK19} and also \cite{Kr_20}), we prove the lemma.
\end{proof}

\textbf{Acknowledgements}
The authors would like to thank the referees for finding a typo in the original version of the manuscript, pointing out  the references \cite{AMP_20} and \cite{DFP_05},
and providing suggestions that have led to the improvement of this paper's presentation.
We also express our sincere gratitude to the referee who  suggested to study the a priori estimate in the $S_p$ space with the weight depending on the $x$ variable.

%
%



\end{document}